\documentclass[11pt,a4paper]{amsart}
\usepackage{tikz, float}
\usepackage{comment}
\usepackage{caption}
\usepackage{subcaption}
\usepackage{amsmath,amsthm,amsfonts,graphicx,amssymb,amscd,dsfont,euscript,enumerate,verbatim,calc,mathtools,listings}
\usepackage{enumitem}
\usepackage{mathrsfs}
\usepackage[hidelinks]{hyperref}
\usepackage[nameinlink, noabbrev]{cleveref}
\newtheorem{thm}{Theorem}[section]

\newtheorem{cor}[thm]{Corollary}
\newtheorem{lem}[thm]{Lemma}

\newtheorem{prop}[thm]{Proposition}

\theoremstyle{definition}
\newtheorem{defn}[thm]{Definition}

\newcommand{\pslnr}{\mathrm{PSL}_n(\mathbb{R})}
\newcommand{\pslnc}{\mathrm{PSL}_n(\mathbb{C})}
\newcommand{\slnc}{\mathrm{SL}_n(\mathbb{C})}

\newcommand{\pslr}{\mathrm{PSL}_2(\mathbb{R})}
\newcommand{\pslc}{\mathrm{PSL}_2(\mathbb{C})}
\newcommand{\slc}{\mathrm{SL}_2(\mathbb{C})}
\newcommand{\rnr}{\mathcal{R}_{\pslnr}(S)}
\newcommand{\rnc}{\mathcal{R}_{\pslnc}(S)}
\newcommand{\Y}{\mathcal{Y}}
\newcommand{\cp}{\mathrm{CP}^1}
\newcommand{\cA}{\mathcal A_n}
\newcommand{\cB}{\mathcal B_n}

\newcommand{\R}{\mathbb R}
\newcommand{\C}{\mathbb C}
\newcommand{\Z}{\mathbb Z}

\newcommand{\wt}{\operatorname{wt}}

\newcommand{\Plaques}{\mathcal P}

\newcommand{\abs}[1]{\left|#1\right|}

\newcommand{\minor}[3]{\Delta_{#1,#2}\!\left(#3\right)}

\newcommand{\HH}{\mathbb{H}}

\newtheoremstyle{named}%
{}{}{\itshape}{}{\bfseries}{.}{.5em}{\thmnote{#3}}
\theoremstyle{named}

\allowdisplaybreaks[1]
\numberwithin{equation}{section}
\begin{document}
\title{Entropy and domination for quasi-Hitchin representations}
\author{Pabitra Barman}
\author{Subhojoy Gupta}
\address{Ashoka University, Rajiv Gandhi Education City, Sonipat, 131029, Haryana, India}
\email{pabitrabarman560@gmail.com}
\address{Department of Mathematics, Indian Institute of Science, Bangalore 560012, India}
\email{subhojoy@iisc.ac.in}

\begin{abstract}
Let $S$ be a closed oriented surface of genus $g\geq 2$. We consider an  $n$-pleated representation  $\rho: \pi_1(S) \to \pslnc$ obtained by bending a Hitchin representation $\rho_0:\pi_1(S) \to \pslnr$  along  a maximal geodesic lamination. The space of such $n$-pleated representations was recently introduced by Maloni-Martone-Mazzoli-Zhang who provided a parametrization via shear-bend cocycles. Our main result is that $\rho_0$  dominates $\rho$ in the Hilbert length spectrum and the translation-length spectrum, with a strict domination for \textit{most} curves, that we call \textit{statistical} domination. Using this, we prove some entropy rigidity results: namely, the Hilbert entropy of any quasi-Hitchin representation in the bending fiber is strictly greater than that of $\rho_0$, the same for the translation length entropy when $\rho_0$ is $n$-Fuchsian,  and in the latter case a new proof that for hyperconvex representations the Hausdorff dimension of the full limit set increases. The proof involves analyzing the weighted planar networks for finite approximants of the monodromy matrix, and establishing a strict matrix domination for generic monodromy using the equidistribution of closed geodesics in the unit tangent bundle of $S$.
\end{abstract}

\maketitle

\tableofcontents

\section{Introduction}
Let $S$ be a closed oriented surface of genus $g\geq 2$, and let $\mathcal{R}_{\pslnr}(S)$ denote the representation-variety of $\pslnr$-representations of the surface-group $\pi_1(S)$, up to conjugation. For $n=2$, there is a component comprising discrete, faithful  representations (i.e. Fuchsian representations), which can be identified with the classical Teichm\"{u}ller space $\mathcal{T}(S)$ since they arise as the holonomy of hyperbolic structures on $S$.  For $n>2$, Hitchin  (\cite{Hit}) showed the existence of a component that we denote by $\text{Hit}_n(S)$  in $\rnr$ diffeomorphic to a ball, that Labourie (\cite{Lab}) proved comprises discrete-faithful representations that are \textit{Anosov} in the sense that the dynamics of a flow on the associated  $\mathbb{R}^n$-bundle on $T^1S$ is Anosov. In \cite{FG}, Fock-Goncharov introduced the notion of \textit{positivity}, and provided an alternative proof of Labourie's result. The study of this Hitchin component and exploring the analogy with Teichm\"{u}ller space is the subject of the active field of \textit{higher} Teichm\"{u}ller theory (see, for example, \cite{WienhardICM}). 

In this article, we consider \textit{complex} deformations of Hitchin representations in the variety $\rnc$ of representations of the surface-group into $\pslnc$, obtained by ``bending" a Hitchin representation along a maximal geodesic lamination $\lambda$, and we first prove a domination result for the length spectra (of closed curves) -- see Theorem \ref{thm:main}. Our preceding paper \cite{BGup} had handled the case when the surface $S$ has punctures, using Fock-Goncharov coordinates on the corresponding representation-variety $\rnc$. When $S$ is a closed surface (as in this paper), analogous coordinates  for $\text{Hit}_n(S)$ were introduced by Bonahon-Dreyer (\cite{BD17}), in which Hitchin representations form a convex cone in an appropriate space of \textit{shear-cocycles} (see \S2 for a brief account). Here, we shall use the recent extension of this by Maloni-Martone-Mazzoli-Zhang (\cite{MMMZ}) that parametrizes a subset  $\mathcal{R}(\lambda) \subset \rnc$ of ``$n$-pleated" representations,  by  ``shear-bend" cocycles (valued in $\mathbb{C}/2\pi i\mathbb{Z}$) transverse to $\lambda$. In what follows, we say that $\rho\in \mathcal{R}(\lambda) $ is obtained by \textit{bending} $\rho_0$ along $\lambda$ if the corresponding difference of shear-bend cocycles is purely imaginary; the set of representations is called the \textit{bending fiber} (see \S2 for details).

\medskip 

\noindent Before we state our results, recall:

\begin{defn}\label{defn:lengths} 
Given a representation $\rho:\pi_1(S)\to \pslnc$ and $\gamma\in\pi_1(S)$, let $A\in \slnc$ be a determinant-one  lift of $\rho(\gamma)$ and let $\lambda_1,\dots,\lambda_n$ be the eigenvalues of $A$ ordered so that $|\lambda_1|\le\cdots\le |\lambda_n|$.
\begin{enumerate}
\item The \emph{Hilbert length} of $\gamma$ with respect to $\rho$ is
\[
\ell^H_{\rho}(\gamma):=\log\Bigg|\frac{\lambda_n}{\lambda_1}\Bigg|.
\]
\item The \emph{translation length} of $\gamma$ with respect to $\rho$ is the translation length of $\rho(\gamma)$ in the symmetric space $\mathbb{X}_n=\pslnc/\mathrm{PSU}(n)$, namely
\[
\ell_{\rho}^X(\gamma):=\sqrt{2\big((\log|\lambda_1|)^2+\cdots+(\log|\lambda_n|)^2\big)}.
\]

\item The $k$-th \emph{fundamental weight length}  of $\gamma$ with respect to $\rho$ is 
\[
\ell_{\rho}^{\omega_k}(\gamma) := \sum\limits_{i=0}^{k-1}\log{\lvert \lambda_{n-i} \rvert}
\]
for each $1\leq k<n$.
\end{enumerate}
(These expressions are well-defined for the conjugacy class of $\gamma$, and independent of the choice of the determinant-one lift $A$, as two such lifts differ by multiplication by an $n$-th root of unity.)
\end{defn}

\noindent \textit{Remark.} In (2) we have chosen a normalization of the Riemannian metric in $\mathbb{X}_n$ such that for $n=2$ the space $\mathbb{X}_2 \cong \HH^3$ and the translation length agrees with the translation distance of $\rho(\gamma)$ in the hyperbolic metric. 

\medskip 

For any of these notions of length $\ell_\rho$, we say $\rho_0$ dominates $\rho$ in the corresponding length spectrum if $\ell_\rho(\gamma) \leq \ell_{\rho_0}(\gamma)$ for all $\gamma \in \pi_1(S)$. Our first result is:

\begin{thm}[Domination]\label{thm:main}  Let $\rho_0 \in \text{Hit}_n(S)$ and let $\rho\in \mathcal{R}(\lambda)$ be an $n$-pleated representation obtained by bending $\rho_0$ along a maximal geodesic lamination $\lambda$. Then we have the inequality of fundamental weight lengths
\[
\ell_{\rho}^{\omega_k}(\gamma) \leq \ell_{\rho_0}^{\omega_k}(\gamma)  \qquad \text{ for any }\  \gamma \in \pi_1(S) 
\]
for each $1\leq k<n$. Moreover, the Hitchin representation $\rho_0$ dominates $\rho$ in the Hilbert length spectrum as well as the translation length spectrum.
\end{thm}

The strategy of the proof generalizes that in the preceding article \cite{BGup}; in particular, the dominating Hitchin representation $\rho_0$ is the one obtained by setting the imaginary part of the shear-bend cocycle to zero. In case of a punctured-surface, we had taken the  modulus of each Fock-Goncharov parameter; moreover, the ``maximal lamination" we had considered there was an ideal triangulation, with finitely many leaves. Here, in the case of a maximal lamination on a closed surface, we are forced to consider transverse arcs crossing infinitely many (possibly uncountable)  leaves; we rely on the methods of \cite{MMMZ} that generalize work of Bonahon (for $\pslc$-representations) and handle the corresponding convergence issues.  In particular, we develop a method to augment the finite product of building blocks used in \cite{BGup} by finite approximants to the ``slithering map" of \cite{MMMZ}. We then show that these approximants are weight matrices of weighted planar networks, and are compatible with taking the real part of the cocycle; the domination argument in \cite{BGup} then applies and passes to the limit. 
\medskip

A key consequence of the domination result will be an entropy rigidity result for \textit{quasi-Hitchin} representations, namely those obtained by bending deformations of Hitchin representations that are still Borel-Anosov (see Definition \ref{defn:quasiH}).

\medskip

\noindent Recall:

\begin{defn}[Entropy]\label{defn:ent} Given a discrete representation $\rho:\pi_1(S) \to \pslnc$, we define its Hilbert entropy to be:
\begin{equation}
H(\rho) =  \limsup\limits_{L\to \infty} \displaystyle\frac{\log \#\{[\gamma] \text{ a conjugacy class in } \pi_1(S)\ \vert\ l^H_\rho(\gamma) \leq L\}}{L}
\end{equation}
where $l^H_\rho(\gamma)$ is the Hilbert length of $\gamma$ with respect to $\rho$.  The \textit{translation length entropy} $E(\rho)$ is defined by the same expression, replacing $l^H_\rho(\gamma)$  by   the \textit{translation length} $\ell_\rho^X(\gamma)$ of $\rho(\gamma)$. 
\end{defn}

\noindent Within the Hitchin component $\text{Hit}_n(S)$, the entropy-rigidity result of Potrie-Sambarino \cite{PotrieSambarinoEigenvaluesEntropy} shows that both these entropies are bounded and the maximum is attained precisely on the $n$-Fuchsian locus (see \cite[Remark 1.5]{DaiLi2}). Indeed, work of Zhang in \cite{Zhang} shows that the Hilbert entropy tends to zero along certain sequences in $\text{Hit}_n(S)$.  For $\rho \in \text{Hit}_3(S)$, Pozzetti-Sambarino (\cite{PS26}) gave an interpretation of the Hilbert entropy $H(\rho)$  in terms of the Hausdorff dimension of non-differentiable points of the limit map.

 \medskip 
 
\noindent Our entropy rigidity result in the bending fiber is based on the observation that a strict inequality of entropy follows from a strict domination for a \textit{positive proportion} of curves (see the notion of \textit{statistical domination} in Lemma \ref{lem:stat}). Here,  domination in the length spectrum, as in Theorem \ref{thm:main} is said to be \textit{strict} if there is a constant $0<\alpha <1$ such that $\ell_\rho(\gamma) \leq \alpha \ell_{\rho_0}(\gamma)$ for all $\gamma \in \pi_1(S)$.

\medskip

\noindent In fact, we shall prove that the set of curves for which there is a strict domination has density one:

\begin{thm}[Statistical domination]\label{thm:stat} Let  $\rho_0$ be a Hitchin representation and define $$\mathcal{S}_0(L) = \{[\gamma] \text{ a conjugacy class in } \pi_1(S)\ \vert\ \ell_{\rho_0}^H(\gamma) \leq L \}$$
and let $\rho \neq \rho_0$ be a quasi-Hitchin representation in the bending fiber of $\rho_0$.  Then there is a $0<\lambda<1$ such that the proportion of curves in this set satisfying a strict domination (with factor $\lambda$) tends to $1$, that is, we have 
\begin{equation}\label{eq:dens}
\displaystyle\frac{ { \{[\gamma] \  \vert\  \ell_{\rho_0}^H(\gamma) \leq L  \text{ and }\ell_\rho^H(\gamma) \leq \lambda \ell_{\rho_0}^H(\gamma)\} }\rvert }{\lvert \mathcal{S}_0(L)\rvert} \longrightarrow 1
 \end{equation}
 as $L \to \infty$. Moreover, the same is true for the translation length when $\rho_0$ is an $n$-Fuchsian representation, except possibly with a different constant $\lambda$.
\end{thm}

\begin{cor}\label{thm:ent} The Hilbert entropy strictly increases in a bending fiber of any Hitchin representation $\rho_0$. That is, if $\rho:\pi_1(S) \to \pslnc$ is a quasi-Hitchin representation in the bending fiber of $\rho_0$, then $H(\rho) \geq H(\rho_0)$ and equality holds only if $\rho=\rho_0$. Moreover, the same is true for the translation length entropy in a bending fiber of an $n$-Fuchsian representation $\rho_0$.
\end{cor} 

\noindent \textit{Remarks.} (i) For the fundamental weight lengths, the statistical domination and \eqref{eq:dens}, and hence the increase in the corresponding entropy, also hold. For the translation length, we expect the results to hold in the bending fiber of \textit{any} Hitchin representation; see the final remarks of \S4.5 and the Appendix. 

(ii) For an $n$-Fuchsian representation $\rho_0$, the Hilbert entropy is $H(\rho_0)=\frac{1}{n-1}$ (see \cite{PotrieSambarinoEigenvaluesEntropy}) and the translation length entropy $E(\rho_0)=\sqrt{6/(n^3-n)}$ (see \cite{DaiLi2}).

(iii) It is plausible that the above corollary can also be derived
from \cite[Theorem~6.1]{ChowOh}, together with our domination  result (Theorem \ref{thm:main}) and Corollary \ref{cor:closed-ineq}, by applying it to the joint representation $(\rho_0,\rho)$. A direct application  of their work requires the joint image,
after passing to a finite-index subgroup if necessary, to be
Zariski-dense in a connected semisimple real algebraic subgroup
of $\pslnc\times\pslnc$; our proof does not require this additional
hypothesis. The statistical domination \eqref{eq:dens} is also stronger than what can be derived from their work.

\medskip

 The key technical part of the proof of Theorem \ref{thm:stat}  is to show that strict domination holds for curves following a certain itinerary on the surface; this uses a finer analysis of the associated weighted planar networks, including their ``transparency" (see Definition \ref{def:transparent}) that allows the strict domination to pass to the limit. In the case $\rho_0$ is an $n$-Fuchsian representation, the statistical domination then follows from the well-known equidistribution of closed geodesics with respect to the usual Liouville measure on the unit tangent bundle of a hyperbolic surface (see \cite{Bowen-eq}). In the case that $\rho_0$ is a Hitchin representation, we use an analogous equidistribution result (from \cite{Sam14}) with respect to the corresponding Bowen-Margulis measure, which also has full support in the unit tangent bundle  (see Appendix A). 

\medskip
For quasi-Fuchsian representations into $\pslc$, the Hilbert entropy and translation length entropy coincide, and also equals the Hausdorff dimension of the corresponding limit set (see \cite{Sullivan84}, \cite{BishopJones97}). In this case, the statement of Corollary \ref{thm:ent} is Bowen's result in (\cite{Bowen79}) that used thermodynamic formalism to prove that the Hausdorff dimension is minimized (with value $1$) exactly at the Fuchsian locus. We provide an alternative proof in \S4.2 using  the notion of statistical domination, which in this case is proved by a geometric argument using pleated-planes in $\HH^3$. In higher rank, there is no comparable notion of pleated planes in the corresponding symmetric space, so we develop and use algebraic and combinatorial techniques alluded to earlier.   

\medskip
For the case $n>2$, a recent result of Farre-Pozzetti-Viaggi (\cite{FarrePozzettiViaggi24}) establishes the analogue of Bowen's result for \textit{hyperconvex} representations, which is a (possibly strict) subset of the quasi-Hitchin representations considered here. For any such representation $\rho$, the Hausdorff dimension of the image $\Lambda_\rho$ of the full limit map $\xi:\partial_\infty \widetilde{S} \to \mathcal{F}(\mathbb{C}^n)$ is given by 
\begin{equation}\label{eq:hdim}
\text{Hdim}(\Lambda_\rho) = \max\limits_i h^i(\rho)
\end{equation}
where for each $1\leq i<n$,  $h^i(\rho)$ is the entropy of the length function $\ell^i_\rho(\gamma) = \log \left\vert\frac{\lambda_{i+1}}{\lambda_i} \right\vert$ (see Definition \ref{defn:lengths}). See \cite[Corollary 5.13]{PS26} for the formula. As a consequence of the Hilbert entropy rigidity in Corollary \ref{thm:ent}, we obtain a new proof of this special case of the main result of \cite{FarrePozzettiViaggi24}:

\begin{cor}\label{cor:dim} For any hyperconvex representation $\rho$ in the bending fiber of an $n$-Fuchsian representation $\rho_0$, the Hausdorff dimension  $\text{Hdim}(\Lambda_\rho) \geq 1$ and equality holds only if $\rho=\rho_0$.  
\end{cor}

It would be interesting to extend our techniques to show that statistical domination holds for the max length function $\max\limits_i \ell^i_\rho$ in \textit{any} bending fiber; this is supported by some numerical experiments and would provide an alternative linear algebraic approach to their result.
However, we do not expect statistical domination for each individual length $\ell^i_\rho$; indeed, it is not a positive linear combination of the fundamental weights.

\medskip

Finally, it would be interesting to relate the results of this paper to the alternative perspective on these representation-varieties arising from the identification with moduli spaces of Higgs bundles via the ``non-abelian Hodge correspondence". In particular, the  ``bending fiber" in Corollary \ref{thm:ent} can be thought of as an analogue of a regular fiber of the Hitchin map, and the corresponding domination and entropy rigidity results there are still conjectural, known only for the fibers of $n$-Fuchsian representations (see \cite{DaiLi2}).

\medskip

\medskip

\noindent\textbf{Acknowledgements.}  This project originated from the PhD thesis work of PB, written under SG's supervision. The authors would like to thank  Beatrice Pozzetti and Krishnendu Gongopadhyay for comments and suggestions on that thesis, and Hee Oh for her comments on an earlier version of this preprint pointing us to the paper of Chow-Oh. SG is grateful to Tengren Zhang for helpful conversations, Giuseppe Martone for insightful suggestions, Andr\'{e}s Sambarino for answering some queries,  Apoorva Khare for providing some references, Qiongling Li for her suggestions on an earlier version, Bram Petri for an insightful question that led to the ``density-one" formulation in Theorem \ref{thm:stat}, and Fran\c{c}ois Labourie for his encouraging comments.  PB sincerely thanks Pranab Sardar for valuable support during his post doctoral stay at IISER Mohali. The ideas and drafts of this work preceded the advent of AI tools; however ChatGPT 5.6 Sol and 6 Astra assisted in the final preparation of the manuscript, for editorial tasks and suggesting ways to polish the arguments and streamline references to prior work, particularly in \S3.2, \S3.3, Lemmas 4.5-7, Proposition 4.19 and the Appendix. The authors independently checked all such suggestions and take full responsibility for the content of this article. This work was supported by the Department of Science and Technology, Govt.of India grant no. CRG/2022/001822, and by the DST FIST program - 2021 [TPN - 700661].

\section{Preliminaries}

\subsection{Shear-bend coordinates for $\pslc$-representations}

It is well-known that the Teichm\"{u}ller space $\mathcal{T}(S)$, comprising Fuchsian representations in $\pslr$, is diffeomorphic to a ball of dimension $6g-6$, and one of the parametrizations is via \textit{shear coordinates} with respect to a \textit{maximal lamination} on the  surface. 

Recall that a geodesic lamination on $S$, equipped with a choice of a hyperbolic metric,  is a closed set that is a union of disjoint simple complete geodesics (either closed or bi-infinite), and it is maximal if its complement comprises ideal triangles (necessarily $4g-4$ in number).  In the case that the maximal lamination $\lambda$  has finitely many leaves, the shear coordinates for a hyperbolic surface (an element of $\mathcal{T}(S)$) are simple to describe: realizing leaves of $\lambda$ as geodesics, the two ideal hyperbolic triangles adjacent to a bi-infinite leaf differ by a real-valued \textit{shear}, measured by the distance between the points on the leaf where the altitudes from the two centroids intersect.  (See \cite{BBFS} for an account.)

More generally, a maximal lamination $\lambda$ could have uncountably many leaves, such that a transverse arc intersects $\lambda$ in a Cantor set;  in this case one can record the mutual shears as a real-valued transverse cocycle (see \cite{Bonahon-Topology}), and $\mathcal{T}(S)$ is diffeomorphic to a finite-sided cone $\mathcal{C}(\lambda)$ in the vector space of such cocycles $\mathcal{H}(\lambda; \mathbb{R})$ (see Theorem A in \cite{BonShear}).

Bonahon extended this to parametrizing subsets of $\pslc$-representations obtained by \textit{bending} a Fuchsian representation along leaves of $\lambda$; in the universal cover $\HH^2$ of a hyperbolic surface, this bending can be described in geometric terms, by equivariantly bending a totally-geodesic copy of $\HH^2$ in $\HH^3$ along the leaves of the lift $\tilde{\lambda}$. The resulting surface is a piecewise totally-geodesic  \textit{pleated plane} in $\HH^3$ (see Definition \ref{defn:pleated}), equivariant under a new surface-group representation into $\pslc$; such surfaces were introduced by Thurston in the context of \textit{discrete} $\pslc$-representations (i.e. Kleinian groups)  and are an important tool in the study of ends of hyperbolic $3$-manifolds.

In particular, Bonahon  introduced \textit{shear-bend} coordinates with respect to $\lambda$, for such pleated surfaces (and their corresponding $\pslc$-representations),  measured by \textit{complex}-valued cocycles in the vector space $\mathcal{H}(\lambda; \mathbb{C}/2\pi i \mathbb{Z})$ of $ \mathbb{C}/2\pi i \mathbb{Z}$-valued cocycles. (Note that a $2\pi$-bend does not affect the representation, which explains quotient by $2\pi i \mathbb{Z}$.) He proved:

\begin{thm}[Bonahon, \cite{BonShear}]\label{thm:bon} The subset of $\pslc$-representations pleated along $\lambda$ is biholomorphic to $$\mathcal{C}(\lambda) + i \mathcal{H}(\lambda; \mathbb{R}/2\pi \mathbb{Z}) \subset \mathcal{H}(\lambda; \mathbb{C}/2\pi i \mathbb{Z})$$ where $\mathcal{C}(\lambda)$ is the cone parametrizing Teichm\"{u}ller space $\mathcal{T}(S)$. 
\end{thm}

\subsection{Bonahon-Dreyer parametrization}  
In \cite{BonDrey}, Bonahon-Dreyer gave a parametrization of the Hitchin component $\text{Hit}_n(S)$ that generalized the shear parameters of Teichm\"{u}ller space; before we state their result, we note two earlier developments alluded to in the Introduction:

First, after the work of Labourie (\cite{Lab}), one salient feature of a Hitchin representation to $\pslnr$ has been the associated \textit{limit map} $\xi: \partial_\infty \tilde{S} \to \mathcal{F}(\mathbb{R}^n)$ from the ideal boundary circle of the universal cover to the space of complete flags in $\mathcal{F}(\mathbb{R}^n)$. 

Second, in \cite{FG}, Fock-Goncharov proved that this limit map $\xi$ has the property that for any cyclically ordered triple of points $x,y,z \in \partial_\infty \tilde{S}$, the images $\xi(x), \xi(y)$ and $\xi(z)$ are in general position, and in fact determine a \textit{positive} triple of flags in $\mathcal{F}(\mathbb{R}^n)$, a notion they introduced inspired by Lusztig's notion of total positivity.   Indeed, for any such triple of flags, they defined a collection of $\frac{1}{2}(n-1)(n-2)$ \textit{positive} real numbers that characterized the triple up to the action of $\pslnr$. 

Thus, given a maximal lamination $\lambda$ on $S$, the ideal vertices of the lift of a complementary ideal triangle to the universal cover determine a positive triple of flags via the limit map $\xi$.  Bonahon-Dreyer's parameters then correspond to the $\frac{1}{2}(n-1)(n-2)$  ``triangle-invariants" of such a triple for each complementary ideal triangle, together with a finite-sided cone in a vector space $\mathcal{H}(\lambda; \mathbb{R}^{n-1})$ of $\mathbb{R}^{n-1}$-valued shearing cocycles. We record their result as:

\begin{thm}[Bonahon-Dreyer, Theorem 0.4 of \cite{BD17}]\label{thm:bondrey}
For $\lambda$ a maximal lamination, $\emph{Hit}_n(S)$ is homeomorphic to a finite-sided cone $\mathcal{C}(\lambda,n) \subset \mathbb{R}^{T} \times \mathcal{H}(\lambda; \mathbb{R}^{n-1})$, where $T$ is the total number of triangle invariants.  
\end{thm}

The shearing cocycle for a given Hitchin representation can be determined by the triangle-invariants of triples of flags and ``double-invariants" of quadruple of flags (the latter invariant was introduced by Fock-Goncharov and generalizes the classical notion of cross-ratio). The exact definition is in terms of the \textit{slithering map}, as in the work of \cite{MMMZ} that we shall refer to later.

\subsection{Pleated representations in higher rank} 

We now briefly summarize the extension of Bonahon-Dreyer coordinates to complex representations into $\pslnc$, developed recently by  Maloni-Martone-Mazzoli-Zhang in \cite{MMMZ}, which provides the framework used throughout this paper.

In what follows, let $\lambda$ be a geodesic lamination on $S$, and let $\tilde{\lambda}$ denote its lift to the universal cover $\tilde{S}$; then $\partial \tilde{\lambda} \subset \partial_\infty \tilde{S}$ denotes the set of ideal endpoints of leaves of  $\tilde{\lambda}$. The paper \cite{MMMZ} introduces a higher-rank analogue of a pleated plane  (pleated along $\tilde{\lambda}$) by replacing the geometric object in the symmetric space $\mathbb{H}^3$ with just the associated limit map to the flag variety:

\begin{defn}[$n$-pleated representation]\label{defn:pleat} 
A representation $\rho:\pi_1(S)\to \pslnc$ is said to be \emph{n-pleated along
$\lambda$} if there exists a $\rho$--equivariant map
\[
\xi:\partial\widetilde{\lambda}\to \mathcal{F}(\mathbb{C}^n),
\]
called the \emph{$\lambda$--limit map}, satisfying the following conditions:
\begin{enumerate}
  \item (\emph{$\lambda$--transversality}) For every leaf of $\widetilde{\lambda}$
  with endpoints $x,y\in\partial\widetilde{\lambda}$, the pair of flags
  $\xi(x),\xi(y)$ is transverse, i.e. is in general position. 

  \item (\emph{$\lambda$--hyperconvexity}) For every complementary ideal triangle
  (plaque) of $\widetilde{\lambda}$ with vertices $x,y,z$, the triple
  $\xi(x),\xi(y),\xi(z)$ is in general position.

  \item (\emph{$\lambda$--Anosov property}) The limit map $\xi$ induces associated $\mathbb{C}^n$-bundles over $T^1\lambda \subset T^1S$, and the dynamics of the lift of the geodesic flow there is Anosov.
\end{enumerate}
\end{defn} 

\noindent \textit{Remark.} The last condition is a weaker version of the (Borel-)Anosov property
introduced by Labourie, alluded to earlier, where the domain of the limit map is the entire ideal boundary $\partial_\infty \tilde{S}$.  The space of such $n$-pleated representations forms an open subset $\mathcal{R}(\lambda,n)\subset \rnc$ that includes the Hitchin representations $\text{Hit}_n(S)$. 

\medskip

Given $\rho\in R(\lambda,n)$ with $\lambda$--limit map $\xi$, they introduce  the \textit{complexified} version of the Bonahon-Dreyer shearing cocycle, via the (complex-valued) Fock--Goncharov triangle and edge invariants
defined by the configurations of flags determined by $\xi$. Indeed, they showed that these invariants define
a complex-valued cocycle in a vector space $\mathcal{Y} (\lambda,n;\mathbb{C}/2\pi i\mathbb{Z})$, whose real part recovers the Bonahon-Dreyer coordinates of a Hitchin representation (as in Theorem \ref{thm:bondrey}), and whose imaginary part records the bending deformation.
Their main result then generalizes Bonahon's Theorem \ref{thm:bon} as follows:

\begin{thm}[{\cite[Theorem~4.15]{MMMZ}}]\label{bon-gen} 
For any maximal lamination $\lambda$, the space of (conjugacy classes) of $n$-pleated representations $\mathcal{R}(\lambda,n)$ is biholomorphic to
\[
C(\lambda,n) + i\,\Y(\lambda,n;\mathbb{R}/2\pi \mathbb{Z}),
\;\subset\;
\Y(\lambda,n;\mathbb{C}/2\pi i\mathbb{Z}),
\]
where $C(\lambda,n)$ is the Bonahon--Dreyer cone parametrizing the Hitchin component.
\end{thm}

In particular, fixing the real part of the cocycle determines a \emph{bending fiber},
consisting of representations obtained by complex bending of a fixed Hitchin
representation along $\lambda$.

\medskip

\noindent Finally, we introduce the following terminology, related to that in \cite{AlessandriniDavaloLi24} and \cite{DumasSanders}:  

\begin{defn}[Quasi-Hitchin representations]\label{defn:quasiH}
An $n$-pleated representation $\rho\in \mathcal{R}(\lambda,n)$ in the bending fiber of a Hitchin representation is called \emph{quasi-Hitchin} if it is Borel-Anosov as a $\pslnc$-representation. In particular, such a representation will define a $\rho$-equivariant limit map $\xi:\partial_\infty \widetilde{S} \to \mathcal{F}(\mathbb{C}^n)$ on the \textit{entire} ideal boundary, that satisfies $\xi_i(x) \oplus \xi_{n-i}(y) = \mathbb{C}^n$ for any $x\neq y$ (here $\xi_p(x)$ denotes the dimension-$p$ subspace of the flag $\xi(x)$. 
\end{defn}

\noindent \textit{Remark.} A \textit{hyperconvex} representation satisfies the additional property that the triple of flags $\xi(x), \xi(y), \xi(z)$ is in general position for any triple of distinct points $x,y,z \in \partial_\infty \widetilde{S}$.

\subsection{Snakes and the monodromy formula}\label{sec:snakes}

We recall the monodromy formula expressing the holonomy $\rho(\gamma)$ of a loop $\gamma$ in terms of Fock--Goncharov coordinates (i.e. the triangle and edge invariants) with respect to an ideal triangulation $\mathcal T$ (see \cite{FG} and \cite{BGup}). Note that an ideal triangulation is a special example of a maximal lamination, which has \textit{finitely many} leaves (the edges of $T$). 

\medskip

Assume that the oriented loop $\gamma$ is in general position with respect to $\mathcal T$, so that $\gamma$ crosses edges of $\mathcal T$ transversely and avoids the vertices. Lifting $\gamma$ to the universal cover, we obtain a sequence of ideal triangles $t_1,t_2,\dots,t_k$ and edges $e_1,e_2,\dots,e_k$ crossed successively by the lifted path. To each triangle $t_i$ and each edge $e_i$ one associates matrices $T(t_i),E(e_i)\in \slnc$ defined in terms of the corresponding triangle and edge invariants. These matrices change the  projective basis associated with one side of $\mathcal T$ to another, where the projective basis associated with a side is the one in which the two endpoint flags become the standard flag and its opposite flag. 

The \emph{monodromy formula} then asserts that the holonomy of $\gamma$ is given, up to conjugation, by the ordered product
\begin{equation}\label{decomp}
\rho(\gamma)
\;=\;
T(t_1)^{\delta_1}E(e_1)\,T(t_2)^{\delta_2}E(e_2)\cdots T(t_k)^{\delta_k}E(e_k),
\end{equation}
where each $\delta_i\in\{-1,1\}$ records whether $\gamma$ turns left/right when passing through the corresponding ideal triangle.  We refer to \cite{BGup} for a more detailed exposition (with examples and figures); as in that paper, we normalize a matrix in $\pslnc$ to obtain a lift in $\slnc$ by choosing the determinant-one representative in the projective class.

\medskip 

\noindent \textbf{Snakes and the triangle matrix.} 
Let $P$ be an ideal triangle with an ordered triple of flags $(A,B,C)$ at its vertices, and let $t$ denote its $\frac{1}{2}(n-1)(n-2)$ triangle invariants. 
Consider the $(n-1)$-triangulation of $P$ whose vertices are indexed by triples of nonnegative integers $(a,b,c)$ with $a+b+c=n-1$; this divides $P$ into $\frac{1}{2}(n-1)(n-2)$ smaller triangles.

To each vertex of the smaller triangle one associates the one-dimensional subspace
\[
V_{a,b,c}:=A_{n-a}\cap B_{n-b}\cap C_{n-c}
\]
where $A_j$ denotes the $j$-dimensional subspace of the flag $A$ (and similar notation for the subspaces of $B$ and $C$).

A \emph{snake} is an oriented path along edges of this $(n-1)$--triangulation starting at a vertex of $\triangle$ and terminating at the opposite side after exactly $n-1$ steps.
Choosing a nonzero vector on the line at the head of the snake and propagating it using a relation across each small (unshaded) triangle produces a projective basis of $\mathbb{C}^n$ associated to the snake.

The triangle matrix (up to a scalar factor to make it determinant one) 
\begin{equation}\label{eq:Tmat}
T(t)=M(t)\,S,
\end{equation}
is the change-of-basis matrix between the projective bases associated to the two snakes along consecutive sides of the ideal triangle, together with a ``side-reversal" matrix $S$ that swaps the opposite flags associated with the endpoints of an edge (see \eqref{eq:S}). Here $M(t)$ is a product of matrices corresponding to elementary snake moves (as the path crosses the smaller triangles -- see for example \cite[Fig. 6]{BGup}) whose entries are Laurent monomials in the triangle invariants of $t$ (see \cite[Proposition 9.2]{FG} for details, and \cite[\S2.5]{BGup} for examples). 
In particular, for $1\le i\le n-1$ there are elementary matrices
\begin{equation}\label{eq:fihi}
F_i:=I_n+E_{i+1,i},\qquad
H_i(x):=\mathrm{diag}(\underbrace{1,\dots,1}_{n-i},\underbrace{x,\dots,x}_{i})
\end{equation} 
such that $M(t)$ is an explicit ordered product of these $F_i$ and $H_j(x)$, for various $i,j \in \{1,2,\cdots, n\}$, where  $x$ is a triangle invariant of $t$. We record a fact here that we shall refer to later: if $F^+$ is the standard flag determined by the ordered basis $\{b_1,b_2,\ldots, b_n\}$ and $F^-$ the opposite flag corresponding to the reversed basis  $\{b_n,b_{n-1},\ldots, b_1\}$, the triangle invariants of the triple of flags $(F^+,F^-,M(t)\cdot F^+)$ are precisely $t$. 
\medskip 

\noindent \textbf{The edge matrix.}
The edge matrix $E(e)$ is the change-of-basis matrix between the projective snake bases corresponding to an (oriented)  edge of $\mathcal{T}$, when viewed as sides of the two adjacent ideal triangles.  This matrix depends on the $(n-1)$ edge-invariants $e=(z_1,\dots,z_{n-1})\in (\mathbb{C}^*)^{n-1}$ associated with the edge as follows (up to a constant factor to make the determinant one):
\begin{equation}\label{eq:Ediag}
E(e)=D(e)\,S,\qquad
D(e):=\mathrm{diag}\bigl(1,\ z_{n-1},\ z_{n-2}z_{n-1},\ \dots,\ z_1z_2\cdots z_{n-1}\bigr),
\end{equation}
where $S$ is the fixed  anti-diagonal ``side-reversal" matrix
\begin{equation}\label{eq:S}
S=\begin{pmatrix}
0 & \cdots & \cdots & 0 & 1 \\
\vdots & & \reflectbox{$\ddots$} & -1 & 0\\
\vdots & \reflectbox{$\ddots$} & \reflectbox{$\ddots$} & \reflectbox{$\ddots$} & \vdots \\
0 & \reflectbox{$\ddots$} & \reflectbox{$\ddots$} & & \vdots \\
(-1)^{n+1} & 0 & \cdots & \cdots & 0
\end{pmatrix}.
\end{equation}

\medskip

\medskip

\noindent We shall refer to the matrix products $T(t)^{\pm 1}E(e)$ appearing in \eqref{decomp} as the \textit{building blocks}.

\subsection{Planar networks for building-blocks}

\begin{defn} A \emph{weighted planar network} $(\Gamma, \omega)$  is a finite acyclic directed planar graph $\Gamma$, with $n$ sources on the left boundary and $n$ sinks on the right boundary, and with a weight function $\omega:\mathcal{E} \to \mathbb{C}$ defined on the set of directed edges. The \emph{weight matrix} $\mathcal{W}(\Gamma, \omega)$ of such a network is the $n\times n$ matrix whose $(i,j)$--entry is the sum of the weights of all directed paths from source $i$ to sink $j$, where the weight of a path is the product of the weights of its edges.
\end{defn}

One fundamental result concerning weighted planar networks  is the following result attributed to Lindstr\"{o}m (\cite{Lind}, see \cite[Lemma 7]{CN} or \cite[Lemma 1]{FZ} for the proof). In what follows, a \textit{vertex-disjoint family} of paths is a collection of paths (in a given planar network), no two of which have a common vertex; its weight is defined to be the product of the weights of the paths in the family (for an example see \cite[Example 6]{BGup}). 

\begin{lem}[Lindstr\"{o}m's lemma]\label{lem:lind}
	For $I,J \subset \{1,2,,\ldots, n\}$ where $\lvert I \rvert = \lvert J\rvert$, any minor $\Delta_{I,J}$ of the weight matrix of a planar
	network is equal to the sum of weights of all vertex-disjoint families of paths from the sources indexed by $I$ to the sinks indexed by $J$.  
\end{lem}

In particular, in the case that all weights are real and positive, the weight matrix is \textit{totally nonnegative} (see, for example, \cite{FZ}).

\smallskip

\noindent The key observation in \cite{BGup} is that:

\begin{lem}[\cite{BGup}]\label{keylem}  Each building block in the decomposition \eqref{decomp} is a weight matrix of a weighted planar network. Moreover, the weights are positive Laurent monomials in the Fock-Goncharov coordinates (i.e., products of some of the coordinates and their reciprocals with positive coefficients). 
\end{lem} 

\begin{proof}[Sketch of the proof]
Recall that from Fock-Goncharov's theory of snakes (described in the previous subsection), each building block appearing in \eqref{decomp} is of the form $T(t)^{\pm 1}E(e)$.

The case when the building block is $T(t)E(e)$, is simpler: In the decomposition $T(t) = M(t)S$ (see \eqref{eq:Tmat}), the matrix $M(t)$ is a product of elementary matrices $F_i$ and $H_i(X)$. Each of these elementary matrices is the weight matrix of an explicit planar network (see  \cite[Figure 16]{BGup}) whose weights are exactly the triangle invariants or $1$. Concatenating these planar networks yields a weighted planar network whose weight matrix is $M(t)$.

Moreover, the matrix $S \cdot E(e)$ is diagonal and is realized by a planar network consisting of a disjoint union of horizontal edges weighted by its diagonal entries. Since concatenation of planar networks corresponds to the multiplication of their weight matrices, concatenating the networks for $M(t)$ and $S \cdot E(e)$ yields the planar network for $T(t)E(e)$. By construction, these individual networks only use triangle invariants, edge invariants, and $1$s as their edge weights. Therefore, the resulting concatenated network inherently carries weights that are positive Laurent monomials in the Fock-Goncharov coordinates.

For the case of the building block $T(t)^{-1}E(e)$, the relation $S^2 = \pm I$ (depending on the parity of $n$) is used to write $T(t)^{-1}E(e) = S \cdot M(t)^{-1} \cdot S \cdot S \cdot E(e)$. Constructing the network for $M(t)^{-1}$ directly introduces negative entries. To resolve this, $M(t)^{-1}$ is factored into individual steps corresponding to elementary snake moves. 

By multiplying each step on both sides by the matrix $S$, matrices of the form $S \cdot Step(k)(t) \cdot S$ are obtained, which map to planar networks with strictly positive weights (that are products of the Fock-Goncharov coordinates and their reciprocals, see \cite[Lemma 3.9]{BGup}). Concatenating these step networks with the network for $S \cdot E(e)$ yields the required weighted planar network for $T(t)^{-1}E(e)$ (see Figure \ref{fig:path-to-top}). Since the weights of the individual step networks are already of the form asserted, the final concatenated network naturally preserves the property that all weights are positive Laurent monomials.
\end{proof}

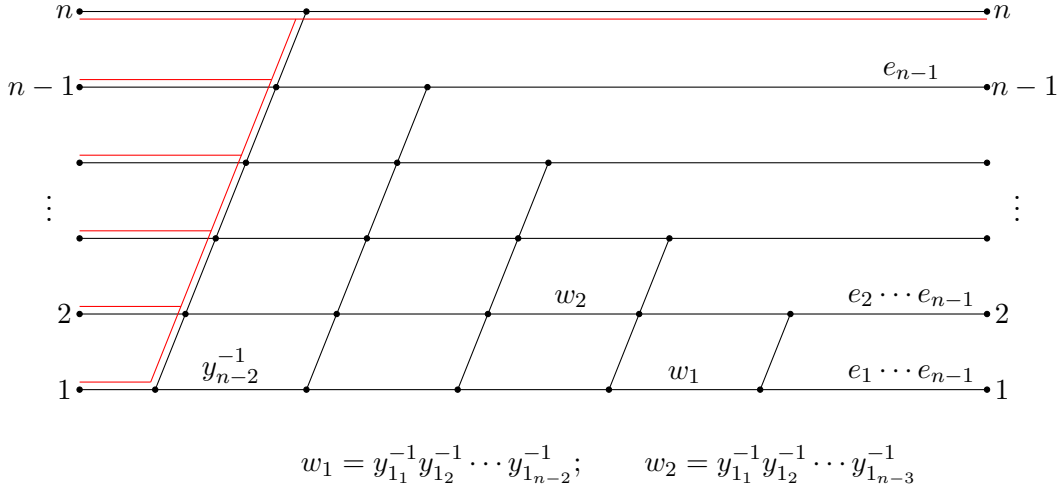
\begin{figure}
\centering
\begin{tikzpicture}
    \foreach \y in {0, ..., 5} {\draw (0,\y) -- (12,\y);}

    \foreach \k in {0, 1, 2, 3, 4} {
    \pgfmathtruncatemacro{\ymax}{5-\k}
    
    \draw (1 + 2*\k, 0) -- (1 + 2*\k + 0.4*\ymax, \ymax);
    
    \foreach \y in {0, ..., \ymax} {
        \filldraw (1 + 2*\k + 0.4*\y, \y) circle (1pt);
    }
    }

    \foreach \z in {0, ..., 5}{\filldraw (0,\z) circle (1pt);}
    \foreach \z in {0, ..., 5}{\filldraw (12,\z) circle (1pt);}

    \draw[red] (0,4.9) -- (12,4.9);
    \draw[red] (intersection of 0,0.1--4,0.1 and 0.9,0--2.9,5) -- (intersection of 0,4.9--12,4.9 and 0.9,0--2.9,5);
    \draw[red] (0,0.1)-- (intersection of 0,0.1--4,0.1 and 0.9,0--2.9,5);
    \draw[red] (0,1.1)-- (intersection of 0,1.1--4,1.1 and 0.9,0--2.9,5);
    \draw[red] (0,2.1)-- (intersection of 0,2.1--4,2.1 and 0.9,0--2.9,5);
    \draw[red] (0,3.1)-- (intersection of 0,3.1--4,3.1 and 0.9,0--2.9,5);
    \draw[red] (0,4.1)-- (intersection of 0,4.1--4,4.1 and 0.9,0--2.9,5);

    \foreach \p/\t in {(-0.2,0)/$1$, (12.2,0)/$1$, (-0.2,1)/$2$, (12.2,1)/$2$, (-0.5,4)/$n-1$, (12.5, 4)/$n-1$, (-0.2,5)/$n$, (12.2,5)/$n$, (-0.4,2.5)/$\vdots$, (12.4,2.5)/$\vdots$, (2,0.3)/$y^{-1}_{n-2}$, (8,0.2)/$w_1$, (6.5,1.2)/$w_2$, (11,0.2)/$e_1\cdots e_{n-1}$, (11,1.2)/$e_2\cdots e_{n-1}$, (11,4.2)/$e_{n-1}$}{\node at \p {\t};}

    \node at (7,-1) {$w_1=y_{1_1}^{-1}y_{1_2}^{-1}\cdots y_{1_{n-2}}^{-1}; \qquad w_2=y_{1_1}^{-1}y_{1_2}^{-1}\cdots y_{1_{n-3}}^{-1}$};
\end{tikzpicture}
\caption{The network for $T^{-1}E$ and the distinguished paths from all sources to the top sink.}
\label{fig:path-to-top}
\end{figure}

\medskip
\noindent\textbf{Two structural lemmas} We record a couple of observations about the structure of the planar networks of the building blocks, that we shall use later:

\begin{lem}\label{lem:path-to-top}
In the planar network for $T(t)^{-1}E(e)$, there is a distinguished path from each source to the $n$-th sink of weight $1$.
\end{lem}

\begin{proof}
For the factor $S\operatorname{Step}(d-r)S$, the network in \cite[Figure~23]{BGup} has a path of weight $1$ from level $r$ to level $r+1$, and horizontal paths of weight $1$ on the levels above it.  Concatenating these paths through Equation~(3.18) of \cite{BGup} gives a path from every source to the top level.  The final factor $SE(e)$ is diagonal and has top diagonal entry $1$ as in \eqref{eq:Ediag}.  This is exactly the path shown in Figure \ref{fig:path-to-top}.
\end{proof}

\begin{lem}\label{lem:universal-routing}
Let $I,J\subset \{1,2,\ldots,n\}$ with $|I|=|J|=r$.  In the weighted planar network associated with the product of the building blocks 
\[
 T^{-1}E\,\cdot\,TE
\]
there is a vertex-disjoint family from sources $I$ to sinks $J$.
\end{lem}

\begin{proof}
Write $I=\{i_1<\cdots<i_r\}$ and put
\[
 I_{\mathrm{top}}=\{n-r+1,\ldots,n\}.
\]
In the $T^{-1}E$ network, construct a path from source-$i_r$ to sink-$n$ using the first rising diagonal, source-$i_{r-1}$ to sink-$(n-1)$ using the second, and so on.  The paths are nested and vertex-disjoint; see Figure \ref{fig:route-up}.  In the $TE$ network, construct paths down to the ordered levels of $J$ using the descending diagonals; see Figure \ref{fig:route-down}.  Concatenating these paths defines the required family.
\end{proof}

\begin{figure}
\centering
\begin{subfigure}{.48\linewidth}
\centering
\resizebox{\linewidth}{!}{
\begin{tikzpicture}
    \foreach \y in {0, ..., 5} {\draw (0,\y) -- (12,\y);}

    \foreach \k in {0, 1, 2, 3, 4} {
    \pgfmathtruncatemacro{\ymax}{5-\k}
    
    \draw (1 + 2*\k, 0) -- (1 + 2*\k + 0.4*\ymax, \ymax);
    
    \foreach \y in {0, ..., \ymax} {
        \filldraw (1 + 2*\k + 0.4*\y, \y) circle (1pt);
    }
    }

    \foreach \z in {0, ..., 5}{\filldraw (0,\z) circle (1pt);}
    \foreach \z in {0, ..., 5}{\filldraw (12,\z) circle (1pt);}

    \coordinate (P1) at (intersection of 0,0.1--4,0.1 and 4.9,0--6.9,5);
    \coordinate (P2) at (intersection of 0,4.9--12,4.9 and 0.9,0--2.9,5);
    \coordinate (P3) at (intersection of 0,3.1--4,3.1 and 0.9,0--2.9,5);
    \coordinate (P4) at (intersection of 0,1.1--4,1.1 and 2.9,0--4.9,5);
    \coordinate (P5) at (intersection of 0,4.1--4,4.1 and 2.9,0--4.9,5);
    \coordinate (P6) at (intersection of 0,3.1--4,3.1 and 4.9,0--6.9,5);

    \draw[red] (0,1.1)-- (P4) (P4)--(P5) (P5)--(12,4.1);
    \draw[red] (0,0.1)--(P1) (P1)--(P6) (P6)--(12,3.1);
    \draw[red] (0,3.1)-- (P3) (P3) -- (P2) (P2) -- (12,4.9);

    \foreach \p/\t in {(-0.2,0)/$1$, (12.2,0)/$1$, (-0.2,1)/$2$, (12.2,1)/$2$, (-0.5,4)/$n-1$, (12.5, 4)/$n-1$, (-0.2,5)/$n$, (12.2,5)/$n$, (-0.4,2.5)/$\vdots$, (12.4,2.5)/$\vdots$}{\node at \p {\t};}
\end{tikzpicture}%
   }
\caption{Routing from $I$ to $I_{\mathrm{top}}$ in $T^{-1}E$.}
\label{fig:route-up}
\end{subfigure}\hfill
\begin{subfigure}{.48\linewidth}
\centering
\includegraphics[width=\linewidth]{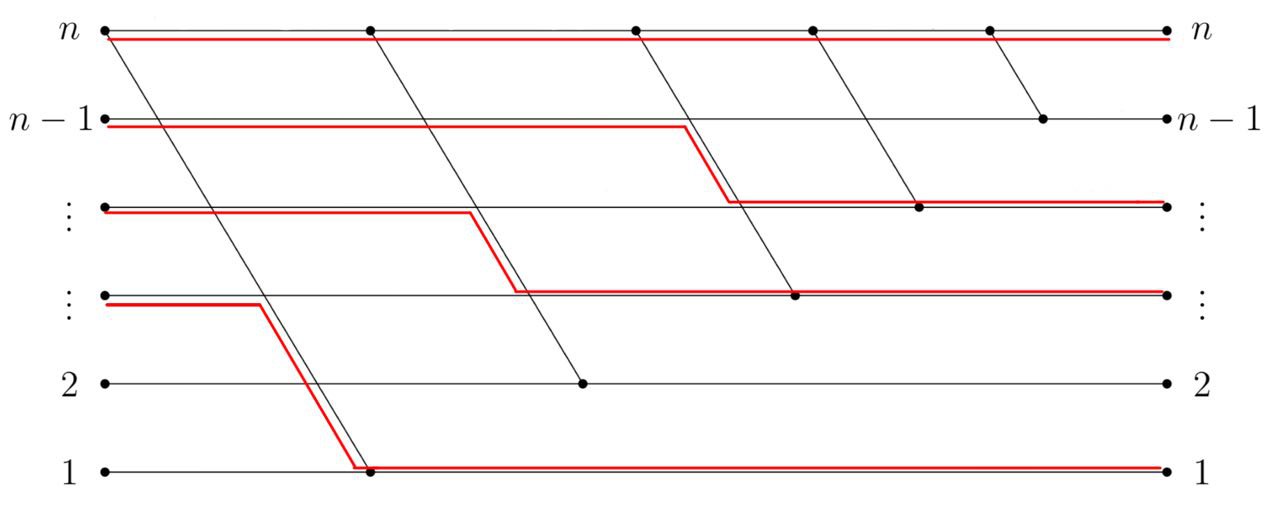}
\caption{Routing from $I_{\mathrm{top}}$ to $J$ in $TE$.}
\label{fig:route-down}
\end{subfigure}
\caption{The path families in Lemma \ref{lem:universal-routing}.}
\end{figure}
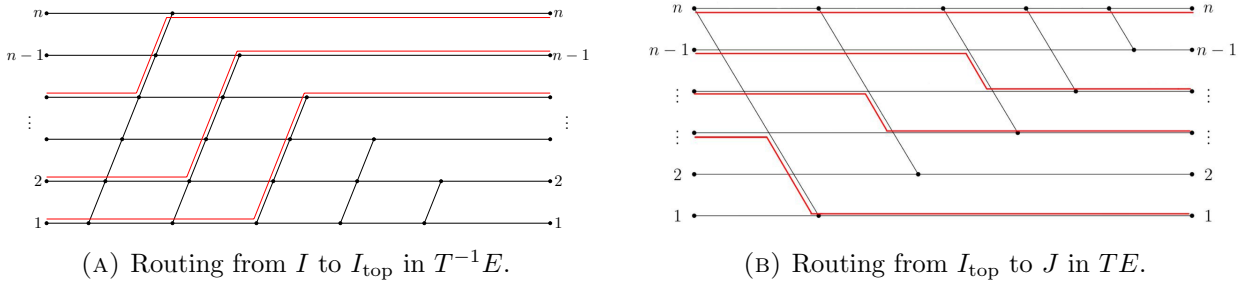

\section{Proof of the domination result} 

The proof of this follows the strategy in our preceding paper \cite{BGup}; however, in our former setting, when $S$ was a punctured surface, the ``maximal lamination" we were considering formed an ideal triangulation, and therefore had \textit{finitely} many leaves. As observed before, a general maximal lamination could have \textit{infinitely} many leaves, and to extend our previous argument we develop here a method of approximating monodromy by weight matrices of weighted planar networks,  using some of the technical machinery in \cite{MMMZ}. 

We shall first recall the proof of the domination result in \cite{BGup} for a finite lamination in \S3.1; this uses planar networks and the monodromy formula introduced in the previous section. Next, in \S3.2,  we shall describe the analogue of the monodromy formula in the closed surface case with ``generalized building blocks" that involve the slithering map. In \S3.3, we shall introduce the corresponding planar networks, and in \S3.4 we shall complete the proof of Theorem \ref{thm:main}.

\subsection{Punctured surface case} In this section assume that $S$ is a punctured surface, and $\rho:~\pi_1(S) \to \pslnc$ has a framing and Fock-Goncharov coordinates (possibly complex parameters) with respect to an ideal triangulation. 
We recall the argument  in \cite{BGup};  we have already reviewed, in the previous section, the decomposition of the monodromy matrix $\rho(\gamma)$ as a product of finitely many ``building-blocks" (see \eqref{decomp}), and the crucial observation that  each building-block was the weight matrix of a planar network (Lemma \ref{keylem}), with entries suitable monomials in the Fock-Goncharov coordinates. The  idea was to consider $\rho_0:\pi_1(S) \to \pslnr$ defined by new Fock-Goncharov coordinates obtained by replacing each with its modulus; since the coordinates are now all real and positive, $\rho_0$ is a positive representation (the analogue of a Hitchin representation for punctured-surface groups).

\medskip

\noindent We shall use the following notation:

\begin{defn}[Matrix domination] Let $A$ be a $n\times n$ matrix with real and non-negative entries and $B$ be an $n\times n$ matrix with complex entries. Then we say that $A$ \emph{dominates} $\lvert B \rvert$, denoted by $A \geq \lvert B\rvert$ (or $\lvert B \rvert \leq A$), if each entry of $A$ is  at least the modulus of the corresponding entry of $B$. 
\end{defn}

\noindent The key observation in \cite{BGup} was:

\begin{lem}[Propositions 3.6 and 3.10 of \cite{BGup}]\label{lem:dom} For any building block $B = T^{\pm 1}E$ in \eqref{decomp}, if one replaces each associated Fock-Goncharov parameter by its modulus, the resulting building block $B_0$ dominates $\lvert B \rvert$. 
\end{lem} 
\begin{proof} Recall from Lemma \ref{keylem} that $B$ is the weight matrix of a planar network $(\Gamma, \omega)$ where the weights $\omega$ are positive Laurent monomials in the Fock-Goncharov coordinates. The latter property ensures the fact that replacing each Fock-Goncharov coordinate with its modulus results in each weight in $\omega$ being replaced with its modulus. The weight matrix of the resulting planar network $(\Gamma, \lvert \omega\rvert)$ is $B_0$. Recall that the $ij$-th entry $b^\prime_{ij}$ of $B_0$ is the sum of the weights of paths from source-$i$ to sink-$j$; by the triangle inequality, we then have that $$b_{ij}^\prime = \lvert w_1 \rvert + \lvert w_2 \rvert + \cdots + \lvert w_k \rvert  \geq \lvert w_1 + w_2 +\ldots +w_k \rvert = \lvert b_{ij}\rvert $$  
where $w_1,w_2,\ldots,w_k$ are the (complex) weights of the paths from source-$i$ to sink-$j$ in $(\Gamma, \omega)$. 
\end{proof}

For a \textit{punctured} surface $S$, the domination result for the Hilbert length spectrum (Proposition \ref{prop:hilbg}) is then a consequence of the following basic linear algebra fact (see Lemma 2.27 of \cite{BGup} for an alternative approach):

\begin{lem}\label{lem:linalg} Let $A$ be a $n\times n$ matrix with real and positive entries and $B$ be an $n\times n$ matrix with complex entries, such that $A \geq \lvert B \rvert$. Then we have $\sigma(A) \geq \sigma(B)$ where $\sigma(M)$ denotes the spectral radius, the largest eigenvalue in modulus, of a matrix $M$. 
\end{lem}
\begin{proof} Since $A \geq \lvert B \rvert$ it follows that $A^k \geq  \lvert B\rvert^k \geq \lvert B^k \rvert$ for each $k\geq 0$. In particular, $\lVert A^k \rVert_\infty \geq \lVert B^k \rVert_\infty$ where $\lVert M \rVert_\infty$ denotes the maximum row-sum norm of a matrix $M$. Recall that Gelfand's formula for the spectral radius asserts that $\sigma(M) = \lim\limits_{k\to \infty} \lVert M^k \rVert_\infty^{1/k}$; applying this proves the inequality. 
\end{proof}

\begin{prop}[Proposition 3.11 of \cite{BGup}]\label{prop:hilbg} The representation $\rho_0$ dominates $\rho$ in the Hilbert length spectrum, that is, $\ell^H_{\rho_0}(\gamma) \geq  \ell^H_{\rho}(\gamma)$ for each $\gamma \in \pi_1(S)$.
\end{prop} 

\begin{proof}
Fix $\gamma \in \pi_1(S)$. From the monodromy formula \eqref{decomp} and Lemma \ref{keylem}, we know that there is an $A \in \slnc$ that is a lift of $\rho(\gamma)$ and is the weight matrix of a planar network $(\Gamma, \omega)$ obtained by concatenating the planar networks of each building block. By the previous lemma, the corresponding planar network associated with  $A_0$ (the lift of $\rho_0(\gamma)$) is obtained by replacing each weight by its modulus, and each entry of $A_0$ is at least the modulus of the corresponding entry of $A$. 

Lemma \ref{lem:linalg} then implies that the largest eigenvalue in modulus of $\rho_0(\gamma)$ is at least that of $\rho(\gamma)$. Applying the same argument to the curve $\gamma^{-1}$, we also obtain that the smallest eigenvalue in modulus of $\rho_0(\gamma)$ is at most that of $\rho(\gamma)$.
Since the Hilbert length of $\gamma$ is the ratio of these largest and smallest eigenvalues, and the choice of $\gamma$ was arbitrary, we conclude that $\ell_{\rho_0}^H(\gamma) \geq l_\rho^H(\gamma)$ for all $\gamma \in \pi_1(S)$.
\end{proof}

For  domination in translation length spectrum, we begin with the following consequence of Lindstr\"{o}m's lemma (Lemma \ref{lem:lind}):

\begin{lem}\label{lem:mindom}
    Let $A, A_0 \in \slnc$ be the lifts of $ \rho(\gamma), \rho_0(\gamma)$ respectively that are the weight matrices of planar networks above. Then for any minor we have the inequality $\Delta_{I,J}(A_0) \geq \lvert \Delta_{I,J}(A) \rvert $.
\end{lem}
\begin{proof}
As in the proof of Proposition \ref{prop:hilbg}, let  $(\Gamma, \omega)$ be the weighted planar network with weight matrix $A$, such that   $(\Gamma, \lvert \omega \rvert)$ has weight-matrix $A_0$. By Lemma \ref{lem:lind} the minor $\Delta_{I,J}(A_0)$ is the sum of weights of a vertex-disjoint family of paths from sources indexed by $I$ to sinks indexed by $J$. By the triangle inequality applied to these weights, as in the proof of Lemma \ref{lem:dom}, the desired inequality follows. 
\end{proof}

The domination result then follows from standard results in the theory of majorization; the following argument is easier than the one in \cite{BGup}:

\begin{prop}[\S4.3 of \cite{BGup}]\label{prop:trlbg}
    The representation $\rho_0$ dominates $\rho$ in the translation length spectrum, that is, $\ell_{\rho_0}(\gamma) \geq  \ell_{\rho}(\gamma)$ for each $\gamma \in \pi_1(S)$.
\end{prop}
\begin{proof}
Let $\gamma\in \pi_1(S)$ and as before let $A_0,A\in \slnc$ be the determinant-one lifts of $\rho_0(\gamma), \rho(\gamma)$ respectively. 
From  Lemma \ref{lem:mindom} it follows that for each $1\leq k \leq n$, we have $\bigwedge^k A_0 \geq  \lvert \bigwedge^k A \rvert $ since the entries of the $k$-th exterior power are exactly the $k\times k$ minors.  By Lemma \ref{lem:linalg} it follows that $$\sigma(\bigwedge^k A_0) = \prod\limits_{i=0}^{k-1} \lvert \lambda_{n-i}^\prime \rvert \geq  \prod\limits_{i=0}^{k-1} \lvert \lambda_{n-i} \rvert  =   \sigma(\bigwedge^k A)$$ where the eigenvalues of $A$ (respectively $A_0$) are $\lvert \lambda_1\rvert \leq \lvert \lambda_2\rvert \leq  \cdots \leq \lvert \lambda_n\rvert$ (respectively $\lvert \lambda_1^\prime\rvert \leq \lvert \lambda_2^\prime\rvert \leq  \cdots \leq \lvert \lambda_n^\prime \rvert $).  

Taking logarithms of both sides, we obtain $$\sum\limits_{i=0}^{k-1}  \log\lvert \lambda_{n-i}^\prime \rvert \geq \sum\limits_{i=0}^{k-1}  \log\lvert \lambda_{n-i} \rvert $$ for each $k\in \{1,2,\ldots, n\}$, and  by the determinant-one condition both sides are equal to zero for $k=n$.  In other words, we have the majorization inequality between the $n$-tuples  $$(\log \lvert \lambda_1\rvert, \log \lvert \lambda_2, \rvert, \ldots, \log \lvert \lambda_n\rvert) \precsim (\log \lvert \lambda_1^\prime\rvert, \log \lvert \lambda_2^\prime, \rvert, \ldots, \log \lvert \lambda_n^\prime\rvert).$$
Applying Karamata's inequality (see, for example, \cite[Proposition 3.C.1]{MOA}) to this, for the convex function $s\mapsto s^2$, we obtain $$\sum\limits_{i=1}^n (\log\lvert \lambda_i \rvert)^2 \leq \sum\limits_{i=1}^n (\log\lvert \lambda_i^\prime \rvert)^2$$ which yields the desired inequality of translation lengths (see Definition \ref{defn:lengths} (2)).
\end{proof}

For the case of the closed surface $S$, the same strategy for proving a domination result is viable when one has an analogue of Lemma \ref{lem:dom}; this shall be done in the rest of this section.

\subsection{Slithering map and generalized building blocks} We shall  need an analogue of Lemma \ref{lem:dom} in the case of a closed surface $S$ and a general maximal lamination, when the monodromy of curve (that typically crosses infinitely many plaques) would be an \textit{infinite} product of building blocks; to make sense of this, we first review the work in \cite{MMMZ}, in particular the construction of the ``slithering map"  (see \cite[Definition 4.11]{MMMZ} for a more detailed definition).

\begin{defn}[Slithering map]\label{defn:slith} For a closed surface $S$ and a maximal lamination $\lambda$, let $\rho:\pi_1(S) \to \pslnc$ be an $n$-pleated representation (as in Definition \ref{defn:pleat}), pleated along $\lambda$, with an associated $\rho$-equivariant limit map $\xi$ that assigns flags in $\mathbb{C}^n$ with the endpoints of leaves of $\tilde{\lambda}$. The \textit{slithering map} compatible with $\xi$  $$\Sigma:\tilde{\Lambda}^2 \to \slnc$$ defined on the space of pairs of leaves of $\lambda$, such that $\Sigma(g_1,g_2)$ is a matrix in $\slnc$ that takes the transverse pair of flags associated with the endpoints of a leaf $g_2$ to that of a leaf $g_1$ (oriented in parallel), satisfying
\[
\Sigma(g,g)=\mathrm{id}, \qquad
\Sigma(g_1,g_3)=\Sigma(g_1,g_2)\Sigma(g_2,g_3),
\]
whenever $g_2$ separates $g_1$ from $g_3$. When $g_1$ and $g_2$ share an endpoint, $\Sigma(g_1,g_2)$ is the unique unipotent transformation that fixes the flag at the common endpoint and takes the flag at the other endpoint of $g_2$ to the flag at the other endpoint of $g_1$.
\end{defn} 

The slithering map already appears in Bonahon's work for $n=2$, where it admits a geometric interpretation in terms of the horocyclic foliation transverse to $\lambda$.  The main technical difficulty in its construction is to show the convergence of an \textit{infinite} product of matrices associated to the plaques of $\widetilde{\lambda}$. This was also handled in \cite{BonShear} (for $\pslc$) as well as in \cite{BonDrey} (for $\pslnr$, see also \cite{SWZ} for the case of finite laminations). Indeed, if two leaves $g_1,g_2$ of $\widetilde{\lambda}$ are separated by infinitely many plaques,
the slithering map between them must be defined as such an infinite product.  To formalize this, Maloni-Martone-Mazzoli-Zhang defined (see Definition 5.1 of \cite{MMMZ}):

\begin{defn}[H\"{o}lder-extendable map]\label{defn:hext}
Let $\mathcal{P}(g_1,g_2)$ be the (possibly infinite) collection of plaques separating $g_1$ from $g_2$, with the linear order induced by the transverse orientation. A map $$M:\mathcal{P}(g_1,g_2) \to \slnc$$
is said to be \textit{H\"{o}lder-extendable} if there exist constants $A,\nu>0$ such that 
\begin{equation*}
    \lVert M(T) - \text{I}\rVert \leq A d_\infty(x_{T,1}, x_{T,3})^\nu
\end{equation*}
where $(x_{,1}, x_{T,2},x_{T,3})$ are the ordered set of vertices of $T$ between $g_1$ and $g_2$ and $d_\infty$ is a metric on the ideal boundary of $\tilde{S}$ that is the standard metric on the circle. 
\end{defn}

By \cite[Proposition 5.2]{MMMZ}, the ordered product $\prod\limits_{P\in \mathcal{F}_n} {M}(P)$  over any finite exhaustion $\mathcal{F}_n \nearrow   \mathcal{P}(g_1,g_2)$  then converges in $\slnc$, and is independent of the exhaustion.

One way to construct the slithering map is to consider the H\"older-extendable map $M$ for which $M(P)$ is the unipotent transformation associated with a plaque $P$ separating the two leaves $g_1,g_2$, that  fixes the flag at the ``middle" vertex of the ideal triangle $P$, and takes the flag at the endpoint of the side closer to $g_2$, to the flag at the remaining vertex (see \cite[Section~6]{MMMZ}). The ordered product over any increasing finite exhaustion converges by \cite[Proposition~5.2]{MMMZ} as noted above, and we have
\begin{equation}\label{decomp-limit} 
\Sigma_\rho(g_1,g_2) = {\prod\limits_{P\in\mathcal{P}(g_1,g_2)}} \overrightarrow{M}(P)
\end{equation} 
that defines the slithering map (see \cite[Theorem 4.12]{MMMZ}). 

\medskip

We shall now provide an analogue of the monodromy decomposition \eqref{decomp} in this setting.  In what follows, recall from Theorem \ref{bon-gen} that the $n$-pleated representation $\rho$ is identified with a cocycle $(\alpha_\rho,\theta_\rho) \in \Y(\lambda,n;\mathbb{C}/2\pi i\mathbb{Z})$. The triangle invariants and edge invariants are then exponentiated versions of $\alpha_\rho$ and $\theta_\rho$: namely,  let
\begin{equation}\label{eq:anbn} 
    \cA=\{(a,b)\in\Z_{>0}^2:a+b=n\} \text{ and }\cB=\{(a,b,c)\in\Z_{>0}^3:a+b+c=n\}.
\end{equation}
Then for a pair of distinct plaques $P,Q$ and $i\in\cA$, and for a labeled plaque $x$ and $j\in\cB$, we have 
\begin{equation}\label{eq:exp-coord}
    e_{i}(P,Q)=\exp\bigl(\alpha_\rho^{i}(P,Q)\bigr) \text{  and  } t_{j}(x)=\exp\bigl(\theta_\rho^{j}(x)\bigr).
\end{equation}
We shall denote these latter collections of coordinates by $e_\rho(P,Q)$ and $t_\rho(x)$ respectively.  These are the analogues of the triangle invariants and edge-invariants mentioned in \S2.2; indeed, the $t_\rho(x)$ are exactly the triangle-invariants of a generic triple of flags. We shall call  $e_\rho(P,Q)$ the \textit{shear-parameters} (between the plaques $P$ and $Q$), defined  in the following way: if  $g_1$ and $g_2$ are the separating sides (in the sense of \cite[\S 2.1.1]{MMMZ}) of $P$ and $Q$ respectively, then $e_\rho(P,Q)$ are the double ratios of the quadruple of flags $$(\xi(y_2), \xi(y_1), \xi(y_3), \Sigma_\rho(g_1,g_2) \cdot \xi(z))$$
where $y_1, y_2$ are the endpoints of $g_1$, $y_3$ is the other vertex of $P$, and $z$ is the vertex of $Q$ opposite to $g_2$ (see \cite[Section 4.1.5]{MMMZ}, and Figure \ref{fig:shear}).

\begin{figure}
\begin{center} 
\includegraphics[scale=.5]{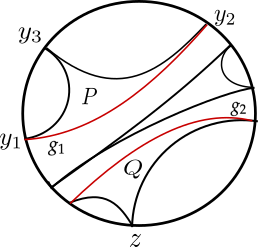}
\caption{The separating sides of the plaques $P$ and $Q$ are $g_1$ and $g_2$ respectively (shown in red), oriented parallel to each other. The slithering map $\Sigma_\rho(g_1,g_2)$ sends the flags at the endpoints of $g_2$ to those at the endpoints of $g_1$.}
\label{fig:shear}
\end{center}
\end{figure}

\medskip

The triangle matrix $T(t_\rho(P))$ as in \eqref{eq:Tmat} is still a change of basis matrix between the sides of  a single plaque $P$ and depends only on its triangle parameters.  Here we choose a normalization of the projective snake basis as in \cite[Section~7.4]{MMMZ}, namely the \textit{canonical snake basis} associated with the side $g_P$ of $P$ is 
\begin{equation*}
 \mathscr B_\rho(P,g_P)=(b_1,\ldots,b_n)
\end{equation*} 
where 
\begin{equation}\label{eq:basisb}
 b_i\in \xi(y_2)^i\cap\xi(y_1)^{n-i+1},
 \qquad
 \xi(y_3)^1=\operatorname{Span}\bigl(b_1+\cdots+b_n\bigr),
\end{equation}
where $(y_1,y_2,y_3)$ are the vertices of  $P$ and $g_P$ has endpoints $y_1,y_2$; the resulting triangle matrix has determinant one. 

As in \S2.4, an edge matrix $E(e_\rho(P,Q))$ can then be thought of as the change of canonical snake basis for the side $g_1$ of $P$ with endpoints corresponding to the pair of flags $(\xi(y_1), \xi(y_2))$, to the same pair considered as arising from the triple $(\xi(y_1), \xi(y_2), \Sigma_\rho(g_1,g_2) \cdot \xi(z))$ obtained by applying the slithering map $\Sigma_\rho(g_1,g_2) $ to the triple of flags associated with $Q$. As for the triangle matrix, this edge-matrix is normalized so that it has determinant one.  Note that when $P$ and $Q$ are adjacent, their separating sides coincide, and $\Sigma_\rho(g,g)=I$, recovering the edge matrix for finite laminations as described in \S2.4. 
\medskip

\noindent The above discussion leads to the following:

\begin{lem}[Generalized building block]\label{lem:genbb}
Let $P,Q$ be distinct plaques, with separating sides $g_P\subset\partial P$ and $g_Q\subset\partial Q$ oriented parallel.  The change of canonical snake basis of $g_P$ to that of a side $g_Q^\prime$ of $Q$ is given by a matrix 
\begin{equation}\label{eq:gen-bb}
 T\bigl(t_\rho(Q)\bigr)^{\delta}
 \;\Sigma_\rho(g_Q,g_P)\;
 E\bigl(e_\rho(P,Q)\bigr),
 \qquad \delta\in\{+1,-1\},
\end{equation}
where $\Sigma_\rho(g_Q,g_P)$ is a slithering map.
\end{lem}

\begin{proof}
By Definition \ref{defn:slith} the slithering map   $\Sigma_\rho(g_P,g_Q)$  (note the order of the arguments) sends the endpoint flags of $g_Q$ to those of $g_P$. From the preceding discussion $E(e_\rho(P,Q)$ takes the canonical snake basis $\mathscr B_\rho(P, g_P)$ to $\Sigma_\rho(g_P,g_Q) \cdot B_\rho(Q, g_Q)$. Postcomposing with $\Sigma_\rho(g_Q,g_P)$ that is the inverse of  $\Sigma_\rho(g_P,g_Q)$, results in the basis $B_\rho(Q, g_Q)$. Finally, postcomposing with $ T\bigl(t_\rho(Q)\bigr)^{\delta}$ changes the basis to the desired canonical basis $B_\rho(Q, g_Q^\prime)$.
\end{proof}

We now explain how these generalized building blocks provide an analogue of \eqref{decomp}. The idea is that given a nontrivial  $\gamma\in\pi_1(S)$ such that a lift $\widetilde c_\gamma\subset\widetilde S$ cuts across several plaques (in the complement of $\widetilde \lambda$), we can choose a \textit{coarse itinerary}, which is a finite (ordered) sequence of such plaques $$P_0,P_1,\ldots,P_m$$
where we assume that the interior of $P_0$ intersects $\widetilde c_\gamma$ and $P_m = \gamma \cdot P_0$. 

For each $1\leq i\leq m$, let \[ g_{i-1}^{+}\subset\partial P_{i-1}, \qquad g_i^{-}\subset\partial P_i \] be the two separating sides, oriented such that they are parallel, of the ordered pair $(P_{i-1},P_i)$.  Here we assume that $g_m^{+}:=\gamma g_0^{+}$. 
For each $i$, we also let $\delta_i\in\{+1,-1\}$ record the turn inside $P_i$ from $g_i^{-}$ to $g_i^{+}$, and  define the generalized building block
\[ \mathcal B_i:= T\bigl(t_\rho(P_i)\bigr)^{\delta_i} \Sigma_\rho(g_i^{-},g_{i-1}^{+}) E\bigl(e_\rho(P_{i-1},P_i)\bigr). \] 
that transforms the canonical snake basis  $\mathscr B_\rho(P_{i-1},g_{i-1}^{+})$ to $\mathscr B_\rho(P_{i},g_{i}^{+})$ by Lemma \ref{lem:genbb}. 

Applying this iteratively,  the change of canonical snake basis $\mathscr B_\rho(P_0,g_0^{+})$ to $\mathscr B_\rho(P_m,g_m^{+})$  is then given by the product 
\begin{equation}\label{eq:coarse-monodromy} 
A := \mathcal B_m\mathcal B_{m-1}\cdots\mathcal B_1 
\end{equation} 
which therefore represents a lift of  $\rho(\gamma)$ that has determinant one (since by our normalization all the factors have determinant one.) 

\medskip

Note that the coarse itinerary  $P_0,\ldots,P_m$ above need not be the entire set of plaques intersecting the segment $\widetilde c_\gamma$ that lifts $\gamma$ -- indeed, there could be \textit{infinitely many} such plaques it crosses between $P_{i-1}$ and $P_i$. The passage through all such intermediate plaques are encoded in the slithering matrix $ \Sigma_\rho(g_i^{-},g_{i-1}^{+})$, which by \eqref{decomp-limit}, is an infinite ordered product of unipotent factors. Consequently, \eqref{eq:coarse-monodromy} has the same finite $T^{\pm1}E$ skeleton as the usual monodromy word for a finite lamination, except that a slithering factor is inserted between the $T$ and $E$-matrices. In the next subsection, we shall replace each of the finitely many slithering factors by a finite approximant to obtain a finite approximation of the monodromy matrix.

\subsection{Transparent networks and approximants}
We shall now prove the properties that permit the punctured-surface argument to survive the insertion of a slithering matrix in the generalized building block \eqref{eq:gen-bb}.  The following property shall be useful in the proof of \textit{strict} domination in a later section. 

\begin{defn}[Transparent network]\label{def:transparent}
A weighted planar network of order $n$ is said to be \textit{transparent} if it contains 
a family of $n$ distinguished paths $h_r$ from  source $r$ to sink $r$ for each $1\leq r\leq n$, that are pairwise vertex-disjoint, and each of weight $1$.  
\end{defn}

\noindent We record a few immediate properties that will be useful later.

\begin{lem}\label{lem:transprop}
Let $\mathcal N$ be a transparent (weighted planar) network of order $n$, with weight
matrix $A$.
\begin{enumerate}
\item If a path in another network reaches sink-$r$ and if the network is concatenated with $\mathcal N$, then it can be continued
through the path $h_r$ without changing its weight. More generally, a vertex-disjoint family reaching pairwise distinct
levels $I\subset\{1,\ldots,n\}$ can be continued through the paths
$\{h_r:r\in I\}$, without changing either its weight or its
vertex-disjointness.

\item If $\mathcal N'$ is another transparent network, then the
concatenation $\mathcal N\mathcal N'$ is transparent.

\item Let $D=\emph{diag}(d_1,\ldots,d_n)$ be an invertible diagonal matrix.  There is a transparent network with
weight matrix $DAD^{-1}$.  

\item There is a transparent network with weight matrix $JAJ$ where  $J$ is the $n\times n$ anti-diagonal matrix with $1$s along the anti-diagonal. 
\end{enumerate}
\end{lem}

\begin{proof}
We leave the proofs of parts (1) and (2) to the reader.

For part (3),  for each $1\leq r \leq n$,  insert a new edge adjacent to  source-$r$ of weight $d_r$, and insert a new edge adjacent to sink-$r$ of weight $d_r^{-1}$, leaving all other edge weights unchanged.
The distinguished path from level $r$ to level $r$ is extended by
these two boundary edges and still has weight $1$. However, every path from source $i$ to sink $j$ has its weight multiplied by $d_i d_j^{-1}$.  Hence the new weight matrix is $DAD^{-1}$.  When $i=j$, the two factors cancel, so every distinguished channel still has weight $1$.

For part (4), let $\overline{\mathcal N}$ be the weighted planar network obtained by reflecting  $\mathcal N$ across a
horizontal line, while retaining the direction and weight of every edge. A path in $\overline{\mathcal N}$ from source-$i$ to sink-$j$ corresponds to a path in $\mathcal N$ from source-$(n+1-i)$ to sink-$(n+1-j)$. Note that $\overline{\mathcal N}$ is transparent, since the reflections of the distinguished paths of $\mathcal{N}$ constitute the desired path family in Definition \ref{def:transparent}. 
If $\overline{A}$ is the weight matrix of
$\overline{\mathcal N}$ , then its $(i,j)$-th entry is the $(n+1-i, n+1-j)$-th entry of $A$, that is, 
\begin{equation*}
    \overline{A} = JAJ
\end{equation*}
where  $J$ is the anti-diagonal matrix with $(i,n+1-i)$-th entry $1$, and the rest $0$.

For later use, note that the edge-reversal matrix defined in \eqref{eq:S} satisfies 
\begin{equation}\label{eq:SJD}
    S=D_\varepsilon J \text{ where }D_\varepsilon:=\text{diag}\bigl(1,-1,1,-1,\ldots,(-1)^{n-1}\bigr).
\end{equation}
Using the fact that  $J=J^{-1}$ we then see that 
\begin{equation*}
  SAS^{-1} = D_\varepsilon(JAJ)D_\varepsilon^{-1}
\end{equation*}
also has an associated weighted planar network obtained by applying part (3) of this lemma to the reflected network $\overline{\mathcal N}$. 
\end{proof}

In what follows we shall use the work in  \cite[Section 7.4]{MMMZ} that gives an alternative  H\"{o}lder extendable map used in the definition of the slithering map (see \eqref{decomp-limit}).  In their notation, if $h$ is a leaf of the lamination and $g_x^h$ is the boundary of a plaque $x$, then there is a  H\"{o}lder-extendable map 
\begin{equation}\label{eq:Lfactor}
L_{\rho}:\mathcal{P}(g_x^h,h) \to \slnc
\end{equation}
whose extension also gives the slithering map, that is, $\Sigma_\rho(g_2,g_1) = {\prod\limits_{P\in\mathcal{P}(g_1,g_2)}} \overrightarrow{L_\rho}(P)$ -- see \cite[Lemma 7.11(2)]{MMMZ}. Each $L_\rho(P)$ is a unipotent map, that depends on the shear parameters $e_\rho(x,P)$ in addition to the triangle parameters $t_\rho(P)$, with the advantage that it has an explicit expression in the canonical snake basis $\mathcal{B}_\rho(x,g_x^h)$ that does not depend on $P$ (see \cite[Lemma 7.11]{MMMZ}).  We shall describe $L_\rho(P)$ below, and relate it to the  matrix $M(t)$ in \eqref{eq:Tmat} involved in the change-of-basis matrix from the canonical snake basis associated with one side of $P$ to another, that depends only on the triangle parameters $t_\rho(P)$.  We start with a decomposition of the latter matrix $M(t)$ in the following lemma; we remark that this decomposition is also implicit in \cite[Lemma 9.4]{FG}.

\begin{figure}
\begin{center} 
\includegraphics[scale=.5]{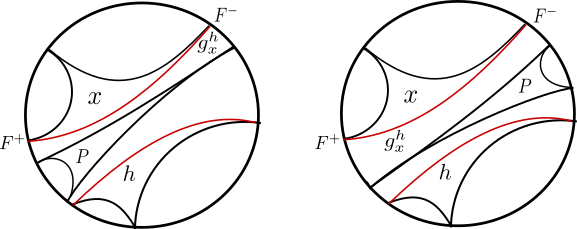}
\caption{The map $L_\rho(P)$ defined in \cite[Section 7.4]{MMMZ} fixes the flag at one of the endpoints of $g_x^h$ and takes the other flag to a flag $H_P = \Sigma_\rho(g_x^h, h_P^-) \cdot \xi(y)$ where $y$ is the vertex of $P$ opposite the separating side. In Lemma \ref{lem:Lnetw} we derive the two expressions for this for the two cyclic orders of the vertices of $P$.}
\label{fig:Lmap}
\end{center}
\end{figure} 

\begin{lem}\label{lem:Unetw}
The matrix $M(t)$ in \eqref{eq:Tmat} has a decomposition \begin{equation} \label{eq:MXD}
 M(t)
 =
 X_{i_1}(t)\cdot X_{i_2}(t)
 \cdots
 X_{i_N}(t) \cdot 
 D(t) 
\end{equation}
where $D(t)$ is a diagonal matrix with each entry a positive Laurent monomial in the $t$-parameters, and each $X_i(t)$ is the triangular matrix $I + m_iE_{i+1,i}$ where $m_i$ is a positive Laurent monomial in the triangle $t$-parameters. Moreover, the unipotent matrix $\widehat{M}(t):= M(t)D(t)^{-1}$ is lower-triangular, and the first column are all $1$s.
\end{lem}

\begin{proof}
    Recall from \S2.4 that $M(t)$ has an expression that is a product of the elementary matrices $F_i$ and $H_j$ in \eqref{eq:fihi}, where $F_i$ is lower-triangular with $1$s along the diagonal and $H_j$ are diagonal - see \cite[Equation (2.8)]{BGup} for the exact expressions. To commute a diagonal $H$-factor to the right, we use the fact that 
\begin{equation*} 
  HF_iH^{-1}
  =
  I+\chi_i(H)E_{i+1,i},
\end{equation*}
where $\chi_i(H)$ is the ratio of the $(i+1)$-st and $i$-th diagonal
entries of $H$; note that the conjugation preserves the same triangular form. Commuting all the diagonal $H_i$-factors to the right, one obtains \eqref{eq:MXD} as desired. Since the diagonal entries of every matrix $H_i$ are positive Laurent monomials in the $t$-parameters, the final triangular matrices have the form $I+ m_iE_{i+1,i}$ where the coefficient $m_i$ is also such a  positive Laurent monomial. 

For the final observation in the statement,  let $e_1,\ldots,e_n$ be the standard basis vectors. In the word defining $M(t)$ as in \cite[Equation (2.8)]{BGup}, apply the factors in the snake word to $e_1$ from right to left. Until the unique occurrence of $F_1$ is reached, each factor (diagonal $H$-factors or $F_j$ with $j\geq2$) fixes $e_1$. The factor $F_1$ then sends $e_1$ to $e_1+e_2$; similarly, applying the remaining
factors from right to left, each subsequent $F_r$ sends
$e_1+\cdots+e_r$ to $e_1+\cdots+e_{r+1}$, for each $r=2,3,\ldots n-1$.
Note that for each such $r$, the intervening diagonal factor
$H_{n-r}$ has its first $r$
diagonal entries equal to $1$, and hence fixes  $e_1+\cdots+e_{r+1}$. Thus
$M(t)e_1=e_1+\cdots+e_n$.
Moreover, every $H$-factor has first diagonal entry $1$,
so $D(t)e_1=e_1$, and hence 
\[
  \widehat{M}(t)e_1
  =M(t)D(t)^{-1}e_1
  =e_1+\cdots+e_n,
\]
which proves that the first column of $\widehat{M}(t)$
consists entirely of $1$s. \end{proof}

\noindent \textit{Remark.} The fact that $\widehat{M}(t)$ is unipotent is independent of the normalization of $M(t)$, that is, if it is replaced by the scalar multiple $c\cdot M(t)$ to make the determinant one, then the diagonal factor is also replaced by $c \dot D(t)$ and hence $\widehat{M}(t)$ remains the same. 

\medskip

We recall the definition of the unipotent matrix $L_\rho(P)$ for \eqref{eq:Lfactor} from \cite[\S7.4]{MMMZ} -- see before the statement of  Lemma 7.11 in their paper. In what follows, let $\mathcal{B}_\rho(x,g_x^h) =\{b_1,b_2,\ldots, b_n\}$ be the canonical snake basis associated with the geodesic side $g_x^h$ of $x$. We denote $F^+$ and $F^-$ to be the pair of opposite flags corresponding to this ordered basis and its reverse, respectively. These are the flags associated with the endpoints of $g_x^h$; note that the dimension-$1$ subspace of the flag $F$ at the remaining vertex of $x$ is $F^{(1)} = \langle b_1+b_2 \cdots +b_n \rangle$ (see \eqref{eq:basisb}).  The definition of  $L_\rho(P)$ depends on the two possible  cyclic orders of vertices $x_{P,1}, x_{P,2}, x_{P,3}$ of $P$ -- see Figure \ref{fig:Lmap}, where $x_{P,2}$ is the ``middle vertex" of $P$ such that the side determined by the pair $x_{P,1},x_{P,2}$ separates $P$ from $g_x^h$ and the side determined by the pair $x_{P,3},x_{P,2}$ separates $P$ from $h$.

For these two cyclic orders we have:

\medskip
\begin{itemize}

    \item[Case 1:] \textit{$x_{P,3} \prec  x_{P,2} \prec x_{P,1}$ in counter-clockwise order  (Figure \ref{fig:Lmap}, left).}  Let  $H_P$ be the unique flag such that the triangle invariants of the triple $(F^-, F^+, H_P)$ are precisely $\{t_j\ \vert\ j\in \mathcal{B}_n\}$ and the double ratios of the quadruple of flags $(F^+,F^-,F,H_P)$ are $e_i(x,P)$. Define $L_\rho(P)$ to be the unique unipotent matrix that fixes $F^+$ and takes $F^-$ to $H_P$. 

    \item[Case 2:] \textit{$x_{P,1} \prec  x_{P,2} \prec x_{P,3}$ in counter-clockwise order  (Figure \ref{fig:Lmap}, right).} Let  $H_P$ be the unique flag such that the triangle invariants of the triple $(F^+, F^-, H_P)$ are $\{t_j\ \vert\ j\in \mathcal{B}_n\}$ and the double ratios of the quadruple of flags $(F^-,F^+,F,H_P)$ are $e_{\hat{i}}(x,P)^{-1}$ where $\hat{(a,b)} = (b,a)$ in $\mathcal{A}_n$ (see \eqref{eq:anbn}).  Define $L_\rho(P)$ to be the unique unipotent matrix that fixes $F^-$ and takes $F^+$ to $H_P$.
    
\end{itemize}

Note that in the notation of \cite{MMMZ}, $\xi(y_1^h) =F^-$ and $\xi(y_2^h) =F^+$, and in both cases, the matrix $L_\rho(P)$ is  $U(\xi(y_1^P), \xi(y_2^P),H_P)$; in case 1, it is upper-triangular and in case 2 it is lower-triangular in the basis $\mathcal{B}_\rho(x,g_x^h)$.

\medskip

\noindent For the following lemma, recall that $S = D_\epsilon J$ is the anti-diagonal matrix \eqref{eq:S} where $D_\epsilon$ is a diagonal matrix and $J$ is the anti-diagonal matrix with $1$s along the diagonal (see \eqref{eq:SJD}).  Moreover, for ease of notation,  we index the shear parameters by their first component, i.e. $e_i:=e_{(i,n-i)}(x,P)$ for each $1\leq i <n$.

\begin{lem} \label{lem:Lnetw}
The unipotent matrix $L_\rho(P)$ for \eqref{eq:Lfactor} is obtained by conjugating $J\widehat{M}(t)J$ or $\widehat{M}(t)$ by a diagonal matrix involving the shear parameters $e_i(x,P)$. The two cases depend on the cyclic order of the vertices of $P$:
\begin{equation}\label{eq:Lform}
L_\rho(P)= \begin{cases}
    D_\epsilon Q_+J \widehat{M} JQ_+^{-1}D_\epsilon^{-1} & \text{if } x_{P,3} \prec  x_{P,2} \prec x_{P,1} \\
      D_\epsilon Q_-\widehat{M}Q_-^{-1}D_\epsilon^{-1} & \text{if } x_{P,1} \prec  x_{P,2} \prec x_{P,3}
\end{cases}
\end{equation}
where $Q_+$ is a diagonal matrix with entries $\prod\limits_{i=r}^{n-1}e_i$ and $Q_-$ is a diagonal matrix with entries $\prod\limits_{i=1}^{r-1}e_i^{-1}$, where $1\leq r\leq n$ and the empty product is $1$.

In either case, $D_\epsilon^{-1}L_\rho(P)D_\epsilon$ is the  weight matrix of a concatenation of transparent weighted planar networks whose non-constant edge-weights are positive Laurent monomials in $t$ and $e$. 
\end{lem}

\begin{proof}
     By the construction in the preceding Lemma, the matrix $\widehat{M}(t)$ is a lower-triangular unipotent matrix that fixes $F^-$ and takes $F^+$ to the flag $F^\prime$ such that the triangle invariants of the triple of flags $(F^+,F^-,F^\prime)$ are precisely $t$ (\textit{cf.} the discussion regarding snakes and the triangle matrix in \S2.4; note that the diagonal matrix $D(t)$ fixes $F^\pm$).  Moreover, $\widehat{M}(t)$ sends the dimension-$1$ subspace $\langle b_1 \rangle$ of $F^+$ to $\langle b_1 + b_2 + \cdots + b_n \rangle$, the dimension-$1$ subspace of $F^\prime$.

     If we assume the dimension-$1$ subspace of $H_P$ is $\langle h_1b_1 + h_2b_2 + \cdots + h_nb_n \rangle$, the key computation is that from the definition of double ratio (see, for example \cite[\S4.1.2]{MMMZ}), the $i$-th double ratio of the quadruple $(F^+,F^-, F, H_P)$ is exactly $-h_i/h_{i+1}$ (see, for example, the calculation in the proof of \cite[Proposition 2.17]{Martone}).

    The conjugation $S\widehat{M}S^{-1}$ is an upper-triangular matrix that fixes $F^+$, and takes $F^-$ to a flag $F^\prime$ such that the triangle invariants of the triple of flags $(F^-,F^+,F^\prime)$ are precisely $t$. Since it sends  the  dimension-$1$ subspace $\langle b_n \rangle$ of $F^-$ to $\langle b_1 - b_2 +b_3 + \cdots \rangle$, the above double-ratios are all $1$ (equivalently, the shear parameters $\alpha_\rho(x,P)$ are all zero). Thus  the further conjugation $Q_+S\widehat{M}S^{-1} Q_+^{-1}$ ensures that each double ratio is the desired $e_i$. Note that since $Q_+$ is diagonal, it fixes the flags $F_\pm$. Thus in Case 1, $L_\rho(P) = Q_+S\widehat{M}S^{-1} Q_+^{-1} = D_\epsilon Q_+J \widehat{M} JQ_+^{-1}D_\epsilon^{-1}$.

    For Case 2 corresponding to the other cyclic order, when the double ratios are $e_{n-i}^{-1}$, the previous computation ensures that in the \textit{reversed} basis $\{b_n, b_{n-1},\ldots, b_1\}$ the conjugation $QS\widehat{M}S^{-1} Q^{-1}$ is the desired unipotent matrix $L_\rho(P)$, where  $Q$ is a diagonal matrix with entries $\prod\limits_{i=r}^{n-1}e_{n-i}^{-1}$. To change to the desired basis  $\{b_1, b_{2},\ldots, b_n\}$, we conjugate by $J$ to obtain $L_\rho(P) = JQS\widehat{M}S^{-1} Q^{-1}J$. Now note that $JQJ = Q_-$; then using the fact that $J^2=I$ we write  $L_\rho(P) = Q_-JS\widehat{M}S^{-1}JQ_-^{-1} = Q_-D_\epsilon \widehat{M} D_\epsilon^{-1}Q_-^{-1} = (D_\epsilon Q_-)\widehat{M}(D_\epsilon Q_-)^{-1}$ since $D_\epsilon$ and $Q_-$ commute (being diagonal). 

    The statement about the associated networks follows since by Lemma \ref{lem:Unetw}, $\widehat{M}(t)$ is the weight matrix of a weighted planar network obtained by concatenating transparent planar networks corresponding to the factors in \eqref{eq:MXD}, where  the transparent planar network for $X_i(t)$ comprises \(n\) horizontal paths of weight \(1\) and one additional directed edge between the levels $i$ and $i+1$, with weight a positive Laurent monomial (see \cite[Fig. 16]{BGup}). For the subsequent conjugations, we can apply parts (3) and (4) of Lemma \ref{lem:transprop}.   \end{proof}

\noindent\textit{Remark.}
To obtain a transparent network \((\Gamma,\omega)\) for
\(L_\rho(P)\) itself, append boundary edges realizing
conjugation by \(D_\epsilon\): at level \(r\), attach a source
edge of weight \(\epsilon_r=(-1)^{r-1}\) and a sink edge of
weight \(\epsilon_r^{-1}\).
The distinguished horizontal path is extended by these edges
and still has weight exactly \(1\).
All edge weights are positive Laurent monomials in \(t,e\),
or constant signs in \(\{1,-1\}\).

Replacing every weight by its modulus removes the boundary
signs and replaces the parameters \(t,e\) by \(|t|,|e|\).
Since these are the parameters of \(\rho_0\), the Hitchin representation in the bending fiber of $\rho$ (see \eqref{eq:rho0coords}), the resulting
network has the weight matrix $ D_\epsilon^{-1}L_{\rho_0}(P)D_\epsilon
 =D_\epsilon L_{\rho_0}(P)D_\epsilon^{-1}$
for both cyclic orders, where the last equality uses \(D_\epsilon^2=I\).

\medskip

\noindent\textbf{Example.}\label{ex:n3-normalized-snake} We illustrate the decomposition \eqref{eq:MXD} in the case $n=3$, in which case the plaque $P$ has a unique triangle parameter $t\in \mathbb{C}^\ast$.
By the formula \cite[(3.7)]{BGup}, we have
\begin{equation*}
  M(t)=F_2H_1(t)F_1F_2
\end{equation*}
that yields, by commuting the diagonal $H$-factor to the right as in the proof of the preceding lemma, 
\[
 M(t)
 =
 \begin{pmatrix}
 1&0&0\\
 1&1&0\\
 1&1+t&t
 \end{pmatrix} = X_2(1) \cdot X_1(1) \cdot X_2(t) \cdot D(t)
\]
where $X_1(t) = I + tE_{21}$, $X_2(t) = I + tE_{32}$ and $D(t)=\text{diag}(1,1,t)$. (Here, for the ease of notation, we have omitted the scalar normalization factor.) Thus
\[
 \widehat M(t)=M(t)D(t)^{-1}
 =
 \begin{pmatrix}
 1&0&0\\
 1&1&0\\
 1&1+t&1
 \end{pmatrix}
\]
which is independent of the normalization.  
Writing \(e_1=e_\rho^{(1,2)}(x,P)\) and
\(e_2=e_\rho^{(2,1)}(x,P)\), the diagonal matrices in
\eqref{eq:Lform} are
\[
 D_\epsilon=\operatorname{diag}(1,-1,1),\qquad
 Q_+=\operatorname{diag}(e_1e_2,e_2,1),\qquad
 Q_-=\operatorname{diag}\bigl(1,e_1^{-1},(e_1e_2)^{-1}\bigr).
\]
The two forms of the unipotent matrix are therefore
\[
 L_\rho(P)=
 \begin{cases}
 \begin{pmatrix}
 1&-e_1(1+t)&e_1e_2\\
 0&1&-e_2\\
 0&0&1
 \end{pmatrix}
 &\text{if }x_{P,3}\prec x_{P,2}\prec x_{P,1}
 \\[5mm]
 \begin{pmatrix}
 1&0&0\\
 -e_1^{-1}&1&0\\
 (e_1e_2)^{-1}&-(1+t)e_2^{-1}&1
 \end{pmatrix}
 &\text{if }x_{P,1}\prec x_{P,2}\prec x_{P,3}.
 \end{cases}
\]

\medskip

\begin{cor}
\label{cor:slithering-networks}
Fix $x$ and $h$ as above, and let $\mathcal F_N \nearrow \mathcal P(g_x^h,h) $ be an increasing finite exhaustion.  Let
\[
 S_{\rho,N}(g_x^h,h)
 :=
 \prod_{T\in \mathcal{F}_N}^{\longrightarrow}
 L_{\rho}(T).
\]
Then $S_{\rho,N}(g_x^h,h)$ is the weight matrix of a transparent
weighted planar network $(\Gamma ,\omega)$ and $ S_{\rho,N}(g_x^h,h)
 \longrightarrow
 \Sigma_\rho(h, g_x^h)$. Moreover, the weight matrix of $(\Gamma ,\lvert \omega\rvert)$ is $D_\varepsilon S_{\rho_0,N}D_\varepsilon^{-1}$ where $\rho_0$ is the Hitchin representation in the bending fiber of $\rho$ (see \eqref{eq:rho0coords}). 
\end{cor}

\begin{proof}
Concatenate the networks from
Lemma~\ref{lem:Lnetw}.  Transparency is preserved by
Lemma~\ref{lem:transprop}(2), and the convergence  (independent of
the exhaustion) follows from
\cite[Proposition~5.2 and Lemma~7.11(2)]{MMMZ}. The observation about the weight matrix of the planar network obtained by taking the modulus of each weight follows from the preceding Remark, since the $D_\varepsilon$ factors cancel out in the finite product. 
\end{proof}

\begin{cor}[Approximants of a generalized building block]
\label{cor:approx-gbb}
Consider a generalized building block
\[
 B_\rho
 =
 T\bigl(t_\rho(Q)\bigr)^\delta
 \Sigma_\rho(g_Q,g_P)
 E\bigl(e_\rho(P,Q)\bigr),
 \qquad
 \delta\in\{-1,+1\},
\]
where $g_P$ is the side of $P$ separating $P$ from $Q$ and $g_Q$ is
the corresponding side of $Q$.  Replace the slithering matrix by the
finite approximant from
Corollary~\ref{cor:slithering-networks}, and define
\[
 B_{\rho,N}
 =
 T\bigl(t_\rho(Q)\bigr)^\delta
 S_{\rho,N}(g_P,g_Q)
 E\bigl(e_\rho(P,Q)\bigr).
\]
Then $B_{\rho,N}$ is the weight matrix of a weighted planar network
$(\mathcal B_{\rho,N},\omega_{\rho,N})$. Moreover, the weight matrix of the planar network $(\mathcal B_{\rho,N}, \lvert \omega_{\rho,N} \rvert)$ is the corresponding generalized building block $B_{\rho_0,N}$ for $\rho_0$, the Hitchin representation in the bending fiber (see \eqref{eq:rho0coords}). 
\end{cor}

\begin{proof}
Write, as in \cite[Equations~(2.4) and~(2.10)]{BGup},
\[
 T(t)=M(t)S,
 \qquad
 E(e)=D(e)S,
\]
where $D(e)$ is diagonal and $S$ is the signed side-reversal matrix.

Consider the case when $\delta=+1$.  Then
\begin{align}
 T(t)S_{\rho,N}E(e)
 &=
 M(t)
 \bigl(SS_{\rho,N}S^{-1}\bigr)
 \bigl(SD(e)S\bigr).
 \label{eq:positive-coarse-building block}
\end{align}
The matrix $M(t)$ is the weight-matrix of a planar network by
\cite[Lemma~3.4 and Proposition~3.6]{BGup}, so is the central factor by Corollary \ref{cor:slithering-networks} and part (4) of Lemma \ref{lem:transprop}, and the remaining factor is a diagonal matrix that is the weight matrix of a planar network with exactly $n$ edges at every level, weighted by the diagonal entries. The desired weighted planar network $(\mathcal B_{\rho,N},\omega_{\rho,N})$ is then the concatenation of these.

If $\delta=-1$, then
\begin{align}
 T(t)^{-1}S_{\rho,N}E(e)
 &=
 \bigl(S^{-1}M(t)^{-1}S\bigr)
 \bigl(S^{-1}S_{\rho,N}S\bigr)
 \bigl(S^{-1}D(e)S\bigr)
 \label{eq:negative-coarse-building block}
\end{align}
and a similar argument applies, where this time the weighted planar network associated with the first factor was constructed in
\cite[Equations~(3.13)--(3.18) and Proposition~3.10]{BGup}. 
The non-constant weights of these networks are Laurent monomials with a positive real coefficient, as that follows from the corresponding property of the H\"{o}lder-extendable $L$-factors (see Lemma \ref{lem:Lnetw}) and of the other factors (see \cite[Lemma 3.4 and the proof of Proposition 3.10]{BGup}). 

For the modulus network, Corollary \ref{cor:slithering-networks} gives
\(D_\varepsilon S_{\rho_0,N}D_\varepsilon^{-1}\) for the
central factors in \eqref{eq:positive-coarse-building block} or \eqref{eq:negative-coarse-building block}. Replacing each weight by its modulus in the corresponding planar network, the 
construction of Lemma \ref{lem:transprop}(4) yields conjugation by the antidiagonal matrix $J$. This is because replacing each weight of the planar network for $D_\epsilon$ by its modulus yields a planar network whose  weight matrix equals the identity matrix $I$. Hence the resulting planar network for the central factor has weight matrix
\begin{equation*} 
J(D_\varepsilon S_{\rho_0,N}D_\varepsilon^{-1})J
 =(JD_\varepsilon)S_{\rho_0,N}(JD_\varepsilon)^{-1} =S^{-1}S_{\rho_0,N}S
  =SS_{\rho_0,N}S^{-1},
\end{equation*}
where the final equality follows from
\(S^2=(-1)^{n-1}I\).  This is precisely the central factor occurring
for the Hitchin representation  \(\rho_0\) in both \eqref{eq:positive-coarse-building block} or \eqref{eq:negative-coarse-building block}.
 \end{proof}

\subsection{Completing the proof of Theorem \ref{thm:main}}

Recall that by Theorem \ref{bon-gen}, the representation $\rho$ is determined by the complex-valued $\lambda$-cocycle $\sigma = (\alpha_\rho,\theta_\rho) \in \Y(\lambda,n;\mathbb{C}/2\pi i\mathbb{Z})$.

By \cite[Theorem~4.15]{MMMZ}, after fixing a base plaque and a canonical snake basis for one of its sides, there is a unique Hitchin representation $\rho_0:\pi_1(S) \to \pslnc$ 
with coordinates
\begin{equation}\label{eq:rho0coords}
 (\alpha_{\rho_0},\theta_{\rho_0})
 = (\operatorname{Re}\alpha_\rho,\operatorname{Re}\theta_\rho).
\end{equation} 

In particular, the corresponding triangle and shear parameters of $\rho_0$ (see \eqref{eq:exp-coord}) are obtained by replacing each corresponding parameter of $\rho$ by its modulus:
\begin{equation}\label{eq:closed-moduli}
\exp({\alpha_{\rho_0}^{i}(P,Q)})=|e_{i}(P,Q)|,
 \qquad
 \exp({\theta_{\rho_0}^{j}(x)})=|t_{j}(x)|
\end{equation}
where $i\in \mathcal{A}_n$ and $j\in \mathcal{B}_n$ as in \eqref{eq:anbn}.

\medskip

\noindent The replacement of Lemma \ref{lem:dom} and  Lemma \ref{lem:mindom} in this context is then the following lemma about the finite approximants of the monodromy matrix of $\gamma$ obtained by concatenating finite approximants of the generalized building block (see Corollary \ref{cor:approx-gbb}).  Recall that $\bigwedge^k A$ denotes the $k$-th Grassmann power or exterior power of $A$,  whose entries are the $k\times k$ minors of $A$.

\begin{lem}\label{lem:closed-monodromy-approximants}
Let $\gamma\in\pi_1(S)$ be an essential curve, and assume that the axis of $\gamma$ in $\widetilde S$ is not a leaf of $\widetilde\lambda$. Let $A,A_0 \in \slnc$ be the normalized (determinant-one) lifts of the monodromies $\rho(\gamma), \rho_0(\gamma)$ respectively. 
Then there is a sequence of approximating matrices
$A_{N},A_{0,N}\in\slnc$ with the following properties:
\begin{enumerate}
\item[(i)] $A_{N}\to A$ and $A_{0,N}\to A_0$ as $N\to \infty$;
\item[(ii)] for each $N$, there is a (finite) weighted planar network $(\mathcal A_N,\omega_N)$ with weight matrix $A_{N}$, and the weight matrix of $(\mathcal A_N, \lvert \omega_N \rvert)$  is $A_{0,N}$.
\item[(iii)] for every $1\leq k\leq n$, we have
\begin{equation}\label{eq:compound-finite-comparison}
 \bigl|\bigwedge^k A_N\bigr|
 \leq \bigwedge^k A_{0,N}
\end{equation}
entrywise.
\end{enumerate}
\end{lem}

\begin{proof}
We shall choose a particularly simple coarse itinerary with a single building block (\textit{cf.} \eqref{eq:coarse-monodromy}): choose a plaque $P$ whose interior meets a lift $\tilde{c}_\gamma$ of $\gamma$, and let $Q=\gamma \cdot P$.  Orient $\tilde{c}_\gamma$  from $P$ toward $Q$.  Let $g_P\subset\partial P$ be the side through which the axis exits $P$, and let $g_Q\subset\partial Q$ be the side through which it enters $Q$, such that $g_P,g_Q$ are the separating sides.  We choose the canonical snake basis $\mathcal B_\rho(P, g_P)$ associated with the side $g_P$.  Its $\gamma$-translate is then the canonical basis $\mathcal B_\rho(Q, g_Q^\prime)$  where $g_Q^\prime = \gamma \cdot g_P$ is a side of $Q$.  Note that $g_Q$ and $g_Q^\prime$ are distinct and determine a positive or negative turn inside $Q$; denote its sign by $\delta$.

By Lemma \ref{lem:genbb} and \eqref{decomp-limit}, the generalized building block
\begin{equation}\label{eq:closed-monodromy-factorization}
 A
 =T\bigl(t_\rho(Q)\bigr)^\delta
  \Sigma_\rho(g_Q,g_P)
  E\bigl(e_\rho(P,Q)\bigr)
\end{equation}
sends the chosen base snake to its $\gamma$-translate and is a determinant-one lift of $\rho(\gamma)$.  The same construction for $\rho_0$ gives an expression for the determinant-one lift $A_0$.

Choose an increasing finite exhaustion of $\Plaques(g_P,g_Q)$ and replace the slithering maps in \eqref{eq:closed-monodromy-factorization} for  $\rho(\gamma)$ by the approximants  $S_{\rho,N}(g_P,g_Q)$ (see Corollary \ref{cor:approx-gbb}) to obtain approximations of $A$ we shall denote by $A_{N}$. The corresponding  approximants  of $A_0$ obtained by using the approximants  $S_{0,N}(g_P,g_Q)$ of the slithering map shall be denoted by $A_{0,N}$; note that (i) holds by Corollary \ref{cor:slithering-networks}. 
By the final statement of Corollary \ref{cor:approx-gbb} $A_{0,N}$ is the weight matrix of the planar network $(\mathcal A_N, \lvert \omega_N\rvert )$ obtained by replacing each weight by its modulus; this proves (ii). 

Finally, apply Lindstr\"om's lemma to every $k\times k$ minor of the two planar networks in (ii).  The triangle inequality, as in the proof of Lemma \ref{lem:dom} then implies 
\[
 \abs{\minor{I}{J}{A_{N}}}
 \leq \minor{I}{J}{A_{0,N}}
\]
for any $|I|=|J|=k$ which is the desired domination \eqref{eq:compound-finite-comparison}.
\end{proof}

\medskip 

\noindent  We finally obtain:

\begin{proof}[Proof of Theorem \ref{thm:main}]
Assume that $\gamma$ is not a closed leaf of $\lambda$. In this case, each of the two normalized (determinant-one) lifts $A,A_0$ of $\rho(\gamma),\rho_0(\gamma)$ respectively is a product of finitely many generalized building blocks. By Lemma \ref{lem:closed-monodromy-approximants} there are  matrices $A_N, A_{0,N}$ such that $A_N \to A$ and $A_{0,N} \to A_0$ as $N\to \infty$.  
Moreover, each minor of $A_{0,N}$ is at least the corresponding minor of $A_N$ in modulus, see \eqref{eq:compound-finite-comparison}. By the convergence, it follows that each $k\times k$ minor of $A_0$ is at least the corresponding minor of $A$ in modulus, for each $1\leq k\leq n$. Note that for $k=1$, this implies the entry-wise domination $A_0 \geq \lvert A \rvert$. The argument in the proofs of Propositions \ref{prop:hilbg} and \ref{prop:trlbg} then imply that $\ell_\rho^H(\gamma)  \leq  \ell_{\rho_0}^H(\gamma)$ and $\ell_\rho^X(\gamma) \leq \ell_{\rho_0}^X(\gamma)$. By Lemma \ref{lem:linalg} it follows that for each $1\leq k<n$ we have $$\sigma(\bigwedge^k A) = \prod\limits_{i=0}^{k-1} \lvert \lambda_{n-i} \rvert  \leq  \prod\limits_{i=0}^{k-1} \lvert \lambda_{n-i}^\prime \rvert  =  \sigma(\bigwedge^k A_0)$$ where the eigenvalues of $A$ (respectively $A_0$) are $\lvert \lambda_1\rvert \leq \lvert \lambda_2\rvert \leq  \cdots \leq \lvert \lambda_n\rvert$ (respectively $\lvert \lambda_1^\prime\rvert \leq \lvert \lambda_2^\prime\rvert \leq  \cdots \leq \lvert \lambda_n^\prime \rvert $).  By Definition \ref{defn:lengths} (3), this implies the inequality of fundamental-weights $\ell_{\rho}^{\omega_k}(\gamma) \leq \ell_{\rho_0}^{\omega_k}(\gamma)$ for any $\gamma \in \pi_1(S)$.

In the case that $\gamma$ is a closed leaf of $\lambda$,  applying  
\cite[Section~3.3 and Theorem~10.1]{MMMZ} we can express the logarithms of the  successive ratios of eigenvalues ($\log \frac{|\lambda_{i+1}|}{|\lambda_{i}|}$) of $\rho(\gamma)$ in terms of
$\operatorname{Re}\alpha_\rho$. Since this real part of the cocycle is identical for $\rho_0$, together with the determinant-one condition, we obtain $\lvert \lambda_i\rvert= \lvert \lambda_i^\prime \rvert $. By Definition \ref{defn:lengths}, it follows that in this case the Hilbert and translation lengths are equal, namely  $\ell_\rho^H(\gamma) =\ell_{\rho_0}^H(\gamma)$, $\ell_\rho^X(\gamma) = \ell_{\rho_0}^X(\gamma)$ and so are the fundamental weight lengths  $\ell_{\rho}^{\omega_k}(\gamma) = \ell_{\rho_0}^{\omega_k}(\gamma)$ for each $1\leq k<n$.
\end{proof}

In the next section, we shall in fact prove a sharper result, as for the entropy rigidity we would need a \textit{strict} domination, at least for ``most" curves.

\section{Proof of statistical domination}

\subsection{Statistical domination and entropy} 
From the definition of Hilbert entropy of a surface-group representation into $\pslnc$ (see Definition \ref{defn:ent}), we have the following immediate corollary (\textit{cf.} Corollary B of \cite{DT}):

\begin{lem}\label{lem:ent-dom}  Given two representations $\rho, \sigma:\pi_1(S) \to \pslnc$ such that $ \rho$ dominates $\sigma$ in a length spectrum, we have that the corresponding entropy $H(\sigma) \geq  H(\rho)$.  Moreover, if the domination is strict, so is the entropy inequality. 
\end{lem} 

\begin{proof} Suppose there exists a positive real $\lambda \leq 1$ such that  $\ell_\sigma(\gamma) \leq \lambda \ell_\rho(\gamma)$ for all $\gamma \in \pi_1(S)$. Here, recall that the domination is said to be \textit{strict} if and only if $\lambda <1$. 

In what follows, $[\gamma]$ varies amongst all conjugacy classes in $\pi_1(S)$. We have an inclusion of subsets
\begin{equation}\label{eq:sub} 
\{[\gamma] \ \vert\ \ell_\rho(\gamma) \leq L\} \subset \{[\gamma]\ \vert\ \ell_\sigma(\gamma) \leq \lambda L\}.
\end{equation}
Using Definition \ref{defn:ent} we then obtain
\begin{equation*}
H(\rho) =  \limsup\limits_{L\to \infty} \displaystyle\frac{\log \#\{[\gamma]  \ \vert\ \ell_\rho(\gamma) \leq L\}}{L} \leq  \limsup\limits_{L\to \infty} \displaystyle\frac{\lambda \cdot \log \#\{[\gamma] \ \vert\ \ell_\sigma(\gamma) \leq \lambda L\}}{\lambda L}  = \lambda H(\sigma).
\end{equation*}
\end{proof}

Our first observation is that for the strict entropy inequality to hold, it is enough to know strict domination for \textit{most} curves -- see condition \eqref{eq:stat} below, which we shall refer to as  \textit{statistical} domination:

\begin{lem}[Statistical domination]\label{lem:stat} Let $\rho,\sigma:\pi_1(S)\to\pslnc$ be two representations, and suppose there exist constants $0<\alpha<1$ and $0<\lambda<1$ such that, for all sufficiently large $L$, the ratio
\begin{equation}\label{eq:stat} 
 \displaystyle\frac{\#\{[\gamma] \text{ a conjugacy class in } \pi_1(S) \ \vert\ \ell_\rho(\gamma) \leq L \text{ and } \ell_\sigma(\gamma) \leq \lambda \ell_\rho(\gamma)\}}{\#\{[\gamma] \ \vert\ \ell_\rho(\gamma) \leq L \}} \geq \alpha.
 \end{equation}
 Then we have the entropy inequality $H(\rho) \leq \lambda H(\sigma)$. 
 
\end{lem}

\begin{proof} The hypothesis easily implies the following variant of \eqref{eq:sub}:
\begin{equation}\label{eq:sub2} 
\alpha \cdot  \#\{[\gamma]\  \vert\ \ell_\rho(\gamma) \leq L\} \leq \# \{[\gamma]\ \vert\ l_\sigma(\gamma) \leq \lambda L\}.
\end{equation}
Taking a log on the LHS, and using Definition \ref{defn:ent}, we then obtain: 
\begin{equation*}
H(\rho) = \limsup\limits_{L\to \infty} \displaystyle\frac{\log \alpha + \log \#\{[\gamma] \ \vert\ \ell_\rho(\gamma) \leq L\}}{L}   \leq  \limsup\limits_{L\to \infty} \displaystyle\frac{\lambda \cdot \log \#\{[\gamma] \ \vert\ l_\sigma(\gamma) \leq \lambda L\}}{\lambda L}  = \lambda H(\sigma).
\end{equation*}
\end{proof}

\subsection{Entropy of quasi-Fuchsian representations}

Before we consider higher-rank representations, we prove here
Theorem \ref{thm:stat} and Corollary \ref{thm:ent} for quasi-Fuchsian representations, which are precisely the Anosov representations from $\pi_1(S)$ to $\pslc$.  
This case admits a geometric picture for statistical domination, which we shall describe.

\begin{defn}\label{defn:pleated}
Let $X$ be a marked hyperbolic structure on $S$, let
$\rho_0:\pi_1(S)\to\pslr$ be its holonomy, and identify
$\widetilde X$ with $\HH^2$. Given a representation
$\rho:\pi_1(S)\to\pslc$, a \emph{$\rho$-equivariant pleated plane}
is a continuous map  $\Psi:\widetilde X\longrightarrow \HH^3$
such that
\[
        \Psi\bigl(\gamma \cdot x\bigr)
        =\rho(\gamma)\Psi(x)
        \qquad
        \text{for every $\gamma\in\pi_1(S)$ and $x\in\widetilde X$,}
\]
and with the following properties:

(i) There is a geodesic lamination $\lambda$ on $X$ such that $\Psi$ is a totally-geodesic embedding on the closure of every complementary component of its lift $\tilde{\lambda} \subset \tilde{X}$. Such a component $P$ is called a \textit{plaque} and its $\Psi$-image lies on a totally-geodesic copy of the hyperbolic plane $H_P \subset \HH^3$, called the \emph{support plane} of $P$. The smallest such geodesic lamination $\lambda$ (when ordered by inclusion) is the \emph{pleating locus}.

(ii) $\Psi$ is path-isometric: for every rectifiable arc $a\subset X$, with lift $\tilde{a} \subset \widetilde{X}$, the parametrized  image  $\Psi(\tilde{a})$  has length exactly $\ell_{X}(a)$,  the length on the hyperbolic surface. 

(iii) The pleating locus is equipped with a transverse measure recording
the bending angles. More precisely, if a leaf $\ell\subset \widetilde\lambda$, of weight
$\theta\in(0,\pi]$ separates two plaques $P_-$ and $P_+$, the support plane $H_{P_+}$ is obtained from
$H_{P_-}$ by an elliptic rotation about the geodesic
$\Psi(\ell)$ through angle $\theta$; the direction of rotation is
determined by the chosen orientations.

More generally, let $a$ be an oriented arc transverse to
$\widetilde\lambda$, with endpoints in plaques $P$ and $Q$.
Approximating the measured lamination in a neighborhood of $a$ by
finitely many weighted leaves, one may compose, in their order along
$a$, the rotations through the corresponding bending angles about
the successive images of their leaves. The convergence theorem for
quakebend cocycles (see
\cite[\S\S~3.4--3.6]{EpsteinMarden}, and also
\cite[\S~8]{BonShear}) implies that these ordered products converge in
$\pslc$ to an isometry  $B_{\widetilde\lambda}(P,Q)\in\pslc$. 
The limit is independent of the finite approximation and satisfies
\begin{equation*} 
    B_{\widetilde\lambda}(P,R)
        =
                B_{\widetilde\lambda}(P,Q) 
        B_{\widetilde\lambda}(Q,R)
\end{equation*} 
where the composition is from left to right. This $B_{\widetilde\lambda}$ is the
\emph{bending cocycle} (see also \cite[\S5.3]{Dumas}). 
The rotation-angle component of the bending defines a transverse measure  on $\widetilde\lambda$; from the equivariance it descends to a measured
geodesic lamination $\lambda$ on $X$ that we call the \textit{bending lamination}. Moreover, 
the \emph{straightening} of $\Psi$ is obtained by replacing the bending cocycle by the trivial cocycle; its intrinsic hyperbolic structure is $X$.
\end{defn}

We had alluded to a pleated plane, as described above, in \S2.1; in particular, the representation $\rho$ above is a $\pslc$-representation pleated along $\lambda$, and the bending co-cycle described above can be thought of as the imaginary part of the $\mathbb{C}/2\pi i \mathbb{Z}$-valued cocycle in Theorem \ref{thm:bon}.  We now record the fact that there is a large class of such $\pslc$-representations, obtained by pleating along \textit{some} lamination $\lambda$; this is an immediate consequence of the monodromy theorem of Gallo-Kapovich-Marden and Thurston's description of complex projective structures: 

\begin{lem}\label{lem:gkm}
If $\rho:\pi_1(S)\to\pslc$ is a (not necessarily discrete) non-elementary representation that
lifts to $\slc$, then there are a hyperbolic structure $X$ on $S$
and a $\rho$-equivariant pleated plane $\Psi:\widetilde X\longrightarrow\HH^3$ that defines a bending lamination
$\lambda$ on $X$. Consequently, $\rho$ lies in the bending fiber of the Fuchsian representation $\rho_0$ corresponding to $X$. 
\end{lem}

\begin{proof}
By the result of Gallo-Kapovich-Marden
\cite[Theorem~1]{GKM}, there is a complex projective structure
$Z$ on $S$ whose holonomy is $\rho$. Thurston's grafting theorem
(see, for example, \cite{Dumas}) associates to $Z$ a unique pair
\[
        (X,\lambda)
        \in\mathcal T(S)\times\mathcal{ML}(S)
\]
such that $Z=\operatorname{Gr}_\lambda(X)$ where $\operatorname{Gr}_\lambda$ is the grafting map. 

Passing to the universal cover, the grafting construction has a
dual bending description. For a finite weighted lamination, one
starts with a totally-geodesic copy of $\widetilde X=\HH^2$ in
$\HH^3$, cuts along the lifts of the weighted leaves, and rotates
successive plaques about the corresponding geodesic lines through
the prescribed angles. For a general measured lamination, the map
is obtained as the locally uniform limit of the maps associated
with finite weighted approximations. The resulting map $\Psi:\widetilde X\longrightarrow\HH^3$ 
is continuous and path-isometric, is totally geodesic on the complementary 
plaques of $
\widetilde\lambda$, and is equivariant for the
holonomy of the grafted projective structure. Since the latter
holonomy is $\rho$, the map $\Psi$ is $\rho$-equivariant.

Equivalently, one may obtain the same pleated plane from the
maximal round disks in the image of the developing map of the projective structure. Each such
round disk has an ideal boundary circle in $\cp$, and hence
determines a totally-geodesic plane in $\HH^3$; the envelope of
these planes is the image of $\Psi$. The changes of the support planes determine the bending cocycle and its
transverse measure. After replacing the atomic weights in
$\lambda_{\mathrm{gr}}$ by their effective geometric bending
angles, and omitting those leaves on which the effective angle is
zero, this measure descends to the required measured geodesic
lamination $\lambda$ on $X$.
\end{proof}

\noindent  \textit{Remark.} When $\rho$ is quasi-Fuchsian as in this section, the above ``bending" description is easier; in that case the limit set (the image of the limit map) is a quasi-circle on $\mathbb{C}P^1$, and its convex hull in $\mathbb{H}^3$ is bounded by two $\rho$-equivariant pleated planes. 

\medskip

To prove statistical domination, we shall need the following
consequence of the geometry of the pleated plane. In the next two lemmas we assume, without loss of generality, that any closed leaf of the bending lamination $\lambda$ has weight in $(0,2\pi)$; a weight can be considered modulo $2\pi$ since bending by $2\pi$ does not affect the representation or the pleated plane. 

\begin{lem}\label{lem:defect0}
In the above set-up, let $\gamma$ be a closed geodesic on
$X$ that transversely intersects the bending lamination  $\lambda$, with positive transverse measure. Then
\begin{equation*}
        \ell_\rho(\gamma)\leq l_X(\gamma)-\delta
\end{equation*} 
for some $\delta>0$. Here $\delta$ may depend on
$\gamma$, $X$, and the bending data $\lambda$.
\end{lem}

\begin{proof}

Let
$a=[u,v]\subset\widetilde X$ be a geodesic segment transverse to
$\widetilde\lambda$ with positive transverse measure. We claim that
\begin{equation}\label{eq:strict-chord}
        d_{\HH^3}\bigl(\Psi(u),\Psi(v)\bigr)
        <
        d_{\HH^2}(u,v).
\end{equation}

Indeed, since $\Psi$ is path-isometric, the parametrized curve
$c=\Psi|_a$ has length $d_{\HH^2}(u,v)$. 
It therefore suffices to show that $c$ is not a geodesic segment
in $\HH^3$.

Suppose first that $a$ crosses an atomic leaf of bending
weight $w\in(0,2\pi)$. Set
$\theta=\min\{w,2\pi-w\}\in(0,\pi]$, its effective bending
angle, and let $\alpha\in(0,\pi)$ be the crossing angle.
Let $\beta$ be
the interior angle of the bent path at the crossing point, and let $  \varphi=\pi-\beta$ 
be its turning angle. The calculation in
\cite[Lemma~3.2]{GSu} gives
\[
        1+\cos\beta
        =
        2\sin^2(\alpha)\sin^2\left(\frac{\theta}{2}\right).
\]
Since
\[
        1+\cos\beta
        =
        1-\cos\varphi
        =
        2\sin^2\left(\frac{\varphi}{2}\right),
\]
we can also write 
\begin{equation}\label{eq:turning-angle}
        \sin\left(\frac{\varphi}{2}\right)
        =
        \sin(\alpha)
        \sin\left(\frac{\theta}{2}\right).
\end{equation}
Thus every transverse crossing of a leaf with non-zero 
bending angle contributes a non-zero turning angle.

Suppose now that the transverse bending measure on $a$
has no atoms. Choose
$z\in\operatorname{int}(a)\cap
\operatorname{supp}(\widetilde\lambda)$, and let $g$ be
the leaf through $z$. By the local convexity of the pleated plane (see, for example, \cite[\S4.2]{Dumas}), there is a neighborhood $U$ of $z$ such that $\Psi|_U$ is an embedding
onto an open patch of the boundary of a convex set $K\subset\HH^3$.

Choose a short subarc $b\subset a\cap U$ and a short
segment $g_0\subset g\cap U$, both containing $z$ in their
interiors. Suppose, for a contradiction, that $\Psi(b)$
is a geodesic segment. Let $H$ be a supporting plane of
$K$ at $p=\Psi(z)$. Both $\Psi(b)$ and $\Psi(g_0)$ lie
in the closed half-space bounded by $H$ that contains
$K$, and meet $H$ at an interior point. They must therefore
lie in $H$: a geodesic not contained in a totally
geodesic plane crosses that plane transversely. These two geodesics are distinct, since $\Psi|_U$ is injective and $a$ crosses $g$ transversely. Thus, their convex hull contains a neighborhood of $p$ in $H$.
Since this convex hull lies in the face $K\cap H$,
the boundary of $K$ is totally geodesic near $p$.
It follows that $\Psi$ is totally geodesic near $z$, that contradicts the assumption that $z\in\operatorname{supp}(\widetilde\lambda)$. Thus $\Psi(b)$, and hence $c=\Psi|_a$, is not a geodesic
segment.

Since a rectifiable curve in $\HH^3$
realizes the distance between its endpoints only when it is the
geodesic segment between them, the chord length between the endpoints of $c$ is strictly shorter than its
length, proving  \eqref{eq:strict-chord}. Choose $x$ on the axis of $\gamma$ so that the segment
$[x,\gamma x]$ has positive bending measure in its
interior. Applying \eqref{eq:strict-chord} to this segment
and using equivariance, we obtain
\[
 \ell_\rho(\gamma)
 \leq d_{\HH^3}\bigl(\Psi(x),\rho(\gamma)\Psi(x)\bigr)
 < d_{\HH^2}(x,\gamma x)
 = l_X(\gamma).
\]
This proves the assertion with 
$\delta
 :=
 l_X(\gamma)
 -
 d_{\HH^3}\bigl(\Psi(x),\rho(\gamma)\Psi(x)\bigr)>0
$. \end{proof}

Refining the proof above, we obtain a defect function that records
the loss of distance between compact pieces of plaques. Fix a complete train track $\tau$ carrying the lamination $\lambda$, and
let $N(\tau)$ be a sufficiently small open neighborhood of
$\tau$. One can enlarge an arbitrary train track carrying
$\lambda$ to a maximal one (where the complementary regions are all ideal triangles), by assigning zero transverse weight to
the additional branches. Let $\widetilde N(\tau)$ be the full lift
of $N(\tau)$ to $\HH^2$. A \emph{truncated plaque} is a connected
component of the complement $\HH^2\setminus\widetilde N(\tau)$; each is 
a compact set contained in a unique
plaque of $\HH^2\setminus\widetilde\lambda$. We denote the
collection of truncated plaques by $\mathcal P$.

We shall say that a truncated plaque $B$ lies \emph{between}
truncated plaques $A$ and $C$ if
\[
        [x,z]\cap B\neq\varnothing
        \qquad
        \text{for every $x\in A$ and $z\in C$,}
\]
where $[x,z]$ denotes the geodesic segment in $\HH^2$ joining
$x$ to $z$.

\begin{lem}\label{lem:defect}
There is a ``defect" function
\[
        \delta:\mathcal P\times\mathcal P
        \longrightarrow\mathbb R_{\geq 0}
\]
with the following property. For $A,B\in\mathcal P$, $x\in A$,
and $y\in B$, if $x'=\Psi(x)$ and $y'=\Psi(y)$ then
\begin{equation}\label{eq:plaque-defect}
        d_{\HH^3}(x',y')
        \leq
        d_{\HH^2}(x,y)-\delta(A,B).
\end{equation}
Moreover:
\begin{itemize}
\item[(i)] $\delta(A,A)=0$ for every $A\in\mathcal P$;

\item[(ii)] if $B$ lies between $A$ and $C$, then
\[
        \delta(A,C)
        \geq
        \delta(A,B)+\delta(B,C);
\]

\item[(iii)] for every $\gamma\in\pi_1(S)$,
\[
        \delta(\gamma\cdot A,\gamma\cdot B)
        =
        \delta(A,B).
\]
\end{itemize}
In addition, $\delta(A,B)>0$ whenever $A$ and $B$ are contained
in distinct plaques of
$\HH^2\setminus\widetilde\lambda$. If they are contained in the
same plaque, then $\delta(A,B)=0$.
\end{lem}

\begin{proof}
For $u,v\in\HH^2$, define the distance defect
\[
        \mathfrak d(u,v)
        :=
        d_{\HH^2}(u,v)
        -
        d_{\HH^3}\bigl(\Psi(u),\Psi(v)\bigr).
\]
Since $\Psi$ is path-isometric, the image under $\Psi$ of the
geodesic segment $[u,v]$ has length $d_{\HH^2}(u,v)$. The distance
between its endpoints is therefore no greater than its length, and
hence
\[
        \mathfrak d(u,v)\geq 0.
\]
The function $\mathfrak d:\HH^2\times\HH^2\to\mathbb R_{\geq0}$
is continuous. Since every truncated plaque is compact, we may
therefore define
\begin{equation}\label{eq:def-delta}
        \delta(A,B)
        :=
        \min_{\substack{u\in A\\ v\in B}}
        \mathfrak d(u,v).
\end{equation}
The inequality \eqref{eq:plaque-defect} follows immediately from
this definition.

Suppose that $A$ and $B$ are contained in distinct plaques of
$\HH^2\setminus \widetilde\lambda$. Then, for every $u\in A$ and
$v\in B$, the segment $[u,v]$ meets the effective bending
lamination with positive transverse measure. By the
argument in the proof of Lemma~\ref{lem:defect0}, whose
single-leaf case is \cite[Lemma~3.2]{GSu}, we have
\[
        d_{\HH^3}\bigl(\Psi(u),\Psi(v)\bigr)
        <
        d_{\HH^2}(u,v).
\]
Thus $\mathfrak d$ is strictly positive on the compact set
$A\times B$, and consequently  $\delta(A,B)>0$. 
On the other hand, if $A$ and $B$ are contained in the same
plaque, then $\Psi$ restricts to an isometric totally-geodesic
embedding on the plaque containing them. Hence
\[
        \mathfrak d(u,v)=0
        \qquad
        \text{for every $u\in A$ and $v\in B$,}
\]
and therefore $\delta(A,B)=0$. In particular, this proves
property~(i).

We next prove property~(ii). Suppose that $B$ lies between $A$ and
$C$. Let $u\in A$ and $w\in C$ be arbitrary. By the definition of
betweenness, there is a point
\[
        v\in [u,w]\cap B.
\]
Since $v$ lies on the geodesic segment from $u$ to $w$, we have
\[
        d_{\HH^2}(u,w)
        =
        d_{\HH^2}(u,v)+d_{\HH^2}(v,w).
\]
The triangle inequality in $\HH^3$ gives
\[
        d_{\HH^3}\bigl(\Psi(u),\Psi(w)\bigr)
        \leq
        d_{\HH^3}\bigl(\Psi(u),\Psi(v)\bigr)
        +
        d_{\HH^3}\bigl(\Psi(v),\Psi(w)\bigr).
\]
Subtracting the latter inequality from the former equality, we
obtain
\begin{equation*}
    \mathfrak d(u,w) \geq
        \mathfrak d(u,v)+\mathfrak d(v,w)\geq
        \delta(A,B)+\delta(B,C)
\end{equation*}        
where the last inequality follows from
\eqref{eq:def-delta}. Since $u\in A$ and $w\in C$ were arbitrary,
taking the minimum over $A\times C$ yields
\[
        \delta(A,C)
        \geq
        \delta(A,B)+\delta(B,C).
\]

Finally, let $\gamma\in\pi_1(S)$. Since the deck action on
$\HH^2$ is isometric and $\Psi$ is $\rho$-equivariant, for every
$u,v\in\HH^2$ we have
\begin{align*}
        \mathfrak d(\gamma\cdot u,\gamma\cdot v)
        &=
        d_{\HH^2}(\gamma\cdot u,\gamma\cdot v)
        -
        d_{\HH^3}\bigl(
            \Psi(\gamma\cdot u),
            \Psi(\gamma\cdot v)
        \bigr)\\
        &=
        d_{\HH^2}(u,v)
        -
        d_{\HH^3}\bigl(
            \rho(\gamma)\cdot \Psi(u),
            \rho(\gamma)\cdot \Psi(v)
        \bigr)\\
        &=
        d_{\HH^2}(u,v)
        -
        d_{\HH^3}\bigl(\Psi(u),\Psi(v)\bigr)\\
        &=
        \mathfrak d(u,v).
\end{align*}
Taking infimum on both sides proves property~(iii).
\end{proof}

We shall also need the following key observation, arising from the
Anosov property of the geodesic flow on $X$ and the equidistribution of closed geodesics (\cite{Bowen-eq}). 

Let $p:\HH^2\to X$ be the universal covering map. Fix truncated
plaques $A,B\in\mathcal P$, and denote their images in $X$ by
$\overline A=p(A)$ and $\overline B=p(B)$. An oriented geodesic
segment $I\subset X$ will be called an \emph{$(A,B)$-traversal} if
it admits a lift $\widetilde I\subset\HH^2$ whose initial endpoint
lies in $g\cdot A$ and whose terminal endpoint lies in $g\cdot B$,
for some $g\in\pi_1(S)$. By the equivariance property in
Lemma~\ref{lem:defect}, every such traversal has distance defect at
least $\delta(A,B)$. For a closed geodesic
$\gamma$, let $N_{A,B}(\gamma)$ denote the maximal number of
pairwise interior-disjoint $(A,B)$-traversals contained in
$\gamma$.

\begin{lem}\label{lem:lall}
For $L>0$, let $\mathcal S(L)$ be the set of closed geodesics on $X$ of length at most $L$. There exists $0< \beta < 1$  such that
\begin{equation}\label{eq:prop}
 \lim_{L\to\infty}
 \frac{
   \left|
     \left\{
       \gamma\in\mathcal S(L):
       N_{A,B}(\gamma)\geq\beta\,\ell_X(\gamma)
     \right\}
   \right|
 }{
   |\mathcal S(L)|
 }
 =1.
\end{equation}
\end{lem}

In other words, for any $\epsilon>0$, for all sufficiently large $L$ an  $(1-\epsilon)$-proportion of the closed
geodesics of length at most $L$ contain a number of pairwise
disjoint traversals from $\overline A$ to $\overline B$ which is
bounded below by a $\beta$-fraction of their own length.

\begin{proof}
We first count oriented closed geodesics; the unoriented
statement follows by forgetting orientation. Let $\phi_t:T^1X\to T^1X$ denote the geodesic flow. Choose
points $x\in\operatorname{int}(A)$ and
$y\in\operatorname{int}(B)$, and let
$\widetilde\sigma:[0,T]\longrightarrow\HH^2$
be the oriented geodesic segment from $x$ to $y$, parametrized
by arclength. Let $v\in T^1X$ be the projection of its initial
unit tangent vector.

By continuous dependence of geodesics on their initial
conditions, there is a small open neighborhood
$U\subset T^1X$ of $v$ such that, for every $w\in U$, the
geodesic segment
\[
        t\longmapsto\pi\bigl(\phi_t(w)\bigr),
        \qquad 0\leq t\leq T,
\]
is an $(A,B)$-traversal. Here $\pi:T^1X\to X$ is the usual
projection. Indeed, after lifting the initial point to a
sufficiently small neighborhood of $x$, the endpoint at time
$T$ lies in a neighborhood of $y$ contained in
$\operatorname{int}(B)$.

Let $\mu$ be the normalized Liouville measure on $T^1X$.
Choose a continuous function
$f:T^1X\longrightarrow[0,1]$ which is not identically zero
and whose support is contained in $U$. Since $\mu$ has full
support,
\[
        m:=\int_{T^1X}f\,d\mu>0.
\]
For each oriented closed geodesic $\gamma$, let $\nu_\gamma$
be the invariant probability measure supported on the
corresponding closed orbit:
\[
        \int_{T^1X}F\,d\nu_\gamma
        =
        \frac{1}{\ell_X(\gamma)}
        \int_0^{\ell_X(\gamma)}
        F\bigl(\phi_t(v_\gamma)\bigr)\,dt,
\]
where $v_\gamma$ is a unit tangent vector pointing in the
chosen orientation of $\gamma$. The equidistribution of
closed geodesics, by the main result of \cite{Bowen-eq},
then implies
\begin{equation}\label{eq:bowen-average}
        \frac{1}{|\mathcal S(L)|}
        \sum_{\gamma\in\mathcal S(L)}
        \nu_\gamma
        \xrightarrow{\;\ast\;}
        \mu
        \qquad\text{as }L\to\infty.
\end{equation}
Set
\[
        X_\gamma:=\int_{T^1X}f\,d\nu_\gamma\in[0,1].
\]

\medskip 

\noindent \textit{Claim.} \textit{The deviation from mean satisfies}
\begin{equation}\label{eq:mean-deviation}
        \frac{1}{|\mathcal S(L)|}
        \sum_{\gamma\in\mathcal S(L)}
        |X_\gamma-m|
        \longrightarrow0
\end{equation}
\textit{as $L\to \infty$.}

\medskip 

\noindent \textit{Proof of claim.} For $R>0$, define the continuous function
\[
        f_R(w):=\frac1R\int_0^R f\bigl(\phi_t(w)\bigr)\,dt.
\]
Since $\nu_\gamma$ is invariant under the geodesic flow,
\[
        \int_{T^1X}f_R\,d\nu_\gamma=X_\gamma.
\]
Consequently,
\[
        |X_\gamma-m|
        \leq
        \int_{T^1X}|f_R-m|\,d\nu_\gamma.
\]
Averaging over $\mathcal S(L)$ and applying
\eqref{eq:bowen-average}, with $R$ fixed, gives
\[
        \limsup_{L\to\infty}
        \frac{1}{|\mathcal S(L)|}
        \sum_{\gamma\in\mathcal S(L)}|X_\gamma-m|
        \leq
        \int_{T^1X}|f_R-m|\,d\mu.
\]
The Liouville measure $\mu$ is ergodic for the geodesic flow.
Thus, by the ergodic theorem, $f_R\to m$ almost everywhere
with respect to $\mu$ as $R\to\infty$. Since $0\leq f_R\leq1$,
dominated convergence implies
\[
        \int_{T^1X}|f_R-m|\,d\mu\longrightarrow0
\]
as $R\to\infty$, proving \eqref{eq:mean-deviation}. $\qed$

\medskip

Let $p_L$ be the proportion of geodesics
$\gamma\in\mathcal S(L)$ for which $X_\gamma\geq m/2$.
For every geodesic not satisfying this inequality,
$|X_\gamma-m|\geq m/2$. Hence
\[
        0\leq\frac{m}{2}(1-p_L) 
        \leq
        \frac{1}{|\mathcal S(L)|}
        \sum_{\gamma\in\mathcal S(L)}|X_\gamma-m|.
\]
By the previous claim  we then conclude that $p_L \to 1$ as $L\to \infty$.
There are only finitely many closed geodesics of length
at most $4T$. Since $|\mathcal S(L)|\to\infty$, it follows
that the proportion of geodesics in $\mathcal S(L)$ satisfying
both
\begin{equation}\label{eq:good-orbit}
        X_\gamma\geq\frac{m}{2}
        \qquad\text{and}\qquad
        \ell_X(\gamma)>4T
\end{equation}
tends to $1$ as $L\to\infty$.

Fix such a geodesic $\gamma$, and let $\ell=\ell_X(\gamma)$.
Consider the set of times
\[
        E_\gamma
        =
        \left\{
          t\in\mathbb R/\ell\mathbb Z:
          \phi_t(v_\gamma)\in U
        \right\}.
\]
Since $0\leq f\leq1$ and $\operatorname{supp}(f)\subset U$,
\eqref{eq:good-orbit} gives
\begin{equation}\label{eq:measure-E}
        |E_\gamma|
        \geq
        \int_0^\ell
        f\bigl(\phi_t(v_\gamma)\bigr)\,dt
        =
        X_\gamma\,\ell
        \geq
        \frac{m}{2}\ell.
\end{equation}

Choose a maximal subset $\{t_1,\ldots,t_N\}\subset E_\gamma$
whose elements have pairwise circular distance at least $2T$.
By maximality, the circular intervals of radius $2T$ centered
at the $t_j$ cover $E_\gamma$. Hence
\[
        |E_\gamma|\leq4TN,
\]
and therefore, by \eqref{eq:measure-E},
\begin{equation}\label{eq:number-traversals}
        N\geq
        \frac{|E_\gamma|}{4T}
        \geq
        \frac{m}{8T}\ell.
\end{equation}
For each $j$, the orbit segment
\[
        t\longmapsto
        \pi\bigl(\phi_{t_j+t}(v_\gamma)\bigr),
        \qquad 0\leq t\leq T,
\]
is an $(A,B)$-traversal. The separation of the times $t_j$
ensures that these traversals have pairwise disjoint interiors
along the parametrized closed geodesic. Defining
\[
        \beta
        :=
        \min\left\{
          \frac{m}{8T},\frac12
        \right\}\in(0,1),
\]
we conclude that
\[
        N_{A,B}(\gamma)\geq\beta\,\ell_X(\gamma)
\]
whenever \eqref{eq:good-orbit} holds. Since the proportion
of such geodesics tends to $1$, this proves
\eqref{eq:prop} for oriented closed geodesics.

Finally, the number of unoriented geodesics with at least
one exceptional orientation is at most the number of
exceptional oriented geodesics. Since forgetting orientation
halves the total count, the exceptional proportion still
tends to zero; this proves the unoriented statement as well.
\end{proof}

We can now conclude the $\pslc$-case of Theorem \ref{thm:stat} and Corollary \ref{thm:ent}, namely the statistical domination theorem for $\pslc$-representations pleated along $\lambda$ (\textit{cf.} Lemma \ref{lem:gkm}) and the entropy rigidity for quasi-Fuchsian representations. 
 Note that a stronger statement than the statistical domination in Theorem \ref{thm:stat} is true, since for any non-Fuchsian representation there is a strict domination for \textit{all} curves by work of Deroin-Tholozan -- see \cite[Corollary B]{DT}, and the entropy rigidity is Bowen's result in \cite{Bowen79} mentioned in the Introduction. Our approach here provides an alternative proof of the latter result, independent of the work of Deroin-Tholozan.

Recall from the Introduction that in this case the definitions of Hilbert length and translation length coincide; for quasi-Fuchsian representations, the entropy (which in Definition \ref{defn:ent}  is known to be a limit -- see, for example, \cite{Roblin}) coincides with the Hausdorff dimension of the limit set.

\begin{proof}[Proof of Theorem \ref{thm:stat} (Statistical Domination) for $\pslc$]
Let $\rho:\pi_1S \to \pslc$ be a representation obtained by pleating a Fuchsian representation $\rho_0$  along $\lambda$. 
With the normalization used here, the Hilbert length
$\ell^H_\rho(\gamma)$ is the translation length of
$\rho(\gamma)$ in $\HH^3$. In other words, if the hyperbolic surface corresponding to $\rho_0$ is $X$, then $\ell^H_{\rho_0}(\gamma)=\ell_X(\gamma)$, the hyperbolic length of the geodesic representative of $\gamma$ on $X$. 

Let $ \Psi:\HH^2\longrightarrow\HH^3$
be the $\rho$-equivariant pleated plane, and let $   X=\HH^2/\rho_0\bigl(\pi_1(S)\bigr)$ 
be its straightening. Since $\Psi$ is path-isometric (see (ii) of Definition \ref{defn:pleated}), it is
$1$-Lipschitz with respect to the distances between endpoints. If
$x$ lies on the axis in $\HH^2$ of a nontrivial
$\gamma\in\pi_1(S)$, we have 
\begin{align*}
        \ell_\rho(\gamma)
        &\leq
        d_{\HH^3}\bigl(\Psi(x),\rho(\gamma)\cdot \Psi(x)\bigr)\\
        &=
        d_{\HH^3}\bigl(\Psi(x),\Psi(\gamma\cdot x)\bigr)\\
        &\leq
        d_{\HH^2}(x,\gamma\cdot x)
        =
        \ell_X(\gamma)
        =
        \ell_{\rho_0}(\gamma).
\end{align*}
Note that by Lemma~\ref{lem:ent-dom} we then have $H(\rho)\geq H(\rho_0)$ if $\rho$ is quasi-Fuchsian, in which case $H(\rho)$ is well-defined (and finite).

When $\rho$ is not Fuchsian, the bending along $\lambda$ is non-trivial, and  Lemma \ref{lem:defect} applies. 
Choose truncated plaques $A,B\in\mathcal P$
which are contained in distinct plaques of
$\HH^2\setminus\widetilde\lambda$. By
Lemma~\ref{lem:defect}, the defect $ \delta_0:=\delta(A,B)>0$. 

Applying Lemma~\ref{lem:lall} to this pair, there are constants
$\alpha, \beta>0$ such that, for all sufficiently large $L$, at least a
proportion $\alpha$ of the geodesics $\gamma\in\mathcal S(L)$ contain
at least $\beta\ell_X(\gamma)$ pairwise interior-disjoint
$(A,B)$-traversals. (In fact, by \eqref{eq:prop} the proportion $\alpha$ can be taken to be arbitrarily close to $1$.) 

We first record how the defects of these traversals add up when there are multiple traversals.
Let $\gamma$ be a closed geodesic containing $m := N_{A,B}(\gamma)>0$ pairwise
interior-disjoint $(A,B)$-traversals. Parametrize $\gamma$ by
arclength and lift one period to a segment
\[
        \widetilde\gamma:[0,\ell_X(\gamma)]\longrightarrow\HH^2
\]
such that
\[
        \widetilde\gamma\bigl(\ell_X(\gamma)\bigr)
        =
        \gamma\cdot\widetilde\gamma(0).
\]
After changing the initial point if necessary, the chosen
traversals lift to pairwise interior-disjoint subsegments of this
fundamental segment. On each such subsegment, Lemma
\ref{lem:defect} implies  a defect of at least $\delta_0$ between intrinsic length and the
distance between the images of its endpoints. On each of the
remaining sub-segments, path-isometry gives the corresponding
non-strict inequality. Applying the triangle inequality in
$\HH^3$ at the successive endpoints therefore yields
\begin{equation}\label{eq:accumulated-defect}
        d_{\HH^3}\bigl(
          \Psi(\widetilde\gamma(0)),
          \Psi(\widetilde\gamma(l_X(\gamma)))
        \bigr)
        \leq
        \ell_X(\gamma)-m\delta_0.
\end{equation}
Equivariance and the definition of translation length then imply
\begin{equation}\label{eq:closed-defect}
        \ell_\rho(\gamma)
        \leq
        \ell_X(\gamma)-m\delta_0.
\end{equation}
In particular, this refines the strict inequality supplied by
Lemma~\ref{lem:defect0}.

Since by Lemma \ref{lem:lall} we have $m\geq \beta l_X(\gamma)$ for ``almost-all" curves, and we then obtain for any such curve, 
\begin{equation} \label{eq:ineq}
        \ell_\rho(\gamma) \leq 
        \bigl(1-\beta \delta_0 \bigr)\ell_X(\gamma) = 
        \bigl(1-\beta \delta_0\bigr) \ell_{\rho_0}(\gamma).
\end{equation} 
Replacing $\beta$ by $\min \{\beta, \frac{1}{2\delta_0} \}$, we can ensure 
\begin{equation*}
        \nu:=1-\beta \delta_0\in(0,1).
\end{equation*}

From Lemma~\ref{lem:lall} and \eqref{eq:ineq} we then obtain 
\begin{equation*}      
    \lim\limits_{L\to \infty} \frac{
          \#\bigl\{
             \gamma\in\mathcal S(L)\  \vert\  
             \ell_\rho(\gamma)
             \leq\nu\ell_{\rho_0}(\gamma)
          \bigr\}
        }{
          \lvert\mathcal S(L)\rvert
        }
        =1
\end{equation*}
which is precisely the statistical domination property in Theorem \ref{thm:stat}. In particular, for all sufficiently large $L$, the inequality in  \eqref{eq:stat} in the statement of Lemma~\ref{lem:stat} holds for sufficiently large $L$,  with $\sigma:=\rho$ and $\rho := \rho_0$. By that Lemma, we then obtain $H(\rho)>H(\rho_0)$ which establishes Corollary \ref{thm:ent}.
\end{proof}

\noindent \textit{Remark.} In fact, for the one-parameter  family of representations $\rho_t:\pi_1(S)\to \pslc$ obtained by bending a Fuchsian representation $\rho_0$ along a measured geodesic lamination $t\lambda$,  the entropy $H(\rho_t)$ increases monotonically for $t$ in some interval $[0,t_0]$. This shall be studied in forthcoming work; see also the recent work of Bridgeman-Canary-Sambarino in \cite{BCS26} for related results.

\subsection{Punctured surface case: strict matrix domination}

In the case of surface-group representations into $\pslnc$ where $n>2$, we do not have the above geometric picture since there is no good analogue of a pleated plane in the associated symmetric space $\mathbb{X}_n$. However, the algebraic setup discussed in \S3  will be used to provide a  statistical domination result (Theorem \ref{thm:stat}). We start with the following algebraic notion of (strict) matrix domination:

\begin{defn}[Strict matrix domination] A $n\times n$ matrix $A$ with complex entries is said to be strictly dominated by an $n\times n$ matrix $B$ with real and positive entries if  there exists $0<\alpha <1$ (referred to as the \textit{dominating factor}) such that $\lvert a_{ij}\rvert \leq \alpha  \cdot b_{ij}$ for all $1\leq i,j\leq n$.  We shall denote this by $\lvert A \rvert < B$. 
\end{defn}

Consider the case when the surface $S$ has punctures; $\rho:\pi_1(S) \to \pslnc$ is a representation with Fock-Goncharov coordinates (with respect to some ideal triangulation) and $\rho_0:\pi_1(S) \to \pslnr$ is the positive representation obtained by replacing each such coordinate with its modulus. 

In what follows, the ``planar network associated with the matrix $M$" refers to the planar network whose weight matrix is $M$. In the case of a punctured surface $S$, when the monodromy matrix is a product of ``building block" matrices as in \eqref{decomp}, a planar network associated with it is constructed by concatenating the planar networks for building blocks obtained in \cite{BGup} (see \S2.5).

We shall use one  more structural result about these planar networks. For the next lemma, let $\mathbf t$ be the collection of triangle coordinates of one ideal triangle and let $\mathbf e=(e_1,\ldots,e_{n-1})$ be the edge coordinates on one chosen side.  

\begin{lem}\label{lem:coordinate-detecting}
Fix $c\in\mathbf t\cup\mathbf e$ and $1\leq r<n$.  In the planar network associated with the product of two building blocks
$$T(\mathbf t)E(\mathbf e)\,\cdot\,T^{-1}E$$
there are two vertex-disjoint path families $\mathcal F_c$ and $\mathcal F_c'$ with the same $r$ sources and the same $r$ sinks such that the ratio of weights $$\frac{\wt(\mathcal F_c)}{\wt(\mathcal F_c')}=c.$$
\end{lem}

\begin{proof}
If $c$ is a triangle parameter, choose the two paths in the $TE$ network which differ by a snake move (see Figure \ref{fig:triangle-ratios}).  They lie between two consecutive horizontal levels and have a common source and sink.  Continue both horizontally through the $T^{-1}E$ network.  Add the same $r-1$ horizontal paths on levels disjoint from these two levels.  The resulting two families have ratio $c$.

For an edge coordinate, the two distinguished paths in the $TE$ network leave the marked edge factor on adjacent levels (see Figure \ref{fig:edge-ratios}).  In the following $T^{-1}E$ network, use the first available rising path to bring them to a common vertex and then continue horizontally (see Figure \ref{fig:edge-ratio-example} for an example).  Adding the same $r-1$ horizontal paths on the remaining levels completes the two path families whose ratio is $c$ as desired. 
\end{proof}

\begin{figure}[ht]
\centering
\begin{subfigure}{.48\linewidth}
\centering
\includegraphics[width=\linewidth]{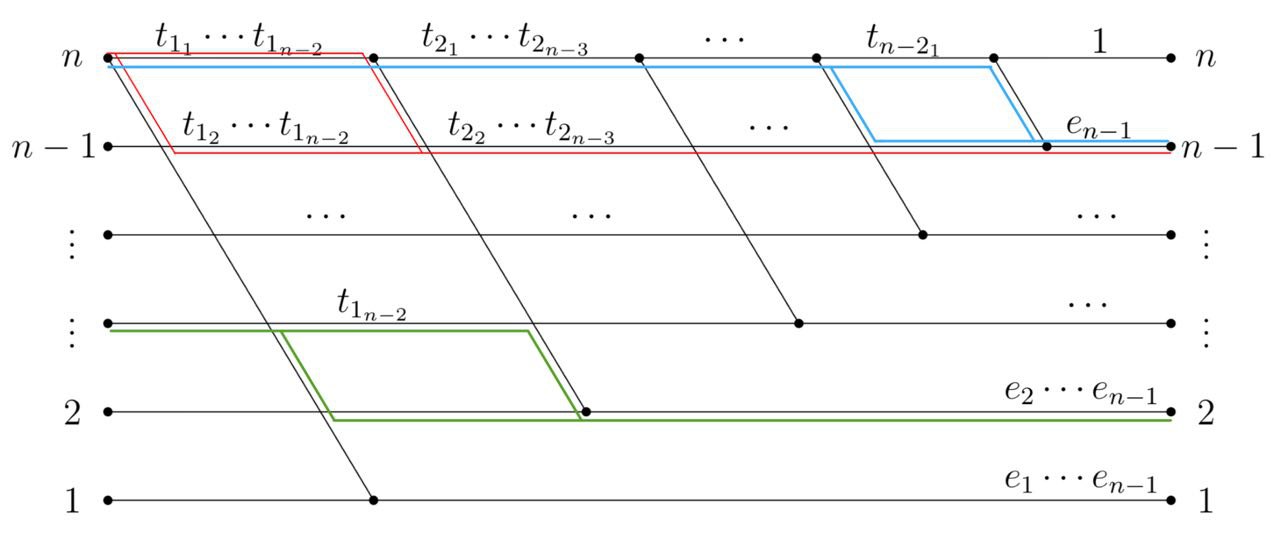}
\caption{Pairs detecting triangle coordinates.}
\label{fig:triangle-ratios}
\end{subfigure}\hfill
\begin{subfigure}{.48\linewidth}
\centering
\includegraphics[width=\linewidth]{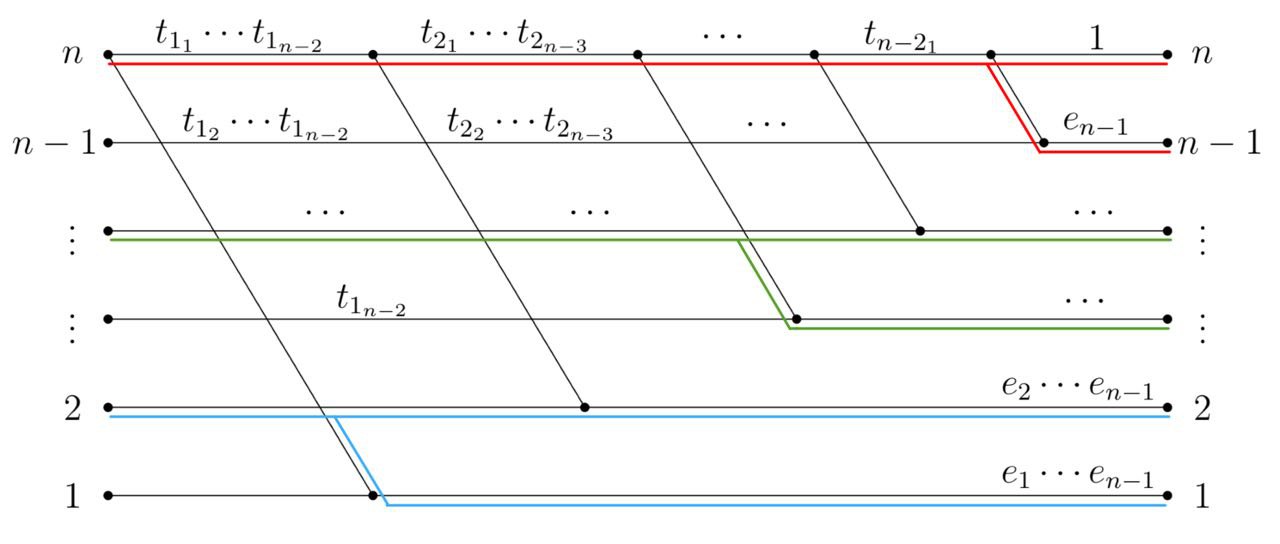}
\caption{Pairs detecting edge coordinates.}
\label{fig:edge-ratios}
\end{subfigure}
\caption{Coordinate-detecting paths in the $TE$ network.}
\end{figure}

\begin{figure}[ht]
\centering
\begin{tikzpicture}
    \foreach \y in {0, ..., 3} {\draw (0,\y) -- (4.8,\y);}

    \draw (0, 3) -- (1.5, 0) (1.5, 3) -- (2.5, 1) (3, 3) -- (3.5, 2);

    \foreach \y in {0, ..., 3} {\draw (5,\y) -- (9.8,\y);}

    \draw (5, 0) -- (6.5, 3) (6.5, 0) -- (7.5, 2) (8, 0) -- (8.5, 1);

    \draw[red] (0,1.9) -- (9.8, 1.9);
    \draw[red] (intersection of 1.4,3--2.4,1 and 0,0.9--9.8,0.9) -- (intersection of 4.9,0--6.4,3 and 0,0.9--9.8,0.9);
    \draw[red] (intersection of 1.4,3--2.4,1 and 0,1.9--9.8,1.9) -- (intersection of 1.4,3--2.4,1 and 0,0.9--9.8,0.9);
    \draw[red] (intersection of 4.9,0--6.4,3 and 0,1.9--9.8,1.9) -- (intersection of 4.9,0--6.4,3 and 0,0.9--9.8,0.9);

    \foreach \p/\t in {(-0.2,0)/$1$, (10,0)/$1$, (-0.2,1)/$2$, (10,1)/$2$, (-0.2,2)/$3$, (10,2)/$3$, (-0.2,3)/$4$, (10,3)/$4$, (2.2,3.2)/$t_1$, (.8,3.2)/$t_2t_3$, (1.2,2.2)/$t_3$, (3,0.2)/$e_1e_2e_3$, (3.6,1.2)/$e_2e_3$, (4.2,2.2)/$e_3$, (5.8,0.2)/$t_4$, (7.3,0.2)/$t_5t_6$, (7.8,1.2)/$t_6$, (9,0.2)/$e_4e_5e_6$, (9.2,1.2)/$e_5e_6$, (9.2,2.2)/$e_6$, (2.5,-0.6)/$TE$, (7.5,-0.6)/$T^{-1}E$}{\node at \p {\t};}
\end{tikzpicture}
\caption{For $n=4$, two paths with weight ratio $e_2$ become paths with a common source and sink after the following $T^{-1}E$ network.}
\label{fig:edge-ratio-example}
\end{figure}
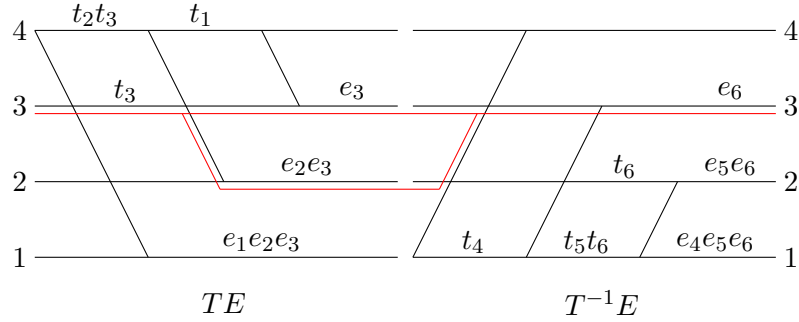

\noindent We shall in fact prove a stronger result about a strict inequality of all \textit{minors} of the monodromy matrices $A$ and $A_0$:

\begin{prop}[Strict matrix domination for a punctured surface]\label{prop:punctured-strict}
Assume that at least one Fock--Goncharov coordinate of $\rho$ is not real and positive.  Then there is $\gamma\in\pi_1(S)$ with determinant-one lifts $
 A,A_0\in\slnc$
of $\rho(\gamma)$ and $\rho_0(\gamma)$ respectively such that, for all $I,J\subset \{1,2,\ldots,n\}$ with
$1\leq |I|=|J|<n$, we have  the inequality of minors 
\begin{equation}\label{eq:minor-strict}
 \abs{\minor{I}{J}{A}}
 <\minor{I}{J}{A_0}.
 \end{equation} 
In particular, $|A|<A_0$ entrywise.
\end{prop}

\begin{proof}
Choose an ideal triangle together with one of its sides so that the associated collection  of parameters $\mathbf t\cup\mathbf e$ contains a coordinate $c\notin\R_{>0}$.  Choose $\gamma$ whose monodromy word contains, in this order, the three distinct building blocks
\begin{equation*} 
 (T^{-1}E)(TE),
 \qquad
 T(\mathbf t)E(\mathbf e)\,T^{-1}E,
 \qquad
 (T^{-1}E)(TE)
 \tag{\(*\)}\label{eq:strict-pattern}
 \end{equation*} 
which need not be consecutive.

Fix $I,J$ with $|I|=|J|=r<n$.  By  Lemma \ref{lem:coordinate-detecting}, the central piece in \eqref{eq:strict-pattern} contains two $r$-path families with common source and sink sets and weight ratio $c$.  Use Lemma \ref{lem:universal-routing} in the left routing piece to connect $I$ to their common source set, and in the right routing piece to connect their common sink set to $J$.  Extend horizontally through the other blocks; this uses the fact that planar networks associated with the building blocks have horizontal paths across each level (\textit{cf.} Lemma \ref{keylem}).  This gives two $I$-to-$J$ families with weights $u$ and $v$ satisfying $u/v=c\notin\R_{>0}$.

Let $w_1,\ldots,w_N$ be the path-weights in the $I$-to-$J$ family in the planar network associated with $A$.  By Lindstr\"{o}m's Lemma (see Lemma \ref{lem:lind}) we have
\begin{equation*} 
 \minor{I}{J}{A}=\sum_s w_s,
 \qquad
 \minor{I}{J}{A_0}=\sum_s|w_s|.
\end{equation*} 
Since the two weights $u,v$ are part of this collection of weights, but are not positive real multiples of each other, the triangle inequality is strict and we have \eqref{eq:minor-strict}. 
\end{proof}

\subsection{Closed surface case: strict matrix domination}

We shall use the generalized building blocks and their approximants, already introduced in \S3.3; in particular, we shall use the fact that their corresponding planar networks are \textit{transparent} (see Definition \ref{def:transparent}). One consequence of being transparent is:

\begin{lem}\label{lem:neutral-insertion}
Every path and path-family construction in Lemmas \ref{lem:path-to-top}, \ref{lem:universal-routing} and \ref{lem:coordinate-detecting} remains valid after inserting finitely many transparent networks between the planar networks associated with the $T$ and $E$-factors.  The source and sink sets are unchanged, vertex-disjointness is preserved, and every weight ratio is unchanged.
\end{lem}

\begin{proof}
At each inserted transparent network, continue each path along the distinguished horizontal path on its current level.  Recall that since they are transparent, these distinguished paths are pairwise vertex-disjoint and have weight $1$.  Two compared families use the same continuation on every level they occupy, so their weight ratio is unchanged.
\end{proof}

\noindent Another basic observation is: 

\begin{lem}\label{prop:strict-limit}
Let $(\Gamma_N,\omega_N)$ be finite planar networks whose weight matrices converge to $A$, and suppose the networks $(\Gamma_N,\lvert \omega_N\rvert )$ have weight matrices converging to $A_0$.  Fix $I,J$ with $|I|=|J|$.  Assume that every $\Gamma_N$ contains two distinguished $I$-to-$J$ path families with weights $u,v\in\C^*$ independent of $N$, such that $u/v\notin\R_{>0}$.  There exists \(0<q<1\), which isindependent of \(N\), such that
\[
 |\Delta_{I,J}(A_N)|\le q\,\Delta_{I,J}(A_{0,N})
\]
for every \(N \geq 1\); consequently the same inequality holds in the limit.
\end{lem}

\begin{proof}
The proof is an application of Lindstr\"{o}m's lemma (Lemma \ref{lem:lind}) that expresses a minor as the sum of weights of a path-family, and the triangle inequality. 
For each $N$, write the sum of all remaining weights of the $I$-$J$ path family in Lindstr\"{o}m's lemma as $r_N$ and the sum of their moduli as $R_N$, so that $R_N\geq|r_N|$. Let $A_N$ be the weight matrix of $(\Gamma_N, \omega_N)$ and $A_{0,N}$ be the weight matrix of $(\Gamma_N, \lvert \omega_N\rvert)$. Then we have:
\begin{equation*} 
 \minor{I}{J}{A_{0,N}}
 -\abs{\minor{I}{J}{A_N}}
= (\lvert u \rvert + \lvert v \rvert + R_N)-\lvert u+v+r_N\rvert
\geq |u|+|v|-|u+v| =:d
\end{equation*}
where note that the lower bound $d$ is independent of $N$. 
Since $A_{0,N} \to A_0$, we also have the upper bound 
$$\minor{I}{J}{A_{0,N}} \leq M$$ for all $N\geq 1$, and hence the ratio 
\begin{equation*} 
\frac{\abs{\minor{I}{J}{A_N}}}{\minor{I}{J}{A_{0,N}}} \leq 1 - \frac{d}{M} < 1
 \end{equation*} 
 where note that the upper bound is strictly less than $1$ and independent of $N$. Taking a supremum proves the strict domination $\abs{\minor{I}{J}{A_N}}\leq q \minor{I}{J}{A_{0,N}}$ with factor
$$q  := \sup\limits_{N\geq 1} \frac{\abs{\minor{I}{J}{A_N}}}{\minor{I}{J}{A_{0,N}}} <1.$$
Taking $N\to \infty$ we obtain the desired strict domination in the limit. 
\end{proof}

\noindent We shall need a closed curve whose coarse plaque itinerary contains the finite pattern from Proposition \ref{prop:punctured-strict}; we record the following basic fact that allows us to do this: 

\begin{lem}[Realizing coarse itinerary]\label{lem:closed-itinerary}
Let $K\subset\widetilde S$ be a compact geodesic segment transverse to $\widetilde\lambda$.  Prescribe finitely many additional left or right exits through plaques before and after $K$.  Then there is a nonempty open set $U\subset T^1S$ such that the lift of any closed geodesic determined by a point in $U$ realizes the prescribed finite coarse plaque itinerary. 
\end{lem}

\begin{proof}
For a geodesic entering an ideal triangle through a fixed side, the two possible exit sides correspond to two nonempty open intervals of endpoints at infinity, separated by the opposite ideal vertex.  Starting with the endpoints of the geodesic containing $K$, choose a finite succession of  nested open intervals of endpoint pairs which realize the prescribed finite sequence of exits; the resulting set of unit tangent vectors is nonempty and open. The geodesic flow on the unit tangent bundle of a closed hyperbolic surface is Anosov, and its periodic orbits are dense, and hence intersects $U$. A closed geodesic corresponding to such a periodic vector $(x,v) \in U$ has a lift containing the prescribed itinerary (\textit{cf.} \cite[Chapter~18]{KH95}).
\end{proof}

\noindent Recall that for a closed surface $S$, an $n$-pleated representation $\rho:\pi_1(S) \to \pslnc$ is represented by  $(\alpha_\rho, \theta_\rho) \in \Y(\lambda,n;\mathbb{C}/2\pi i\mathbb{Z})$  (see Theorem \ref{bon-gen}). Here we shall assume $n\geq 3$ since we have already handled the $n=2$ case separately in \S4.2. We observe:

\begin{lem}\label{lem:marked-coordinate}
If $(\alpha_\rho, \theta_\rho) \in \Y(\lambda,n;\mathbb{C}/2\pi i\mathbb{Z})$ is such that $(\operatorname{Im}\alpha_\rho,\operatorname{Im}\theta_\rho)\neq 0$
then either $t_{j}(x) \notin \mathbb{R}_{>0}$ 
for some labeled plaque $x$ and $j\in\cB$, or $ e_{ i}(P,Q)\notin \mathbb{R}_{>0}$ 
for some distinct plaques $P,Q$ and $i\in\cA$.  (Here recall that $\cA$ and $\cB$ are index-sets defined in \eqref{eq:anbn}.) 

\end{lem}

\begin{proof} Recall that the triangle and edge invariants are related to the components of the cocycle by \eqref{eq:exp-coord}. A class in $\R/2\pi\Z$ is zero exactly when its complex exponential is positive real; since $(\operatorname{Im}\alpha_\rho,\operatorname{Im}\theta_\rho)\neq 0$ one of the exponentiated coordinates must fail to be real and positive. 
\end{proof}

\noindent The closed-surface analogue of Proposition \ref{prop:punctured-strict} is then:

\begin{prop}[Strict matrix domination for a closed surface]\label{prop:closed-strict}
Let $S$ be closed, let $\lambda$ be an arbitrary maximal geodesic lamination on $S$, and let $\rho:\pi_1(S) \to \pslnc$ determined by $(\alpha_\rho, \theta_\rho) \in \Y(\lambda,n;\mathbb{C}/2\pi i\mathbb{Z})$ with $n\geq3$.  Let $\rho_0$ be the Hitchin representation determined by the real part of the cocycle as in \eqref{eq:rho0coords}.  If $(\operatorname{Im}\alpha_\rho,\operatorname{Im}\theta_\rho)\neq0$,
then there exist a nontrivial $\gamma\in\pi_1(S)$, and determinant-one lifts
$A,A_0\in\slnc$ of $\rho(\gamma), \rho_0(\gamma) \in \pslnc$ respectively, 
such that
\begin{equation}\label{eq:closed-strict-minors}
 \abs{\minor{I}{J}{A}}
 <\minor{I}{J}{A_0}
\end{equation}
for every $I,J\subset \{1,2,\ldots, n\}$ with $1\leq|I|=|J|<n$.  In particular $|A|<A_0$. 
\end{prop}

\begin{proof}
By  Lemma \ref{lem:marked-coordinate}, choose parameter $c\notin\R_{>0}$. 
If $c$ is a triangle parameter $t_{j}(x)$, choose the plaque of $x$ as the marked $T$-turn.  If $c$ is an edge invariant $e_{i}(P,Q)$, choose the coarse transition from $P$ to $Q$ as the marked $E$-change, and choose the following turn to be positive.  In either case, use  Lemma \ref{lem:closed-itinerary} to choose a closed geodesic whose coarse plaque itinerary  contains the pattern \eqref{eq:strict-pattern}, with the central marked $TE$ piece involving the parameter $c$.  Note that by \eqref{eq:coarse-monodromy} the determinant-one lift $A$ is a product of generalized building blocks, with slithering factors inserted between the $T^\pm E$ building blocks corresponding to the pattern \eqref{eq:strict-pattern}.

The closed geodesic $\gamma$ lifts to an arc between the starting plaque $P_0$ and the terminal plaque $\gamma P_0$.  By Corollary \ref{cor:slithering-networks} we can replace each exact slithering matrix by transparent approximants, to obtain an approximant of a generalized building block (see Corollary \ref{cor:approx-gbb}) that is the weight matrix of a weighted planar network.  Concatenating these building block networks gives a matrix $A_N$ whose decomposition contains the $T^\pm E$ blocks corresponding to \eqref{eq:strict-pattern}, and the remaining approximants of the slithering factors are weight matrices of transparent networks. By Corollary \ref{cor:slithering-networks} and \eqref{eq:coarse-monodromy} we have $A_N \to A$ as $N\to \infty$, where $A$ is the determinant-one lift of $\rho(\gamma)$.  Moreover, by Corollary \ref{cor:approx-gbb} and \eqref{eq:closed-moduli}, replacing each weight by its modulus results in a weighted planar network with weight matrix $A_{0,N}$ such that  $A_{0,N} \to A_0$ where $A_0$  is a lift of $\rho_0(\gamma)$. 

Fix $I,J$ with $|I|=|J|=r<n$.  Applying Lemmas \ref{lem:universal-routing} and \ref{lem:coordinate-detecting} to the $T^\pm E$ factors of \eqref{eq:strict-pattern} and Lemma \ref{lem:neutral-insertion} to the inserted slithering approximants, we obtain two full $I$-to-$J$ path families,  with weights $u,v$ independent of $N$ such that $\frac uv=c\notin\R_{>0}$. 
Note that the independence of $N$ is crucial, and uses the fact that the slithering approximants are transparent.  By Lemma \ref{prop:strict-limit} we conclude that \eqref{eq:closed-strict-minors} holds.  Since $I,J$ were arbitrary, all proper minors are strict.
\end{proof}

\begin{cor}\label{cor:closed-ineq} For the curve $\gamma$ constructed in the above Proposition, we have the following strict inequalities of the Hilbert and translation lengths: 
\begin{equation*}
    \ell^H_\rho(\gamma) < \ell^H_{\rho_0}(\gamma) \text{ and } \ell_\rho(\gamma) < \ell_{\rho_0}(\gamma).
\end{equation*}
\end{cor}

\begin{proof} The strict inequality of Hilbert lengths follows exactly as in the proof of Proposition \ref{prop:hilbg}: by the strict matrix domination $\lvert A \rvert < A_0$ of Proposition \ref{prop:closed-strict} there is  $0<\alpha <1$ such that $\lvert A \rvert \leq \alpha A_0$; Lemma \ref{lem:linalg} then implies $$ \rvert \lambda_n \lvert  \leq \alpha \cdot  \lambda^\prime_n  $$ where $\lambda_n, \lambda_n^\prime$ are the largest eigenvalues (in modulus) of $A$ and $A_0$ respectively. 

By the strict domination of proper minors, if we define
\begin{equation}\label{eq:beta}
\beta:=\max_{1\leq |I|=|J|<n}
\frac{|\minor{I}{J}{A}|}{\minor{I}{J}{A_0}}.
\end{equation}
where note that $0<\beta <1$, we have by Lemma \ref{lem:linalg} that
$$\sigma\left(\lvert \bigwedge\limits^{n-1}  A\rvert\right) \leq \beta \sigma\left(\bigwedge\limits^{n-1} A_0 \right)$$
which implies that the product $$\lvert \lambda_n\lambda_{n-1} \cdots \lambda_2\rvert \leq \beta \cdot \lambda_n^\prime\lambda_{n-1}^\prime \cdots \lambda_2^\prime $$
which by the determinant-one normalization of $A$ and $A_0$, implies that $$\frac{1}{\lvert \lambda_1\rvert} \leq \beta \cdot \frac{1}{\lambda_1^\prime}.$$
The strict inequality  $\ell^H_\rho(\gamma) < \ell^H_{\rho_0}(\gamma)$ follows since Hilbert lengths are the ratios of the largest to smallest eigenvalues in modulus (see Definition \ref{defn:lengths}). 

 Similarly, the strict inequality of translation lengths follows from the same argument as in the proof of Proposition \ref{prop:trlbg}: applying Lemma \ref{lem:linalg} again, we have 
$$\beta \prod\limits_{i=0}^{k-1} \lvert \lambda_{n-i}^\prime \rvert \geq  \prod\limits_{i=0}^{k-1} \lvert \lambda_{n-i} \rvert$$ 
for each $1\leq k \leq n-1$, and for $k=n$, the equality $\prod\limits_{i=0}^{n-1} \lvert \lambda_{n-i}^\prime \rvert =  \prod\limits_{i=0}^{n-1} \lvert \lambda_{n-i} \rvert =1$. Taking logarithms, we have the \textit{strict} majorization 
$$(\log \lvert \lambda_1\rvert, \log \lvert \lambda_2 \rvert, \ldots, \log \lvert \lambda_n\rvert) \prec (\log \lvert \lambda_1^\prime\rvert, \log \lvert \lambda_2^\prime \rvert, \ldots, \log \lvert \lambda_n^\prime\rvert)$$
which by Karamata's inequality (since $s\mapsto s^2$ is a strictly convex function) implies a strict inequality of translation lengths (see, for example, \cite[Proposition 3.C.1.a]{MOA}). 
\end{proof}

\subsection{Completing the proof of Theorem \ref{thm:stat}}\label{sec:stat-domination}

In this section we shall complete the proof of Theorem \ref{thm:stat}, the statistical domination result. Throughout, let $\rho:\pi_1(S) \to \pslnc$ be a quasi-Hitchin representation in the bending fiber of a Hitchin representation $\rho_0:\pi_1(S) \to \pslnr$. Moreover  $\mathcal{S}(L)$ shall denote the set of closed curves on $S$ with Hilbert length with respect to $\rho_0$ at most $L$.

\medskip 

\noindent The main work in this section is to prove the statistical domination \eqref{eq:dens}, that we re-state here for the reader's convenience;  namely that for sufficiently large $L$, a ``density-one" proportion of curves in  $\mathcal{S}_0(L)$ satisfy a strict domination with factor $0<\lambda <1$:
\begin{equation*}
\displaystyle\frac{ { \{[\gamma] \  \vert\  \ell_{\rho_0}^H(\gamma) \leq L  \text{ and }\ell_\rho^H(\gamma) \leq \lambda \ell_{\rho_0}^H(\gamma)\} }\rvert }{\lvert \mathcal{S}_0(L)\rvert} \longrightarrow 1
\end{equation*} 

Moreover, when $\rho_0$ is $n$-Fuchsian, the analogous statistical domination holds for the translation length $\ell_\rho^X$.

\medskip 

 The proof of the above statement and its translation length counterpart relies on the observation that from the argument in the  proof of Proposition \ref{prop:closed-strict}, the length inequality in Corollary \ref{cor:closed-ineq} compounds when the curve has \textit{multiple} segments with the fixed coarse plaque itinerary. We record this as the following two propositions, the first for  Hilbert lengths, whose proof is a warmup for the proof of the second proposition concerning translation lengths.

 \noindent We start with a basic observation: 

 \begin{lem}[Compounding domination]\label{lem:compound}  Let $(\Gamma_0, \omega_0)$ be a weighted planar network with associated weight matrix $M_0$ such that the weight matrix  $M^\prime_0$ of  $(\Gamma_0, 
 \lvert \omega_0\rvert )$ strictly dominates, i.e. $$\beta M^\prime_0 \geq  \lvert M_0 \rvert$$ for some $0<\beta <1$. 
 Then if $(\Gamma, \omega)$ is a concatenation of planar networks including $m$ copies of   $(\Gamma_0, \omega_0)$, its weight matrix $M$ and the weight matrix $M^\prime$ of  $(\Gamma, \lvert \omega\rvert)$ satisfy the strict domination $$\beta^m M^\prime \geq \lvert M \rvert.$$ 
 \end{lem}

\begin{proof}
Let the concatenation be
\[
        (\Gamma,\omega)
        =
        (\Gamma_1,\omega_1)\circ\cdots\circ
        (\Gamma_N,\omega_N).
\]
with  weight matrices
\[
        M_i=W(\Gamma_i,\omega_i)
\]
The corresponding weight matrices  of $(\Gamma,\lvert \omega\rvert)$ are denoted by $M^\prime_i$.  

By the triangle inequality (as in the proof of Lemma \ref{lem:dom}) we have:
\begin{equation}\label{eq:factorwise-modulus-bound}
        \lvert M_i\rvert\leq M_i'
        \qquad (1\leq i\leq N).
\end{equation}
and for the copies of $(\Gamma_0,\omega_0)$, the hypothesis implies the strict inequality 
\begin{equation}\label{eq:factorwise-strict-bound}
        \lvert M_i\rvert
        \leq \beta M_i'
        \qquad (i\in S)
\end{equation}
where $0<\beta <1$.
Since concatenation of networks corresponds to multiplication of their weight
matrices, we have 
\[
        M=M_1M_2\cdots M_N,
        \qquad
        M'=M_1'M_2'\cdots M_N'.
\]
and hence 
\[
        \lvert M\rvert
        \leq
        \lvert M_1\rvert\lvert M_2\rvert\cdots\lvert M_N\rvert \leq \beta^m   M_1'M_2'\cdots M_N'  = \beta^m M^\prime
\]
which is the desired strict domination.
\end{proof}

\begin{prop}[Multiple traversals, Hilbert length]\label{prop:multgain1}Let $\gamma$ be a closed curve on $S$ with $m$ pairwise-disjoint segments, each having the coarse itinerary \eqref{eq:strict-pattern}. Let $A, A_0 \in \slnc$ be the determinant-one lifts of $\rho(\gamma), \rho_0(\gamma)$ respectively. Then we have 
\begin{equation}\label{eq:mult-H}
    \ell_{\rho}^H(\gamma) \leq \ell_{\rho_0}^H(\gamma) - 2m \log (\alpha^{-1}) 
\end{equation}
for some $0<\alpha<1$.
\end{prop}
\begin{proof}
    Here $\alpha$ is the strict domination factor, for all minors of the product of $T^\pm E$ building blocks corresponding to the pattern \eqref{eq:strict-pattern} in the proof of Proposition \ref{prop:closed-strict}.
    The curve $\gamma$ traverses \eqref{eq:strict-pattern} $m$ times implies that in the planar networks $(\Gamma_N, \omega_N)$  associated with the approximants $A_N \to A$, we have $m$ distinct instances of the planar network corresponding to the pattern \eqref{eq:strict-pattern}. Recall that the approximants $A_{0,N} \to A_0$, correspond to the planar networks $(\Gamma_N, \lvert \omega_N\rvert)$.  By concatenation, it follows that from Lemma \ref{lem:compound} that we have the matrix domination  $\lvert A \rvert \leq \alpha^m A_0$, and applying the same lemma to exterior powers, we have $\lvert \bigwedge\limits^kA \rvert \leq  \alpha^m\bigwedge\limits^k A_0$ for each $1\leq k<n$.  Let $\lambda_n, \lambda_n^\prime$ be the largest eigenvalues (in modulus) of $A$ and $A_0$ respectively. It follows from Lemma \ref{lem:linalg} that $\sigma(A) \leq \alpha^m \sigma(A_0)$ and hence there is the following inequality:
    \begin{equation*}
        \log \lvert  \lambda_n \rvert \leq  \log \lvert  \lambda_n^\prime  \rvert - m\log(1/\alpha).
    \end{equation*}
    Applying the domination inequality for the $(n-1)$-th exterior-power, we also have $$\sigma\left(\lvert \bigwedge\limits^{n-1}  A\rvert\right) \leq \alpha^m \sigma\left(\bigwedge\limits^{n-1} A_0 \right)$$
    which implies, by Lemma \ref{lem:linalg} and the determinant-one condition, that:
     \begin{equation*}
        -\log \lvert  \lambda_1 \rvert \leq  -\log \lvert  \lambda_1^\prime  \rvert - m\log(1/\alpha)
    \end{equation*}
    and adding these inequalities yields \eqref{eq:mult-H}. \end{proof}

\begin{prop}[Multiple traversals, translation length]\label{prop:multgain2} As in the preceding  proposition, let $\gamma$ be the closed curve on $S$, traversing \eqref{eq:strict-pattern} $m \geq 1$ times,  and $A,A_0 \in \slnc$ be the representatives of $\rho(\gamma), \rho_0(\gamma)$. Then we have
\begin{equation}\label{eq:repeated-X}
 \ell_\rho(\gamma)
 \leq
 \ell_{\rho_0}(\gamma)- (2\sqrt{2}\delta/\sqrt n) m
\end{equation}
for some $\delta>0$, where recall that $\ell_\rho(\cdot)$ denotes the translation length. 
\end{prop}

\begin{proof}
Let $C_N$ be the weight matrix of the planar network corresponding to the  pattern \eqref{eq:strict-pattern}, and let $C_{0,N}$ be the weight matrix of corresponding network where each weight is replaced by its modulus. The path-family construction in the proof of
Proposition~\ref{prop:closed-strict} gives, for every
$I,J\subset\{1,\ldots,n\}$ with
$1\leq |I|=|J|<n$, two $I$-to-$J$ path families whose weights are
independent of $N$ and whose ratio does not belong to
$\mathbb R_{>0}$.

By applying Lemma \ref{prop:strict-limit} to each pair $I,J \subset \{1,2,\ldots,n\}$ where $1\leq |I|=|J| \leq n-1$, there is a constant
$0<q<1$, independent of $N$, such that
\begin{equation}\label{eq:local-uniform-exterior}
  \left|\bigwedge^k C_N\right|
  \leq
  q\,\bigwedge^k C_{0,N}
\end{equation}
entrywise for every $1\leq k<n$.

Suppose that $\gamma$ contains $m$ pairwise interior-disjoint
occurrences of \eqref{eq:strict-pattern}.  We can choose the occurrences
with the same orientation and choose the finite exhaustions of their
slithering factors equivariantly.  Consequently there would be $m$ distinct occurrences of the matrix $C_N$ and $C_{0,N}$  in the approximants $A_N \to A$ and $A_{0,N} \to A_0$ respectively (\textit{cf.} Lemma \ref{lem:closed-monodromy-approximants}). 
In other words, we can write $$A_N = B_0C_NB_1 \cdots C_NB_m$$
where $C_N$ occurs $m$ times, and correspondingly $$A_{0,N} = B_0^\prime C_{0,N}B_1^\prime \cdots C_{0,N}B_m^\prime$$
where the remaining $B^\prime$-matrices have all non-negative entries and by Lindstr\"{o}m's lemma and the triangle inequality, satisfy the domination $\lvert \bigwedge\limits^k B_j \rvert \leq \bigwedge\limits^k B_j^\prime$ for each $0\leq j\leq m$ and $1\leq k <n$.

From the functorial property of the exterior power, and  the basic inequality $|XY|\leq |X||Y|$ it follows that $$\left\vert\bigwedge^k A_N \right\vert \leq q^m \left\vert\bigwedge^k A_{0,N} \right\vert  $$
holds for each $1\leq k\leq n-1$, where we have collected the $m$ $q$-factors corresponding to \eqref{eq:local-uniform-exterior}.

Passing to the limit, we obtain  $$\left\vert\bigwedge^k A \right\vert \leq q^m \left\vert\bigwedge^k A_{0} \right\vert  $$ and  Lemma \ref{lem:linalg} then implies 
$$ q^m \sigma(\bigwedge^k A_0) = q^m \prod\limits_{i=0}^{k-1} \lvert \lambda_{n-i}^\prime \rvert \geq  \prod\limits_{i=0}^{k-1} \lvert \lambda_{n-i} \rvert  =   \sigma(\bigwedge^k A)$$  where the eigenvalues of $A$ (respectively $A_0$) are $\lvert \lambda_1\rvert \leq \lvert \lambda_2\rvert \leq  \cdots \leq \lvert \lambda_n\rvert$ (respectively $\lvert \lambda_1^\prime\rvert \leq \lvert \lambda_2^\prime\rvert \leq  \cdots \leq \lvert \lambda_n^\prime \rvert $).  
Taking logarithms and writing 
\begin{equation*}
x_i=\log|\lambda_{n-i+1}|,
 \qquad
  y_i=\log|\lambda^\prime_{n-i+1}|,
 \qquad 1\leq i\leq n,
\end{equation*}
we then have 
\begin{equation}\label{eq:quantitative-partial-sums}
 \sum_{i=1}^k x_i
 \leq
 \sum_{i=1}^k y_i-m\delta
\end{equation}
for each $1\leq k\leq n-1$, where $\delta := \log (1/q)$; note that  $\sum_{i=1}^n x_i = \sum_{i=1}^n y_i =0$ by the determinant-one condition. Let \[
 s_k=\sum_{i=1}^k(y_i-x_i)
\]
for each $1\leq k\leq n-1$,  and note that each $s_k \geq m\delta$. In what follows define $s_0=s_n=0$, and let 
\[
 c_i :=x_i+y_i
\]
which is a non-increasing sequence  since both $\{x_i\}_{1\leq i\leq n}$ and $\{y_i\}_{1\leq i\leq n}$ are non-increasing. Since
\[
 y_i-x_i=s_i-s_{i-1},
\]
summation by parts gives
\begin{equation*}
 \sum_{i=1}^n(y_i^2-x_i^2) =  \sum_{i=1}^n(y_i-x_i)(y_i + x_i)  = \sum_{i=1}^n(s_i-s_{i-1})c_i = \sum_{i=1}^{n-1}s_i(c_i-c_{i+1}).
\end{equation*}
Since  each $c_i-c_{i+1}\geq 0$ we have
\begin{equation*}
 \sum_{i=1}^n(y_i^2-x_i^2) \geq
 m\delta\sum_{i=1}^{n-1}(c_i-c_{i+1}) =m\delta(c_1-c_n)\\
 =m\delta
 \bigl((x_1-x_n)+(y_1-y_n)\bigr) =m\delta
 \bigl(
   \ell^H_\rho(\gamma)+\ell^H_{\rho_0}(\gamma)
 \bigr)
\end{equation*}
which can be summarized as 
\begin{equation}\label{pf-ineq}
\ell_{\rho_0}(\gamma)^2 - \ell_{\rho}(\gamma)^2 \geq 2m\delta \bigl(
   \ell^H_\rho(\gamma)+\ell^H_{\rho_0}(\gamma)
 \bigr).
\end{equation}

\medskip 
\noindent We now observe:

\smallskip

\noindent \textit{Claim.} \textit{The inequality $ \ell^H_\rho(\gamma) \geq \frac{2}{\sqrt 2n} \ell_{\rho}(\gamma)$ always holds.} 

\noindent \textit{Proof of claim.} It suffices to prove the elementary inequality
\begin{equation*}
 z_1-z_n\geq\frac{2}{\sqrt n}
 \left(\sum_{j=1}^n z_j^2\right)^{1/2}
\end{equation*}
for every non-increasing vector
$z_1\geq\cdots\geq z_n$ that satisfies $\sum_jz_j=0$.

For this, let $M=z_1$ and $m=z_n$.  Since $m\leq z_j\leq M$,
\[
 (M-z_j)(z_j-m)\geq0,
\]
and therefore
\[
 z_j^2\leq(M+m)z_j-Mm.
\]
Summing and using $\sum_jz_j=0$ gives
\[
 \sum_{j=1}^n z_j^2\leq-nMm \leq  n\frac{(M-m)^2}{4}
\]
which completes the proof. $\qed$

\smallskip 

\noindent Applying the claim to \eqref{pf-ineq} then gives
\begin{equation*}
 \bigl(\ell_{\rho_0}(\gamma)\bigr)^2
 -
 \bigl(\ell_\rho(\gamma)\bigr)^2 \geq
 \frac{4\delta m}{\sqrt{2n}}
 \bigl(
   \ell_{\rho_0}(\gamma)+\ell_\rho(\gamma)
 \bigr)
\end{equation*}
which implies \eqref{eq:repeated-X} by cancelling the non-zero $\bigl(
   \ell_{\rho_0}(\gamma)+\ell_\rho(\gamma)
 \bigr)$ factor. \end{proof}

\medskip 

\noindent The proof of Theorem  \ref{thm:stat} then follows from the arguments in the proofs of Lemma \ref{lem:lall} and the entropy-rigidity for $\pslc$-representations in \S4.2:

\begin{proof}[Proof of Theorem \ref{thm:stat}]
We can consider the closed geodesics to be  oriented, so that they correspond bijectively to conjugacy classes of elements of $\pi_1 S$; since $l^H_{\rho_0}(\gamma) = l^H_{\rho_0}(\gamma^{-1})$, passing to unoriented geodesics would only change the counts below by a factor of $2$.

We first consider the case when $\rho_0$ is an $n$-Fuchsian representation; suppose $\rho_0 = \iota \circ j$, where $j : \pi_1(S) \to \pslr$ is Fuchsian and $\iota : \pslr \to \pslnr$ is the irreducible representation. Then the Hilbert length satisfies $\ell_{\rho_0}^H(\gamma) = (n-1)\, l_j(\gamma)$, where $l_j(\gamma)$ is the hyperbolic length of $\gamma$ on the hyperbolic surface $X = \mathbb{H}^2 / j(\pi_1 S)$: indeed, if $A \in \pslr$ has translation length $\ell$, the eigenvalue moduli of $\iota(A)$ are $e^{(n-1)\ell/2}, e^{(n-3)\ell/2}, \ldots, e^{-(n-1)\ell/2}$. (See, for example, \cite[Example~1.3]{PotrieSambarinoEigenvaluesEntropy}.)
Let $\mathcal{S}_0(L)$ be the set of closed geodesics of Hilbert length (with respect to $\rho_0$) at most $L$; this is precisely the set of closed geodesics of (hyperbolic) length at most $L/(n-1)$ on $X$. By Lemma \ref{lem:lall} (applied with the length bound $L/(n-1)$), there is a $\beta>0$ such that for an almost-full proportion of the closed geodesics in $\mathcal{S}_0(L)$, the number of times the pattern \eqref{eq:strict-pattern} occurs along a geodesic is at least a $\beta$-proportion of the total geodesic length. Proposition \ref{prop:multgain1}, together with the same argument in \S4.2 (see \eqref{eq:ineq} and the following argument that invokes Lemma \ref{lem:lall}) implies statistical domination. 

Note that the second statement of the proposition involving translation length also follows from this, since for $\rho_0$ $n$-Fuchsian, the closed geodesics of translation length at most $L$ coincides with the set of closed geodesics of (hyperbolic) length at most $\sqrt{6/(n^3-n)}L$ on $X$ (\textit{cf.} Remark (ii) after Corollary \ref{thm:ent}). Once again, Lemma \ref{lem:lall} applies to show that for sufficiently large $L$ an almost-full proportion of such geodesics have a $\beta$-proportion of traversals of the pattern \eqref{eq:strict-pattern}. This time we use Proposition \ref{prop:multgain2} and the argument in \S4.2 to conclude that statistical domination holds. 

Now let $\rho_0$ be an arbitrary Hitchin representation. The key facts used in the proof of Lemma \ref{lem:lall} were: (i) closed geodesics of length at most $L$ equidistribute, as $L \to \infty$, with respect to a measure $\mu$ on the unit tangent bundle of the surface, (ii) the measure $\mu$ has full support, and (iii) $\mu$ is ergodic
for the flow (used in the proof of the claim \eqref{eq:mean-deviation}. It is known that all three properties hold for the Hilbert length $l^H_{\rho_0}$, with the geodesic flow replaced by the Hilbert flow $\psi^H_t$ defined in the Appendix (see Proposition \ref{prop:app-hilbert-flow}) and with $\mu = \widetilde{m}_H$ its measure of maximal entropy: these are summarized as Proposition \ref{prop:hil-equi} in the Appendix; the ergodicity follows from the uniqueness of the measure of maximal entropy. 

We note that to adapt Lemma \ref{lem:lall} the only point to note is that the Hilbert flow differs from the geodesic flow on $X$ (where $X$ is now an auxiliary choice of a hyperbolic metric on $S$). However, since $S$ is compact, the reparametrizing density of $\psi^H_t$ is bounded above and below. Thus, the time spent in an open set $U \subset T^1X$ and hence the number of traversals of \eqref{eq:strict-pattern} with respect to both the flows, are uniformly comparable, and the proof of Lemma \ref{lem:lall} goes through. 
\end{proof}

\noindent \textit{Remark.} As mentioned in the Introduction, in the case of dominating translation-lengths we expect the statistical domination to hold for an \textit{arbitrary} Hitchin representation $\rho_0$, not just an $n$-Fuchsian one. The additional fact needed is the equidistribution of closed geodesics, when ordered by translation length, to a measure that has \textit{full support} in the unit tangent bundle of $S$ (\textit{cf.} the final remark of the Appendix).  The equidistribution of closed geodesics, when ordered by any fundamental-weight length, holds by the same argument as in the Appendix since it is a linear functional on the Jordan projection.

\subsection{Proof of Corollaries \ref{thm:ent} and \ref{cor:dim}} \label{sec:cors}

We can now formally complete the proofs of Corollary \ref{thm:ent} and Corollary \ref{cor:dim}.

\medskip 

\begin{proof}[Proof of Corollary \ref{thm:ent} (Entropy rigidity for bending)]
For both the Hilbert entropy and translation length entropy, the strict entropy inequality in a bending fiber follows immediately from the statistical domination of lengths Theorem \ref{thm:stat}, by applying Lemma \ref{lem:stat}. \end{proof}

\medskip 

\begin{proof}[Proof of Corollary \ref{cor:dim} (Hausdorff dimension rigidity for $n$-Fuchsian fibers)]

We use the fact that, in our normalization, for an $n$-Fuchsian representation  the Hilbert entropy $H(\rho_0) = \frac{1}{n-1}$ (see \cite[pg. 892]{PotrieSambarinoEigenvaluesEntropy}). It is also well-known that for such a $\rho_0$ the full limit map $\xi$ is $C^1$-smooth and its image $\Lambda_{\rho_0}$ has Hausdorff dimension $\text{Hdim}(\Lambda_{\rho_0}) =1$.

Let $\rho\neq \rho_0$ be hyperconvex representation in the bending fiber of $\rho_0$; we need to prove that the Hausdorff dimension $\text{Hdim}(\Lambda_\rho) >1$. The hyperconvexity assumption is needed to use the formula \eqref{eq:hdim}.

For any conjugacy class $[\gamma]$ of an element $\gamma \in \pi_1(S)$ we have $\ell_\rho^H(\gamma) \leq (n-1)L \implies l_\rho^i(\gamma) \leq L $ for some $i\in \{1,2,\ldots,n-1\}$. 
In particular, we have 
\begin{equation*}
\#\{ [\gamma] \ \vert\ \ell_\rho^H(\gamma) \leq (n-1)L\} \leq \sum\limits_{i=1}^{n-1}\# \{[\gamma]\ \vert\ l_\rho^i(\gamma) \leq L\} \leq (n-1) \max\limits_{i} \# \{[\gamma]\ \vert\ l_\rho^i(\gamma) \leq L\}
\end{equation*}
We then obtain: 
\begin{equation*}
\frac{\log \#\{ [\gamma]\ \vert\ \ell_\rho^H(\gamma) \leq (n-1)L \}}{L} \leq \frac{\log(n-1)}{L} +  \max\limits_{i} \frac{\log \# \{[\gamma]\ \vert\ l_\rho^i(\gamma) \leq  L\}}{L}
\end{equation*}
Taking a limit on both sides as $L\to \infty$, we obtain from Definition \ref{defn:ent} that 
\begin{equation*}
(n-1)H(\rho) \leq \max\limits_i h^i(\rho) = \text{Hdim}(\Lambda_\rho).
\end{equation*}
However, by Corollary \ref{thm:ent} we know $H(\rho) > H(\rho_0) = \frac{1}{n-1}$, and hence $\text{Hdim}(\Lambda_\rho) >1$ as desired. 
\end{proof}

\appendix

\section{Equidistribution of closed curves ordered by Hilbert length}
\label{app:hilbert-equidistribution}

In this appendix, we provide the context and proof of the equidistribution result used in the proof of Theorem \ref{thm:stat}. This seems to be  well known to experts; indeed, the proofs below are assembled from the results of \cite{Sam14}. 

We begin by recalling the construction of the \textit{Hilbert flow} associated to a Hitchin representation $\rho_0$.  Throughout, let
$\Gamma=\pi_1(S)$, fix an auxiliary hyperbolic structure on $S$, and let
$\phi_t:T^1S\to T^1S$ be its geodesic flow.  We identify the Gromov
boundary $\partial_\infty\Gamma$ with the ideal boundary of the universal
cover, and let $\partial_\infty^{(2)}\Gamma
  :=\bigl(\partial_\infty\Gamma\times\partial_\infty\Gamma\bigr)
  \setminus\Delta$.
  
As we shall explain in Proposition \ref{prop:app-hilbert-flow}, the Hilbert flow is an $\mathbb{R}$-action on $\Gamma\backslash(\partial_\infty^{(2)}\Gamma \times \mathbb{R})$ that is a  H\"older reparametrization of $\phi_t$ such that the period of a closed orbit corresponding to a closed geodesic $\gamma$ equals the Hilbert length $\ell^H_{\rho_0}(\gamma)$.

\subsection{The Hilbert flow}
We shall use Sambarino's construction of reparametrized geodesic flows
from H\"older cocycles \cite{Sam14}.
We first verify that a Hitchin representation lies in the class of
representations to which Sambarino's results apply.

\begin{lem}\label{lem:app-strict-convex}
A Hitchin representation  $\rho_0:\Gamma\to\pslnr$ is strictly convex in
the sense of \cite[Definition~1.1]{Sam14}.
\end{lem}

\begin{proof}
By Labourie \cite{Lab}, the representation $\rho_0$ is irreducible and
admits a H\"older-continuous, $\rho_0$-equivariant limit map 
$ \xi:\partial_\infty\Gamma\longrightarrow\mathcal F(\mathbb R^n)$
which satisfies
\[
  \mathbb R^n=\xi_1(x)\oplus\xi_{n-1}(y)
  \qquad\text{whenever }x\ne y.
\]
where $\xi_1$ and $\xi_{n-1}$ are its line and hyperplane
components.
Thus the pair $(\xi_1,\xi_{n-1})$, together with irreducibility, satisfies the definition of a strictly convex representation in
\cite[Definition~1.1]{Sam14}; see also the discussion of Hitchin
representations in \cite[\S~1]{Sam14}.
\end{proof}

Fix a norm on $\mathbb R^n$ and the dual norm on $(\mathbb R^n)^*$.
For an element of $\pslnr$, choose a determinant-one representative;
the expressions below are independent of this choice.  Following \cite[\S~6]{Sam14} define
the H\"{o}lder cocycles $\beta_1,\bar\beta_1:\Gamma\times\partial_\infty\Gamma\to\mathbb R$ by
\begin{equation}\label{eq:app-cocycles}
  \beta_1(\gamma,x)
  :=\log\frac{\|\rho_0(\gamma)v\|}{\|v\|},
  \qquad
  \bar\beta_1(\gamma,x)
  :=\log\frac{\|\rho_0(\gamma)^*w\|}{\|w\|},
\end{equation}
where $v\in\xi_1(x)\setminus\{0\}$, where
$w\in(\mathbb R^n)^*\setminus\{0\}$ satisfies $\ker w=\xi_{n-1}(x)$, and
$\rho_0(\gamma)^*w=w\circ\rho_0(\gamma)^{-1}$, and  set
\[
  c_H:=\beta_1+\bar\beta_1.
\]

In what follows, for every nontrivial $\gamma\in\Gamma$, let
$\gamma^+,\gamma^-\in\partial_\infty\Gamma$ denote respectively its
attracting and repelling fixed points.  Equivalently, after identifying
$\partial_\infty\Gamma$ with $\partial_\infty\mathbb H^2$ using the
auxiliary hyperbolic structure, these are the forward and backward
endpoints of the axis of $\gamma$. For the next lemma, the \textit{period} of $\gamma$ with respect to a H\"{o}lder cocycle $c_H$ is  $\ell_{c_H}(\gamma):=c_H(\gamma,\gamma^+)$; these shall be the periods of the Hilbert flow in Proposition \ref{prop:app-hilbert-flow}.  

\begin{lem}\label{lem:app-hilbert-cocycle}
For every nontrivial $\gamma\in\Gamma$, and the H\"{o}lder cocycle $c_H$ as above,  the period  of $\gamma$ with respect to $c_H$ is precisely the 
Hilbert length with respect to the Hitchin representation $\rho_0$:
\begin{equation}\label{eq:app-hilbert-period}
  \ell_{c_H}(\gamma)
  :=c_H(\gamma,\gamma^+)
  =\ell^H_{\rho_0}(\gamma).
\end{equation}
In particular, all nonzero periods of $c_H$ are positive.
\end{lem}

\begin{proof}
As earlier in the paper, let  $\lambda_1 \leq \lambda_2 \leq\cdots\leq\lambda_n$ 
be the eigenvalue moduli of the  determinant-one
representative of $\rho_0(\gamma)$.  By \cite[Corollary~5.3]{Sam14},
$\beta_1(\gamma,\gamma^+)= \log \lambda_n $ and by \cite[Lemma~6.3]{Sam14}, we have $\bar\beta_1(\gamma,\gamma^+)
  =-\log \lambda_1$
since the largest eigenvalue in modulus of $\rho_0(\gamma)^{-1}$ is the reciprocal of the smallest eigenvalue in modulus of $\rho_0(\gamma)$. 

Adding these we obtain 
\[
  \ell_{c_H}(\gamma)
  =\log \lambda_n-\log\lambda_1
  =\ell^H_{\rho_0}(\gamma)
\]
which proves \eqref{eq:app-hilbert-period}.  Moreover,
\cite[Corollary~5.3]{Sam14}, applied to $\gamma$ and to $\gamma^{-1}$,
shows that both
$\beta_1$ and $\bar\beta_1$ have positive periods, and hence so does their sum.
\end{proof}

Let $[\Gamma]$ denote the set of non-trivial conjugacy classes in $\Gamma$, and
define the exponential growth rate of $c_H$ by
\[
  h_H:=h_{c_H}
  :=\limsup_{T\to\infty}\frac{1}{T}
  \log\#\bigl\{[\gamma]\in[\Gamma] \ \vert\ 
  \ell_{c_H}(\gamma)\leq T\bigr\}.
\]

\begin{lem}\label{lem:app-hilbert-growth}
The Hilbert cocycle has finite positive exponential growth rate $ 0<h_H<\infty$. 
\end{lem}

\begin{proof}
By Lemma~\ref{lem:app-hilbert-cocycle}, the cocycle $c_H$ has positive
periods.  Therefore \cite[Corollary~3.6]{Sam14} gives $h_H>0$.
Since
\[
  \ell_{c_H}(\gamma)
  =\log \lambda_n(\rho_0(\gamma))- \log\lambda_1(\rho_0(\gamma))
  \geq \log \lambda_n(\rho_0(\gamma)).
\]
we have
\[
  \#\bigl\{[\gamma]:\ell_{c_H}(\gamma)\leq T\bigr\}
  \leq
  \#\bigl\{[\gamma]:\log \lambda_n(\rho_0(\gamma))\leq T\bigr\}.
\]
The exponential growth rate of the quantity on the right is finite by
\cite[Proposition~5.4]{Sam14}, since $\rho_0$ is strictly convex by
Lemma~\ref{lem:app-strict-convex}.  This proves $h_H<\infty$.
\end{proof}

\noindent We can now construct the Hilbert flow, whose periods are the Hilbert lengths.

\begin{prop}[The Hilbert flow]\label{prop:app-hilbert-flow}
Let $\Gamma$ act on $\partial_\infty^{(2)}\Gamma\times\mathbb R$ by
\begin{equation}\label{eq:app-boundary-action}
  \eta\cdot(x,y,s)
  :=\bigl(\eta x,\eta y,s-c_H(\eta,y)\bigr).
\end{equation}
This action is properly discontinuous and cocompact, with the quotient 
\[
  X_H:=\Gamma\backslash
  \bigl(\partial_\infty^{(2)}\Gamma\times\mathbb R\bigr)
\]
admitting the translation flow
\[
  [x,y,s]\longmapsto[x,y,s-t]
\]
that satisfies:
\begin{itemize}
    \item[(i)] The translation flow on $X_H$ 
is H\"older conjugate to a H\"older reparametrization of $\phi_t$ on
$T^1S$.  (In what follows, we use this conjugacy to regard it as a flow $\psi^H_t$ on $T^1S$.)

\item[(ii)] The topological entropy of $\psi^H_t$ is $h_H$.

\item[(iii)] The  closed
orbits of the flow correspond to  closed geodesics
on $S$, and the orbit corresponding to a geodesic $\gamma$ has period
$\ell^H_{\rho_0}(\gamma)$. 

\item[(iv)] The translation flow $\psi^H_t$ has a unique invariant
probability measure of maximal entropy, denoted by $\widetilde m_H$.
\end{itemize}
\end{prop}

\begin{proof}
Lemmas~\ref{lem:app-hilbert-cocycle} and
\ref{lem:app-hilbert-growth} verify the hypotheses of Sambarino's
reparametrizing theorem \cite[Theorem~3.2]{Sam14}.  Part~(1) of that
theorem implies the properness and cocompactness of the action
\eqref{eq:app-boundary-action}, together with (i) and (ii). 

For a closed geodesic $\gamma$, the line
$\{(\gamma^-,\gamma^+,s):s\in\mathbb R\}$ projects to a closed orbit, and \eqref{eq:app-boundary-action} shows that its period is
$c_H(\gamma,\gamma^+)=\ell^H_{\rho_0}(\gamma)$.  Conversely, every
closed orbit arises in this way -- see for example the proof of \cite[Corollary~4.1]{Sam14}.  Since $\psi^H_t$ is a
H\"older reparametrization of an Anosov geodesic flow, the existence and
uniqueness of its measure of maximal entropy follow from
\cite[Corollary~2.7]{Sam14}.
\end{proof}

\noindent \textit{Remark.} We shall refer to the measure of
maximal entropy  $m_H$ for the Hilbert flow in (iv) of the above Proposition as the \emph{Bowen--Margulis measure} of the Hilbert flow.

\subsection{Support of the Bowen-Margulis measure}

\begin{lem}\label{lem:app-full-support}
The Bowen-Margulis measure $\widetilde m_H$ in Proposition \ref{prop:app-hilbert-flow} has full support in $T^1S$.
\end{lem}

\begin{proof}
Let
\[
  q:\partial_\infty^{(2)}\Gamma\times\mathbb R\longrightarrow X_H
\]
be the quotient map.  Under the identification of Proposition~\ref{prop:app-hilbert-flow},
\cite[Theorem~3.2(2)]{Sam14} states that the $\Gamma$-invariant measure
which induces $\widetilde m_H$ on $X_H$ is,
up to multiplication by a positive constant,
\begin{equation}\label{eq:app-bms-lift}
  e^{-h_H[\,\cdot,\cdot\,]}
  \,\bar\nu\otimes\nu\otimes ds.
\end{equation}
Here $\nu$ and $\bar\nu$ are the Patterson--Sullivan measures
associated respectively to $c_H$ and to a dual cocycle, and
$[\,\cdot,\cdot\,]$ is a Gromov product for
this pair of cocycles.

By  \cite[p.~451]{Sam14}, such a Patterson--Sullivan measure $\mu$ associated to a cocycle $c$ satisfies the 
 Radon--Nikodym derivative
\[
  \frac{d(\gamma_*\mu)}{d\mu}(x) = e^{-h_c c(\gamma^{-1},x)}
\]
for each $\gamma\in\Gamma$.

The existence and uniqueness of this Patterson--Sullivan measure
are recorded in \cite[Theorem~3.1]{Sam14}, and attributed to Ledrappier.  Note that from the above expression, if the topological entropy $h_c$ is finite (\textit{cf.} Lemma \ref{lem:app-hilbert-growth}) the  Radon--Nikodym derivative is strictly positive for
$x \in \partial_\infty\Gamma$.  In particular, for any  Borel set
$E\subset\partial_\infty\Gamma$,
\[
  (\gamma_*\mu)(E)
  =\int_E  e^{-h_c c(\gamma^{-1},x)}\,d\mu,
\]
and hence
\[
  (\gamma_*\mu)(E)=0
  \quad\Longleftrightarrow\quad
  \mu(E)=0
\]
showing that $\gamma_*\mu$ and $\mu$ have the same support.  Moreover, since
$\gamma$ acts on $\partial_\infty\Gamma$ by a homeomorphism, we have $ \operatorname{supp}(\gamma_*\mu)
  =\gamma\bigl(\operatorname{supp}\mu\bigr)$ which implies that
\[
  \gamma\bigl(\operatorname{supp}\mu\bigr)
  =\operatorname{supp}\mu
\]
for each $\gamma \in \Gamma$, and hence $\operatorname{supp}\mu$ is a nonempty closed $\Gamma$-invariant subset of $\partial_\infty\Gamma$.

However, the action of the nonelementary hyperbolic group $\Gamma$ on its Gromov boundary is known to be minimal
\cite[Proposition~4.2(2)]{KB02}. 

Applying this discussion to $\nu$ and $\bar\nu$ in our setting, we obtain that both these have full support (equal to $\partial_\infty\Gamma$). Since the density in \eqref{eq:app-bms-lift} is strictly positive on
$\partial_\infty^{(2)}\Gamma$, the measure in
\eqref{eq:app-bms-lift} assigns positive mass to every nonempty open
subset of $\partial_\infty^{(2)}\Gamma\times\mathbb R$. Finally, every nonempty open subset of
$X_H$ has a nonempty open inverse image (since the quotient map $q$ is an open map)  and from the above discussion has positive
$\widetilde m_H$-measure.  Thus $\widetilde m_H$ has full support on
$X_H$, and the H\"older conjugacy in
Proposition~\ref{prop:app-hilbert-flow} transfers this property to
$T^1S$.
\end{proof}

\subsection{Equidistribution of closed orbits}

For $L>0$, let $\mathcal S_0(L)$ be the set of closed
geodesics $\gamma$ on $S$ satisfying
$\ell^H_{\rho_0}(\gamma)\leq L$.  For such a geodesic, let
$\widetilde{\mathcal L}_\gamma$ be the invariant probability measure on
the corresponding $\psi^H_t$-orbit.  Thus, for every continuous
function $F:T^1S\to\mathbb R$,
\begin{equation}\label{eq:app-orbit-measure}
  \int_{T^1S}F\,d\widetilde{\mathcal L}_\gamma
  =\frac{1}{\ell^H_{\rho_0}(\gamma)}
  \int_0^{\ell^H_{\rho_0}(\gamma)}
  F\bigl(\psi^H_t(v_\gamma)\bigr)\,dt,
\end{equation}
where $v_\gamma$ is any point on that closed orbit.

Before we state the equidistribution result, we note that by \cite[Corollary~4.1]{Sam14} (the hypotheses of which are satisfied by Lemmas~\ref{lem:app-hilbert-cocycle} and
\ref{lem:app-hilbert-growth}) we obtain the asymptotic count
\begin{equation}\label{eq:app-prime-orbit}
  h_H L e^{-h_HL}\,\lvert\mathcal S_0(L)\rvert
  \longrightarrow 1
  \qquad\text{as }L\to\infty.
\end{equation}
Note that although \cite[Corollary~4.1]{Sam14} states a count of conjugacy classes of \textit{primitive} geodesics, it is easy to derive the asymptotic count in \eqref{eq:app-prime-orbit}. This is because the {non-primitive} geodesics are \textit{positive} powers of primitive geodesics, and contribute a count that is bounded by half the exponential rate by \cite[Corollary~4.1]{Sam14}. 
In fact, one consequence of this count is that in the definition of Hilbert entropy (see Definition \ref{defn:ent}) the $\limsup$ can be replaced by a limit. 
\medskip

\noindent Moreover, we have: 

\begin{lem}
\label{lem:app-weighted-equidistribution}
As $L\to\infty$,
\begin{equation}\label{eq:app-weighted-equidistribution}
  h_H L e^{-h_HL}
  \sum_{\gamma\in\mathcal S_0(L)}
  \widetilde{\mathcal L}_\gamma
  \xrightarrow{\;\ast\;}
  \widetilde m_H.
\end{equation}
\end{lem}

\begin{proof}
By \cite[Lemma~2.9]{Sam14}, a Markov coding for the auxiliary geodesic
flow induces a Markov coding for its H\"older reparametrization
$\psi^H_t$.  By \cite[Corollary~2.10]{Sam14}, the corresponding
suspension flow is topologically weakly mixing.  The prime-orbit
theorem and Bowen's periodic-orbit equidistribution theorem, in the
forms recorded as \cite[Theorems~4.2 and~4.4]{Sam14}, therefore apply.
The proof of \cite[Proposition~4.3]{Sam14}
then implies the convergence 
\[
  h_H L e^{-h_HL}
  \sum_{\substack{\mathcal O\text{ orbit with period}\
                  p(\mathcal O)\leq L}}
  \frac{1}{p(\mathcal O)}\operatorname{Leb}_{\mathcal O}
  \xrightarrow{\;\ast\;}\widetilde m_H.
\]
where $p(\mathcal O)$ denotes the period of the closed orbit $\mathcal O$. This proves
\eqref{eq:app-weighted-equidistribution} since from the correspondence in Proposition~\ref{prop:app-hilbert-flow},
$p(\mathcal O)=\ell^H_{\rho_0}(\gamma)$ and the normalized orbit measure
$p(\mathcal O)^{-1}\operatorname{Leb}_{\mathcal O}$ is exactly
$\widetilde{\mathcal L}_\gamma$ as defined in
\eqref{eq:app-orbit-measure}. 
\end{proof}

\noindent We finally state the equidistribution result used in the proof of Theorem \ref{thm:stat}:

\begin{prop}\label{prop:hil-equi}
Let $\rho_0:\pi_1S\to\pslnr$ be a Hitchin representation and let
$\psi^H_t$ be its Hilbert flow.  Then $\psi^H_t$ has a unique measure of
maximal entropy, the Bowen-Margulis measure $\widetilde m_H$, with full support, such that 
\begin{equation}\label{eq:equilimit}
  \frac{1}{\lvert\mathcal S_0(L)\rvert}
  \sum_{\gamma\in\mathcal S_0(L)}
  \widetilde{\mathcal L}_\gamma
  \xrightarrow{\;\ast\;}
  \widetilde m_H
\end{equation}
as $L \to \infty$. 
\end{prop}

\begin{proof}
The existence and uniqueness of $\widetilde m_H$ are contained in
Proposition~\ref{prop:app-hilbert-flow}, and its full support is
Lemma~\ref{lem:app-full-support}.  For \eqref{eq:equilimit}, express the left-hand side of \eqref{eq:equilimit} as
\[
  \frac{
    h_H L e^{-h_HL}
    \displaystyle\sum_{\gamma\in\mathcal S_0(L)}
    \widetilde{\mathcal L}_\gamma
  }{
    h_H L e^{-h_HL}\lvert\mathcal S_0(L)\rvert
  }
\]
and note that the numerator converges weakly to $\widetilde m_H$ by
Lemma~\ref{lem:app-weighted-equidistribution}, while the denominator
converges to $1$ by \eqref{eq:app-prime-orbit}. 
\end{proof}

\noindent\textit{Remark.}
The entire discussion in this Appendix, including Sambarino's cocycle construction in \cite{Sam14}, also hold for  periods determined by a positive \textit{linear} functional  on the Jordan projection $\lambda(\rho_0(\gamma))= (\log\lvert \lambda_1\rvert , \log\lvert \lambda_2 \rvert, \ldots , \log \lvert \lambda_n\rvert)$. In particular, the equidistribution result holds when the closed geodesics are ordered by any fundamental weight length (\textit{cf.} Definition \ref{defn:lengths}). 
However, an analogous statement for the symmetric-space translation length
would require equidistribution of closed geodesics ordered by the translation length $\ell_{\rho_0}(\gamma)=\|\lambda(\rho_0(\gamma))\|$
where $\lVert \cdot \rVert$ is the \textit{non-linear} Euclidean norm.  
One possible approach to proving the statement could be to apply the recent work of Chow-Fromm where they prove an equidistribution theorem for closed curves with ``norm-like" ordering (see \cite[Theorem 7.3]{ChowFromm}).

 \bibliographystyle{alpha}
 \bibliography{references}

\end{document}